\documentclass[leqno,10.5pt]{article} 

\usepackage[latin1,utf8]{inputenc}  
\usepackage{lmodern}
\usepackage{verbatim} 
\usepackage{amsmath, amssymb}
\usepackage{amsthm} 
\usepackage{amscd}
\usepackage{color}
\usepackage{mathrsfs}
\usepackage{float}
\usepackage{ulem} 
\usepackage{amsmath,amssymb,amsthm}
\usepackage{amscd}
\usepackage{hyperref} 
\usepackage{enumitem}
\setenumerate[0]{label=(\roman*)} 
\usepackage{soul} 
\usepackage[mathscr]{euscript} 

\theoremstyle{plain} 
\newtheorem{theorem}{\sc Theorem}[section]
\newtheorem{lemma}[theorem]{\sc Lemma}
\newtheorem{corollary}[theorem]{\sc Corollary}
\newtheorem{proposition}[theorem]{\sc Proposition}

\theoremstyle{definition} 
\newtheorem{definition}[theorem]{\sc Definition}
\newtheorem{remark}[theorem]{\sc Remark}
\newtheorem{example}[theorem]{\sc Example}

\newtheorem{conj}[theorem]{\sc Conjecture}

\newcommand{\C}{\mathbb{C}}
\newcommand{\R}{\mathbb{R}}
\newcommand{\Q}{\mathbb{Q}}
\newcommand{\Z}{\mathbb{Z}}
\newcommand{\N}{\mathbb{N}}
\newcommand{\Shift}{\tau}
\newcommand{\Mshift}{\boldsymbol{\tau}}
\newcommand{\Indset}{\mathcal{S}}

\newcommand{\AND}{\quad \textrm{and} \quad}

\newcommand{\un}{\mathbf{n}}
\newcommand{\ur}{\underline{r}}
\newcommand{\bDelta}{\boldsymbol{\Delta}}
\newcommand{\Eval}{\mathrm{Eval}}

\renewcommand\epsilon\varespilon 
\renewcommand\phi\varphi 
\renewcommand \subset \subseteq

\newcommand{\Span}[2][]{\left\langle #2\right\rangle_{#1}} 

\newcommand\la{\langle}
\newcommand\ra{\rangle}
\newcommand{\balpha}{\boldsymbol{\alpha}}
\newcommand{\bm}{\boldsymbol{m}}
\newcommand{\bp}{\boldsymbol{p}}
\newcommand{\br}{\mathbf{r}}
\newcommand{\cA}{\mathcal{A}}

\newcommand{\cM}{\mathcal{M}}
\newcommand{\fL}{\mathfrak{L}}
\newcommand{\fT}{\mathfrak{T}}
\newcommand{\fM}{\mathfrak{M}}
\newcommand{\fR}{\mathfrak{R}}

\newcommand{\sS}{\mathscr{S}} 

\newcommand{\tc}{\widetilde{c}}
\newcommand{\tP}{\widetilde{P}}
\newcommand{\tQ}{\widetilde{Q}}

\newcommand{\Matp}{\cM_{n}} 
\newcommand{\Matr}{\cM_{n,\fR}} 

\newcommand{\Det}{\mathfrak{D}} 
\newcommand{\den}{\mathrm{den}}
\newcommand{\End}{\mathrm{End}}
\newcommand{\Hom}{\mathrm{Hom}}
\newcommand{\lcm}{\mathrm{lcm}}
\newcommand{\ord}{\mathrm{ord}_\infty}

\newcommand{\MellinInv}{\fM^{\mathrm{Inv}}}
\newcommand{\hMellinInv}{\widehat{\fM}^{\mathrm{Inv}}}
\newcommand{\hp}{\widehat{p}}
\newcommand{\hq}{\widehat{q}}

\newcommand{\hA}{\widehat{A}}
\newcommand{\hP}{\widehat{P}}
\newcommand{\hQ}{\widehat{Q}}

\title{On the linear independence of $p$-adic polygamma values}

\author{\textsc{Makoto Kawashima} and \textsc{Anthony Po\"{e}ls}}
\date{ }

\begin{document}

\maketitle

\begin{abstract}
In this article, we present a new linear independence criterion for values of the $p$-adic polygamma functions defined by J.~Diamond. As an application, we obtain the linear independence of some families of values of the $p$-adic Hurwitz zeta function $\zeta_p(s,x)$ at distinct shifts $x$. This improves and extends a previous result due to P.~Bel \cite{B1}, as well as irrationality results established by F.~Beukers \cite{B}. Our proof is based on a novel and explicit construction of Pad\'{e}-type approximants of the second kind of Diamond's $p$-adic polygamma functions. This construction is established by using a difference analogue of the Rodrigues formula for orthogonal polynomials.
\end{abstract}


\medskip

{\textit{MSC 2020}: 11J72 (primary); 11J61(secondary). 

\section{Introduction}

A classical problem in Diophantine approximation is the study of the irrationality or linear independence of values of $L$--functions at positive integers. A famous example is the case of the Riemann zeta function. In their seminal 2001 papers \cite{B-R, R1}, K.~Ball and T.~Rivoal established that, given an odd integer $a\geq 3$, the dimension $\delta(a)$ of the $\Q$--vector space spanned by $1,\zeta(3),\zeta(5),\dots,\zeta(a)$ satisfies
\begin{align*}
    \delta(a) \geq \frac{1}{3}\log a,
\end{align*}
where $\zeta:\C\setminus\{1\}\rightarrow \C$ is the Riemann zeta function. Their work has inspired many others, see for example S.~Fischler, Sprang and W.~Zudilin \cite{FSZ}, L.~Lai and P.~Yu \cite{LaiYuZeta}, and Fischler~\cite{F} for irrationality and linear independence results of odd values of the Riemann zeta function. In this article, we are interested in proving similar properties for $p$--adic $L$--functions, where $p$ denotes a fixed prime number. We still have few answers to this kind of questions. In~2005, F.~Calegari \cite{Cal} used the theory of $p$--adic modular forms to establish the irrationality of $\zeta_p(3)$, for $p=2,3$, and $L_2(2, \chi)$, where $\chi$ is the Dirichlet character of conductor~$4$. Subsequently, F.~Beukers \cite{B} provided an alternative proof of these results. Using classical continued fractions discovered by T.~J.~Stieltjes, he also proved the irrationality for some values of the $p$-adic Hurwitz zeta function $\zeta_p(s, x)$ at $s=2,3$ (see Definition~\ref{def: p-adic Hurwitz zeta function} of Section~\ref{section: notation} for the precise definition of $\zeta_p(s,x)$).  In~\cite{B1}, P.~Bel adapted the approach of \cite{B-R} to obtain similar properties for certain $p$--adic functions. This was later generalized by M.~Hirose, N.~Sato and the first author \cite{HKS}. The following linear independence criterion is a consequence of the proof of \cite[Th\'{e}or\`{e}me~3.2]{B1}.
\begin{theorem}[Bel, 2010]
    \label{Bel1}
    Let $m \geq 1$ be an integer and $p$ be a prime number such that
    \begin{align} \label{B ineq p}
        \log\,p\ge (1+\log 2)(m+1)^2.
    \end{align}
    Then, the $m+1$ elements of $\Q_p$:
    \begin{align*}
        1,\zeta_p(2,p^{-1}),\ldots,\zeta_p(m+1,p^{-1})
    \end{align*}
    are linearly independent over $\Q$.
\end{theorem}

Theorem $\ref{Bel1}$ is established by constructing certain Pad\'{e} approximations of the first kind discovered by T.~Rivoal ({\it confer} \cite{R2}), which are derived from the asymptotic expansion of the Hurwitz zeta function. In~\cite{B2} Bel also proved the irrationality of $\zeta_p(4,x)$ for $x=1/p$ with $p\ge 19$. Lower bounds for the dimension of the vector space spanned by $p$--adic Hurwitz zeta values and Kubota-Leopoldt $p$--adic $L$--values can be found in the recent papers of J.~Sprang \cite{SprangL}, also see L.~Lai \cite{Li}, and Lai and Sprang \cite{Li-Sprang}. Their results were obtained by constructing approximations of $p$--adic $L$--values, in a similar way to Ball-Rivoal for the (complexe) Riemann zeta function~\cite{B-R}. Before stating our main result, which requires some technical definitions, let us present two of its consequences. The first one improves and generalizes Theorem~\ref{Bel1}. Define the function $g:\Z_{\geq 1}\rightarrow \R^+$ by $g(M) = 0$ for $M=1,2$, and
\begin{align}
    \label{eq: intro def g(M)}
    g(M) = (M+1)\log \bigg(\frac{2(M+1)^{M+1}}{M^M}\bigg) \qquad \textrm{if $M> 2$.}
\end{align}

\begin{theorem} \label{main cor}
    Let $p$ be a prime number and $m,r$ be positive integers. Assume that
    \begin{align} \label{ineq p}
        \Big(r+\dfrac{1}{p-1}\Big)\log\,p>g(m)+m+m\left(1+\dfrac{1}{2}+\cdots+\dfrac{1}{m}\right).
    \end{align}
    Then the $m+1$ elements of $\Q_p$:
    \begin{align*}
        1, \zeta_p(2,p^{-r}), \ldots, \zeta_p(m+1, p^{-r})
    \end{align*}
    are linearly independent over $\Q$.
\end{theorem}

The special case  $d=m=1$ was proved by Beukers in \cite[Theorem $9.2$]{B}. Note that $(m+1)^m/m^m$ tends to $e$ as $m$ tends to infinity, so that
\begin{align*}
    g(m) = (m+1)\log \Big(2(m+1)\frac{(m+1)^{m}}{m^m}\Big) \sim_{m\to \infty} m \log m.
\end{align*}
Consequently, for $r=1$, we improve Theorem~\ref{Bel1}  by replacing the condition $\log p \gg m^2$ with the condition $\log p \gg m\log m$. The table below shows how our condition compares to that of Bel (B.) for $m \leq 8$. For $m \geq 3$, we just wrote a crude order of magnitude given by the conditions.

\begin{figure}[H]
   \centering
    \begin{align*}
        \begin{array}{|c|c|c|c|c|c|c|c|c|}
        \hline
        m & 1 & 2 & 3 & 4 & 5 & 6 & 7 & 8\\
        \hline
        p\ge (\textrm{new}) & 5 & 144 & 7\cdot 10^6 & 10^{9} & 3\cdot 10^{11} & 7\cdot 10^{13} & 2\cdot 10^{16} & 8\cdot 10^{18} \\
        \hline
        p\ge (\textrm{B.}) & 874 & 4148779 & 6\cdot 10^{11} & 2 \cdot 10^{18} & 3\cdot 10^{26} & 10^{36} & 10^{47} & 4\cdot 10^{59} \\
        \hline
        \end{array}
    \end{align*}
  \caption{Comparison between our condition and that of P.~Bel}\label{table: our bound vs Bell}
\end{figure}

Our main theorem also allows us to obtain  the linear independence of values of the $p$--adic Hurwitz zeta function at distinct shifts. 
For example, we can deduce the following result, see the end of Section~\ref{section: proof main thm} for more details.

\begin{theorem} \label{main cor 2}
    Let $p \geq 3$ be a prime number and $a,b,m,\delta$ be positive integers with $\delta = a-3(m+1)b > 0$. Assume that
    \begin{align} \label{ineq p cor 2}
        \left(\delta - \frac{3m+2}{p-1} \right)\log p > g({2m+1})+ (2m+1)\left(1+\dfrac{1}{2}+\cdots+\dfrac{1}{m}\right).
    \end{align}
    Then the $2m+1$ elements of $\Q_p$:
    \begin{align*}
        1, \zeta_p(2,p^{-a}), \ldots, \zeta_p(m+1,p^{-a}), \zeta_p\big(2,p^{-a}+p^{-b}\big), \ldots, \zeta_p\big(m+1,p^{-a}+p^{-b}\big)
    \end{align*}
    are linearly independent over $\Q$.
\end{theorem}

The hypothesis $\delta>0$ ensures that \eqref{ineq p cor 2} is satisfied for large enough $p$ (with $m,a,b$ fixed).

\medskip

\noindent{\textbf{Main result.}} Values of Kubota-Leopoldt $p$--adic $L$--functions at positive integers can be expressed as linear combinations of values of $p$--adic polygamma functions at rational points, as shown by J.~Diamond ({\it confer} \cite[Theorem $3$]{Dia2}). This motivates us to investigate in this paper the linear independence of $p$--adic polygamma values. In order to state our main result, namely Theorem~\ref{main 1} below, we need the following notation.

\medskip
Let $\overline{\Q}_p$ be an algebraic closure of the field of $p$--adic numbers $\Q_p$. We write  $|\cdot |_p:\overline{\Q}_p\longrightarrow \R_{\ge0}$ for the $p$--adic norm on $\overline{\Q}_p$ with the normalization $|p|_p=p^{-1}$. We denote by $G_p:\Q_p\setminus \Z_p \rightarrow \Q_p$ the log gamma function of Diamond, and by $\omega:\Q_p^\times \rightarrow \Q_p^\times$ the Teichm\"{u}ller character, whose precise definitions are recalled in Section~\ref{section: notation}. The function $G_p$ is a $p$--adic analog of the classical log\,$\Gamma$-function and satisfies $G_p(x+1)=G_p(x) + \log_p(x)$, where $\log_p$ stands for the Iwasawa $p$--adic logarithm. For any integer $s\geq 0$, the $(s+1)$--th derivative of $G_p^{(s+1)}$ of $G_p$ is called the \textsl{$s$-th Diamond's $p$-adic polygamma function}. We set $q_p = p$ if $p\geq 3$ and $q_p = 4$ if $p=2$. The $p$-adic Hurwitz zeta function $\zeta_p(s,x)$ satisfies the classic identity
\begin{align}
    \label{eq intro relation G_p and zeta_p}
    G_p^{(s)}(x) = (-1)^s(s-1)!\omega(x)^{1-s}\zeta_p(s,x)
\end{align}
for each $x\in\Q_p$ with $|x|_p \geq q_p$ and each integer $s\geq 2$ (see Eq.~\eqref{eq: link zeta_p and derivative of G_p}). Given positive integers $d,m$ and $\balpha=(\alpha_1,\ldots,\alpha_d)\in \Q^d$, define
\begin{align}
      \den(\balpha)& =\min\{ n\in \Z_{\geq 1} \,;\,  n\alpha_i\in \Z \ \text{for all $i=1,\dots,d$} \}, \notag \\
      \mu(\balpha) & = \displaystyle \den(\balpha)\prod_{\substack{q:{\rm{prime}} \\ q|{\den}(\balpha)}}q^{\tfrac{1}{q-1}}, \label{eq def: den(alpha) and mu(alpha)}\\
      M & =d(m+1)-1, \notag \\
      f(\balpha,d,m) & = g(M) + M\bigg(1+\sum_{j=1}^m \dfrac{1}{j}\bigg)+\log \Big(\mu(\balpha)^M\prod_{i=2}^d \mu(\alpha_i)^{m+1}\Big)-\log|\mu(\balpha)|_p, \label{eq def: F_p(alpha,m)}
\end{align}
where the function $g$ is given by \eqref{eq: intro def g(M)} as previously.

\begin{theorem} \label{main 1}
    Let $d,m$ be a positive integers and $\balpha=(\alpha_1,\ldots,\alpha_d)\in \Q^d$ with $\alpha_1=0$ and
    \begin{align}
        \label{eq: main thm: condition alpha_i}
        \alpha_i-\alpha_j\notin \Z \quad \textrm{for any $i\neq j$}.
    \end{align}
    Let $x\in \Q$ satisfying $|x|_p \geq q_p\max\{1,|\alpha_2|_p,\dots,|\alpha_d|_p\}$ and
    \begin{align} \label{ineq x}
       \dfrac{\log p}{p-1}+\log |x|_p > M \log \big(\mu(x)|\mu(x)|_p\big) + f(\balpha,d,m).
    \end{align}
    Then for $i=1,\dots,d$, we have $|x+\alpha_i|_p > 1$, and the $dm+1$ elements of $\Q_p$:
    \begin{align*}
        1, G_p^{(2)}(x+\alpha_1), \ldots, G_p^{(m+1)}(x+\alpha_1), \ldots, G_p^{(2)}(x+\alpha_d), \ldots, G_p^{(m+1)}(x+\alpha_d)
    \end{align*}
    are linearly independent over $\Q$.
\end{theorem}

\begin{remark}
    The special case $d=m=1$ was proved by Beukers, see \cite[Theorem $9.2$]{B}. If $x$ is a power of $p$, we have $\mu(x)|\mu(x)|_p=1$, and \eqref{ineq x} becomes $\dfrac{\log p}{p-1}+\log |x|_p > f(\balpha,d,m)$. Condition~\eqref{ineq x} is not optimal and could be relaxed by replacing $g$ defined in \eqref{eq: intro def g(M)} with a smaller function of~$M$, see Remark~\ref{remark function g(M) relaxed}. Note that if $d=1$, then $M=m$, $\mu(\balpha)=1$, and we simply have
    \[
        f(\balpha,1,m)  = g(m) + m+\sum_{j=1}^m \dfrac{m}{j}.
    \]
\end{remark}

Theorem~\ref{main 1} combined with \eqref{eq intro relation G_p and zeta_p} yields the following consequence.

\begin{theorem}  \label{cor 1}
    Let $d$, $m$, $\balpha = (\alpha_1,\dots,\alpha_d)$, $p$ and $x$ satisfying the hypotheses of Theorem $\ref{main 1}$. Then the $dm$ elements of $\Q_p$
    \begin{align*}
        \omega(x+\alpha_i)^{1-s}\zeta_p(s,x+\alpha_i) \quad (1\leq i \leq d \AND 2\leq s \leq m+1)
    \end{align*}
    together with $1$ are linearly independent over $\Q$, where $\omega$ denotes the Teichm\"{u}ller character on $\Q_p^\times$ .
\end{theorem}

To the best of our knowledge, this is the first $p$--adic linearly independence criterion involving distinct shifts $x+\alpha_i$. Note that according to \cite[Proposition~11.2.9]{C}, given an integer $s\geq 2$ and $x\in\Q_p$ with $|x|_p\geq q_p$, we have
\begin{align*}
    \zeta_p(s, x + 1) - \zeta_p(s,x) = -\omega(x)^{s-1}x^{-s}.
\end{align*}
In particular, if $x\in\Q$ is such that $\omega(x)^{s-1}\in\Q$ (which is always the case if $p-1$ divides $s-1$ for example), then $1, \zeta_p(s,x), \zeta_p(s,x+1)$ are linearly dependent over $\Q$. Condition~\eqref{eq: main thm: condition alpha_i} appearing in Theorems~\ref{main 1} and~\ref{cor 1} is therefore necessary and quite natural.

\medskip

We can deduce Theorem~\ref{main cor} (resp. Theorem~\ref{main cor 2}) from Theorem \ref{cor 1} by choosing the parameters $d=1$ and $x=p^{-r}$ (resp. $d=2$, $x=p^{-a}$ and $\alpha_2 = p^{-b}$). Note that $\omega(p^{-r}) = p^{-r}\in\Q$ (resp.  $\omega(p^{-a}) = \omega(p^{-a}+p^{-b}) = p^{-a} \in \Q$ if $a>b$ and $p\geq 3$) by \eqref{eq: value of Teichmuller}.

\medskip

\noindent\textbf{Our strategy.} The proof of Theorem~\ref{main 1} is based on Pad\'{e} approximants of second kind. This a similar approach to Beukers in \cite[Theorem 9.2]{B}, although we will use different tools in a more general context. Our constructions heavily rely on formal integration transforms $\phi_f$ (see Section~\ref{section: Pade approximation}). This method was employed, though expressed differently, in \cite{DHK1,DHK2,DHK3,KP,Kaw}. While holonomic series were considered in the aforementioned studies, in this paper we examine their ``difference analogs'', in other words, formal Laurent series which satisfy a difference equation. For each integer $s\geq 2$, define the formal Laurent series $R_s(z)$ by
\begin{align*}
    R_s(z) = \sum_{k=s-2}^{\infty}(k-s+3)_{s-2}B_{k-s+2}\cdot\frac{(-1)^{k+1}}{z^{k+1}},
\end{align*}
where $B_k$ denotes the $k$--th Bernoulli number and $(a)_k = a(a+1)\cdots(a+k-1)$ is the Pochhammer symbol. Then, for each $x\in \C_p\setminus \Z_p$, the $s$--th polygamma function $G_p^{(s)}$ evaluated at $x$ is equal to $-R_s(x)$. The Laurent series $R_s(z)$ is the image of a $G$-function in the sense of Siegel \cite{S} under a modified Mellin transform $\MellinInv$, whose definition is inspired by~\cite[D\'{e}finition~1]{B-D} and~\cite[Section $7$]{An2}. This will allow us to show that $R_s(z)$ is a solution to a certain difference equation, see Proposition~\ref{prop: difference equation R_alpha,s}. We will then construct Pad\'{e} approximants of second kind $(P_n(z),Q_{n,i,s}(z))_{(i,s)}$ of the series $(R_s(z+\alpha_i))_{(i,s)}$. This uses a difference analogue of the Rodrigues formula for Pad\'{e} approximants, as outlined in \cite[Section 2]{Kaw} by the first author.

\medskip

To prove the main theorem using Siegel's method \cite{S}, it is necessary to demonstrate the non-vanishing of the determinant formed by those Pad\'{e} approximants, which involves Bernoulli numbers. In previous works such as \cite{DHK1,DHK2,DHK3,KP,Kaw}, this step is done by computing a ``closed form'' of the involved determinants, which can be quite a difficult and challenging problem in general, see \cite{Kratt} for example. Here, we develop new tools which allows us to prove the non-vanishing property rather ``simply'', without having to obtain such explicit formula. This approach is expected to apply in different settings. The rest of the proof is classical, although quite technical. Given a rational number $x$, we estimate the growth of the sequences $(|P_n(x)|)_n$, $(|Q_{n,i,s}(x)|)_n$, $(|P_n(x)R_s(x+\alpha_i)-Q_{n,i,s}(x)|_p)_n$, as well as that of $(|D_n(x)|)_n$ and $(|D_n(x)|_p)_n$, where $D_n(x)$ denotes the common denominator of $P_n(x)$ and $Q_{n,i,s}(x)$. Suitable growth conditions ensure that the $p$--adic numbers $R_s(x+\alpha_i)$ together with $1$ are linearly independent. For the estimates of $(|P_n(x)|)_n$, $(|Q_{n,i,s}(x)|)_n$, which is directly related to the term $g(M)$ appearing in \eqref{eq def: F_p(alpha,m)} we explain in Section~\ref{section: poincare-perron rec} how we can use Perron's second Theorem to improve the rough estimates that we get. To our knowledge, the majority of the Diophantine results based on Poincar\'{e}-Perron theorem involve recurrences of order $2$. In our case,  the order of the Poincar\'{e}-Perron recurrences involved is equal to $M+1$, where the parameter $M$ is as in Theorem~\ref{main 1}. In particular, it is strictly larger than $2$ when $M>1$.

\medskip

\noindent\textbf{Outline of our article.} In Section~\ref{section: notation} we introduce our notation and recall the definitions and some properties of Diamond's $p$--adic polygamma functions $G_p^{(s)}(z)$ and the $p$-adic Hurwitz zeta functions $\zeta_p(s,x)$. Several formal transforms will play an important role in our constructions, such as a modified inverse Mellin transform and formal integration transforms introduced in Section~\ref{Section: modified formal transforms} and~\ref{section: Pade approximation} respectively. In Section~\ref{section: Mellin transform and polygamma} we define some formal series $R_{\alpha,s}(z)$ which are related to the $p$--adic polygamma functions. Using the modified inverse Mellin transform, the formal integration transforms, and some basic properties of the difference operator established in Section~\ref{section: prop of diff operator}, we construct some Pad\'{e} approximants of the functions $R_{\alpha,s}(z)$ is Section~\ref{subsection: Pade approximants for polygamma}. To prove our main theorem, we need to study these Pad\'{e} approximants in more depth. First, as explained previously, we need to prove that they are linearly independent, which amounts to showing the non-vanishing of some determinant. This is a consequence of the main result of Section~\ref{section: non-singular matrices}. Secondly, we need to establish several estimates, such as the growth of our Pad\'{e} approximants and their denominators, as well as the $p$--adic growth of the Pad\'{e} approximations. This is done in Section~\ref{section: estimates} and~\ref{section: poincare-perron rec}. Finally, Section~\ref{section: proof main thm} is devoted to the proof of a slightly more general version of Theorem~\ref{main 1}.

\section{Notation}
\label{section: notation}

In this section, we introduce some notation and we give the definition of the log gamma function of Diamond $G_p(z)$, Diamond's $p$-adic polygamma functions and the $p$--adic Hurwitz zeta function $\zeta_p(s,x)$, as well as some basic properties they satisfy and that we will use later. In subsection~\ref{subsection: Pade theory}, we recall some elements of Pad\'{e} approximation theory.

\subsection{The $p$--adic Hurwitz zeta function}

The floor (resp. ceiling) function is denoted by $\lfloor \cdot \rfloor$ (resp. $\lceil \cdot \rceil$). For any $a\in \Z$, we denote by $\Z_{\leq a}$ the set of integers $n$ with $n \leq a$. We define similarly $\Z_{<a}$, $\Z_{\geq a}$ and $\Z_{>a}$. The rising factorials are the polynomials $(x)_n = x(x+1)\cdots(x+n-1)$ ($n\in\Z_{\geq 0}$), with the convention that $(x)_0 = 1$. Given a (unitary) ring $R$ and an integer $n\geq 1$, we denote by $R^\times$ the unit group of $R$. 

\medskip

In the following, we fix a prime number $p$ and we set
\begin{align}
    \label{eq def q_p}
    q_p = \left\{ \begin{array}{ll} p & \textrm{if $p\geq 3$} \\ 4 & \textrm{if $p=2$.} \end{array}\right.
\end{align}
As usual, $\Q_p$ is the field of $p$--adic numbers, and $\C_p$ is the $p$--adic completion of an algebraic closure of $\Q_p$. We write  $|\cdot |_p:\C_p\longrightarrow \R_{\ge0}$ for the $p$--adic norm with the normalization $|p|_p=p^{-1}$. We denote by $v_p:\C_p\rightarrow \{p^r \,;\, r\in\Q\}\cup \{\infty\}$ the valuation which extends the usual $p$--adic valuation on $\Z$. With this notation, we have $|x|_p = p^{-v_p(x)}$ for each $x\in\C_p$. The ring of $p$--adic integers is the subset $\Z_p=\{x\in \Q_p\,;\, |x|_p \leq 1 \}$, and the group of units of $\Z_p$ is $\Z_p^\times = \{x\in \Q_p\,;\, |x|_p = 1 \}$.

\medskip

Given a positive integer $d$ and $\balpha=(\alpha_1,\ldots,\alpha_d)\in \Q^d$, recall that the denominator $\den(\balpha)$ of $\balpha$ and the quantity $\mu(\balpha)$ are defined in~\eqref{eq def: den(alpha) and mu(alpha)}. Note that $\den(\balpha) = \lcm\big( \den(\alpha_1),\dots,\den(\alpha_d))$, where $\lcm$ stands for the least common multiple. Thus $\den(\alpha_i)\leq \deg(\balpha)$ and $\mu(\alpha_i) \leq \mu(\balpha)$ for $i=1,\dots,d$.

\medskip

\noindent\textbf{Bernoulli numbers and polynomials.} We define the Bernoulli polynomials $B_n(x)$ by their exponential generating function
\begin{align}
    \label{eq: generating function B_k(x)}
    \dfrac{ze^{xz}}{e^z-1}=\sum_{k=0}^{\infty}B_k(x)\dfrac{z^k}{k!},
\end{align}
and the Bernoulli numbers $B_n$ by $B_n = B_n(0)$. Recall that $B_n'(x) = nB_{n-1}(x)$ for each $n\geq 1$, and that $B_n(x)$ is a monoic polynomial of degree $n$. We also have the classic formulas $B_n(x + 1) = B_n(x) + nx^{n-1}$, as well as
\begin{align*}
    \sum_{k=0}^{n} \binom{n}{k}y^{n-k}B_k(x) = B_n(x+y) \AND \sum_{k=0}^{n-1}\binom{n}{k}B_k(x) = nx^{n-1}.
\end{align*}
For any positive integer $k$, Staudt-Clausen Theorem (see \cite{Cla}) ensures that
\begin{align}
    B_{2k} + \sum_{p-1\mid 2k} \frac{1}{p} \in \Z
\end{align}
where it is understood that $p$ is a prime number. In particular, for any prime number $p$ and any integer $n\geq 0$, we have $|B_n|_p \leq p$. It follows that for any $\alpha\in \C_p$, we have
\begin{align}
    \label{eq: valuation p-adic Bernoulli}
    |B_p(\alpha)|_p = \left|\sum_{k=0}^{n} \binom{n}{k}B_k\alpha^{n-k}\right|_p \leq p\cdot\max\{1,|\alpha|_p\}^n.
\end{align}

\medskip

\noindent\textbf{Volkenborn integral.} A detailed study of the Volkenborn integral would be quite long, so we refer to \cite{Rob} for the missing details. We say that a continuous function $f:\Z_p \rightarrow \C_p$ is \textsl{Volkenborn integrable} if the sequence
\begin{align*}
    p^{-n}\sum_{k=0}^{p^n-1}f(k)
\end{align*}
converges $p$--adically. In that case we call its limit the \textsl{Volkenborn integral} of $f$ and we write
\begin{align*}
    \int_{\Z_p}f(t)dt := \lim_{n\rightarrow \infty} p^{-n}\sum_{k=0}^{p^n-1}f(k)
\end{align*}
({\it confer} \cite{Vol}). For example, continuously differentiable functions and  locally analytic functions are Volkenborn integrable, see~\cite[\S55]{Shi}. For all $x\in\Q_p$ and $n\in\Z_{\geq 0}$, we have
\begin{align*}
    \int_{\Z_p} (x + t)^n dt = B_n(x)
\end{align*}
(see \cite[Section~11.1]{C}).

\medskip

\noindent\textbf{Log gamma and polygamma functions of Diamond.} We denote the  (Iwasawa) $p$--adic logarithm function by $\log_p : \C_p^\times \rightarrow \C_p$. The following properties characterize $\log_p$. We have $\log_p(xy) = \log_p(x)+\log_p(y)$ for each $x,y\in\C_p^\times$, $\log_p(p)=0$ and
\begin{align*}
    \log_p(1+x) = \sum_{n\geq 0}\frac{(-1)^{n+1}x^n}{n}
\end{align*}
for each $x\in \C_p$ with $|x|_p < 1$. In \cite{Dia1}, J.~Diamond introduced a $p$--adic analog $G_p$ of the classical log\,$\Gamma$-function as follows.
\begin{definition}
    The \textsl{log gamma function of Diamond} $G_p: \C_p\setminus\Z_p \rightarrow \C_p$ is the function given for each $x\in \C_p\setminus\Z_p$ by
    \begin{align*}
        G_p(x) = \int_{\Z_p}\Big((x+t)\log_p(x+t)-(x+t)\Big)dt.
    \end{align*}
    For each integer $s\geq 1$, we also define $R_s: \C_p\setminus\Z_p \rightarrow \C_p$ by setting
    \begin{align}
        \label{def R_s via G_p}
        R_s(x) = -G_p^{(s)}(x) \qquad(x\in \C_p\setminus\Z_p).
    \end{align}
\end{definition}
Recall that $G_p$ satisfies the functional equation $G_p(x+1)= G_p(x) + \log_p(x)$ for each $x\in \C_p\setminus\Z_p$, and is locally analytic on $\C_p\setminus\Z_p$. Furthermore, $G_p(z)$ has the following expansion ({\it confer} {\rm{\cite[Theorem~6]{Dia1}}}):
\begin{align} \label{dia}
    G_p(z)=\left(z-\dfrac{1}{2}\right) \log_p(z)-z+\sum_{k=1}^{\infty}\dfrac{B_{k+1}}{k(k+1)}\cdot \dfrac{1}{z^k} \qquad (z\in\C_p, |z|_p > 1).
\end{align}

The successive derivatives of $G_p$ are called the \textsl{Diamond's $p$-adic polygamma functions}. For our purpose, it is more convenient to work with $R_s$. By \eqref{dia}, the function $R_2$ has the following expansion:
\begin{align*}
    R_2(z)=\sum_{k=0}^{\infty}B_{k}\cdot\frac{(-1)^{k+1}}{z^{k+1}} \qquad (z\in\C_p, |z|_p > 1).
\end{align*}
More generally, for each integer $s\geq 2$ and each $z\in\C_p$ with $|z|_p > 1$, we have
\begin{align}
    \label{eq: series expansion of R_s}
    R_s(z) = R_2^{(s-2)}(z) =\sum_{k=s-2}^{\infty}(k-s+3)_{s-2}B_{k-s+2}\cdot\frac{(-1)^{k+1}}{z^{k+1}}.
\end{align}

\medskip

\noindent\textbf{Teichm\"{u}ller character.} Recall that $q_p$ is defined in~\eqref{eq def q_p}. We define the Teichm\"{u}ller character $\omega: \Q_p^\times \rightarrow \Q_p^\times$ as follows. For any $x\in\Z_p^\times = \{x\in\Q_p\,;\, |x|_p =1\}$, we denote by $\omega(x)$ the unique $\phi(q_p)$--th root of unity such that
\begin{align*}
    \langle x \rangle := \frac{x}{\omega(x)} \in 1+q_p\Z_p,
\end{align*}
where $\varphi$ is Euler's totient function. Then, given $y\in\Q_p^\times$, we put
\begin{align*}
    \omega(y) = p^{v_p(y)} \omega\Big( p^{-v_p(y)}y\Big) \AND \langle y \rangle  = \frac{y}{\omega(y)} =  \langle p^{-v_p(y)}y \rangle \in 1+q_p\Z_p.
\end{align*}

\begin{remark}
    There is a canonical isomorphism $\Z_p^\times \cong \mathcal{R}_p\times (1+q_p\Z_p)$, which is precisely given by $x \mapsto (\omega(x), \langle x \rangle)$, were $\mathcal{R}_p$ denotes the subgroup of $\phi(q_p)$--th unit roots of $\Z_p$. If $p\geq 3$, we have
    \begin{align*}
        \omega(x) = \lim_{n\rightarrow \infty} x^{p^n},
    \end{align*}
    and $\omega(x)$ is the unique $(p-1)$--th root of unity that is congruent to $x$ modulo $p$ (see \cite[Theorem~33.4]{Shi}). We deduce easily from the above that
    \begin{align}
        \label{eq: value of Teichmuller}
        \omega(p^{-a}+y) = p^{-a} \omega(1+yp^{a}) = p^{-a} \in \Q
    \end{align}
    for each integer $a\geq 0$ and each $y\in\Q_p$ with $|y|_p < p^a$. In the special case $p=2$, either $x$ or $-x$ is congruent to $1$ modulo $4 = q_p$, and we set $\omega(x) = \pm 1$, so that $x$ is congruent to $\omega(x)$ modulo $4$. Note that this last definition differs from \cite[Definition~33.3]{Shi} (which would give $\omega(x)=1$ for each $x\in\Z^\times_{2}$).
\end{remark}

\noindent\textbf{The $p$--adic Hurwitz zeta function.} Recall that $q_p$ is defined in~\eqref{eq def q_p}. We follow \cite[Definition~11.2.5]{C} for the definition of the $p$--adic Hurwitz zeta function $\zeta_p(s,x)$.

\begin{definition}
    \label{def: p-adic Hurwitz zeta function}
    For $x\in\Q_p$ and $s\in \C_p\setminus\{1\}$ with $|x|_p\geq q_p$ and  $|s|_p<q_pp^{-1/(p-1)}$, we define $\zeta_p(s,x)$ by the equivalent formulas
    \begin{align*}
        \zeta_p(s,x)=\dfrac{1}{s-1}\int_{\Z_p} \la x+t\ra^{1-s} dt = \frac{\la x\ra^{1-s}}{s-1} \sum_{k\geq 0} \binom{1-s}{k}B_kx^{-k}.
    \end{align*}
\end{definition}
For a fixed $x\in \Q_p$ with $|x|_p \geq q_p$, the function $s\mapsto \zeta_p(s,x)$ is the unique $p$--adic meromorphic function on $|s|_p<q_pp^{-1/(p-1)}$ satisfying
\begin{align*}
    \zeta_p(1-n,x)=-\omega(x)^{-n}\dfrac{B_n(x)}{n}
\end{align*}
for each integer $n\geq 1$. In addition, this function is analytic, except for a simple pole at $s=1$ with residue $1$ (see \cite[Proposition~11.2.8]{C}). Note that for $p\geq 3$ the condition ``$x\in\Q_p$ with $|x|_p \geq q_p$'' simply means that $x\in \Q_p\setminus\Z_p$.

\medskip

The following identity (see \cite[Proposition~11.5.6.]{C}) (which is equivalent to \eqref{eq intro relation G_p and zeta_p}) combined with Theorem~\ref{main 1} implies Theorem~\ref{cor 1}.
\begin{equation}
    \label{eq: link zeta_p and derivative of G_p}
    \omega(x)^{1-s}\zeta_p(s,x)=\dfrac{(-1)^{s}}{(s-1)!}G_p^{(s)}(x) = \dfrac{(-1)^{s+1}}{(s-1)!}R_s(x) \qquad (x \in \Q_p, |x|_p\geq q_p).
\end{equation}

\subsection{Pad\'{e} approximation theory}
\label{subsection: Pade theory}

Fix a field $K$ of characteristic $0$. For any subset $S$ of a $K$--vector space $V$, we denote by $\Span[K]{S} \subset V$ the $K$-vector space generated by the elements of $S$. Given an integer $n\geq 0$, we denote by $K[z]$ the ring of polynomials in $z$ with coefficients in $K$, and by $K[z]_{\leq n} \subset K[z]$ the subgroup of polynomials of degree at most $n$.

\medskip

Let us recall the definition of Pad\'{e}-type approximants of Laurent series and their basic properties. We denote by $K[[1/z]]$ the ring of formal power series ring of variable $1/z$ with coefficients in $K$, and by $K((1/z))$ its field of fractions. We say that an element of $K((1/z))$, which can be written as
\begin{align*}
     \sum_{k = -n}^{\infty} \dfrac{a_k}{z^{k}},
\end{align*}
with $n\in\Z$ and $a_k\in K$, is a \textsl{formal Laurent series} in $1/z$. We define the \textsl{order} function at $z=\infty$ by
\begin{align*}
    \ord :K((1/z)) \longrightarrow \Z\cup \{\infty\}; \ \sum_{k} \dfrac{a_k}{z^k} \mapsto \min\{k\in\Z\cup \{\infty\} \,;\, a_k\neq 0\}.
\end{align*}
with the convention $\min \emptyset = \infty$. In particular, for any $f\in K((1/z))$, we have $\ord \, f=\infty$ if and only if $f=0$.

\medskip

Given two $K$--vector spaces $V$ and $W$, we denote by $\Hom_K(V,W)$ the $K$--vector space of $K$--linear homomorphisms $V\rightarrow W$. When $V=W$, we write $\End_{K\text{-lin}}(V) = \Hom_K(V,W)$. Similarly, for any $K$--algebra $\cA$, we define $\End_{K\text{-alg}}(\cA)$ as the $K$-vector space of $K$--algebra endomorphisms of $\cA$.

\medskip

We recall, without proof, the following basic result from Pad\'{e} approximation theory.

\begin{lemma} \label{pade}
    Let $m$ be a positive integer, $f_1,\ldots,f_m\in (1/z)\cdot K[[1/z]]$ and $\boldsymbol{n}=(n_1,\ldots,n_m)\in \N^{m}$.
    Put $N=\sum_{j=1}^mn_j$. For any non-negative integer $M\ge N$, there exists a non-zero vector of polynomials $(P,Q_{1},\ldots,Q_m)\in K[z]^{m+1}$ satisfying the following conditions:
    \begin{enumerate}
      \item \label{item: lem pade: item 1} $\deg P\le M$,
      \item \label{item: lem pade: item 2} $\ord \left(P(z)f_j(z)-Q_j(z)\right)\ge n_j+1$ for $j=1,\dots,m$.
    \end{enumerate}
\end{lemma}

\begin{definition}
    With the notation of Lemma~\ref{pade}, fix a non-zero vector of polynomials $(P,Q_{1},\ldots,Q_m) \in K[z]^{m+1}$ satisfying the properties \ref{item: lem pade: item 1} and \ref{item: lem pade: item 2}.
    \begin{itemize}
      \item [$\bullet$] We say that $(P,Q_{1},\ldots,Q_m) \in K[z]^{m+1}$ is a \textsl{weight $\boldsymbol{n}$ and degree $M$ Pad\'{e}-type approximant of $(f_1,\ldots,f_m)$}.
      \item [$\bullet$] We call the remainders, namely the formal Laurent series  $(P(z)f_j(z)-Q_{j}(z))_{1\le j \le m}$, \textsl{weight $\boldsymbol{n}$ degree $M$ Pad\'{e}-type approximations of $(f_1,\ldots,f_m)$}.
    \end{itemize}
\end{definition}

\section{Modified formal transforms}
\label{Section: modified formal transforms}

Fix a field $K$ of characteristic $0$. We introduce a modified inverse Mellin transform $\MellinInv_K$ for formal power series, and we study some of its properties. This transform will play a key-role in studying the properties of the explicit Pad\'{e} approximants constructed in Section~\ref{subsection: Pade approximants for polygamma}. In the next section, we will compute the inverse Mellin transform of some formal series connected to polygamma functions.

\medskip

\noindent\textbf{Formal series (examples).} Fix $\alpha\in K$. We define the following formal series of $K[[z]]$ as usual
\begin{align*}
    \exp(z) &= \sum_{n\geq 0} \frac{z^n}{n!}, \quad \log(1+z) = \sum_{n\geq 1} \frac{(-1)^{n-1}}{n}z^n \AND
    (1+z)^\alpha = \sum_{k=0}^{\infty}\binom{\alpha}{k}z^k,
\end{align*}
where  $\displaystyle\binom{\alpha}{k}=\alpha(\alpha-1)\cdots(\alpha-k+1)/k!$ (with the convention that $\displaystyle \binom{\alpha}{0}=1$). We also see $1/(z+\alpha)$ as an element of $(1/z)\cdot K[[1/z]]$ by writing
\begin{align*}
    \frac{1}{z+\alpha} = \frac{1}{z}\sum_{n\geq 0} \left(-\frac{\alpha}{z}\right)^n.
\end{align*}

\medskip

\noindent\textbf{Difference and differential operators of $K[[1/z]]$.} Given $\alpha\in K$ we denote by  $\Shift_{\alpha}$ the $\alpha$--shift operator, and by $\Delta_\alpha = \Shift_\alpha-1$ the $\alpha$--difference operator of $K((1/z))$. Note that they stabilize $K[[1/z]]$ and $(1/z)\cdot K[[1/z]]$. They are defined for each $f(z)\in K((1/z))$ by
\begin{align*}
    \Shift_\alpha \big(f(z)\big) = f(z+\alpha) \AND \Delta_\alpha \big(f(z)\big) = f(z+\alpha)-f(z).
\end{align*}
Fix $f(z)\in K((1/z))$ and a sequence $(a_n)_{n\geq 0}$ of elements in $K$. Then, the series
\begin{align*}
    \sum_{n\geq 0}a_n \Delta_\alpha^n \big(f(z)\big) \AND  \sum_{n\geq 0}a_n \frac{d^n}{dz^n} f(z)
\end{align*}
converge in $K((1/z))$, since for each integer $n\geq 0$, the formal series $\Delta_\alpha^n (f(z))$ and $\frac{d^n}{dz^n} f(z)$ belong to $(1/z)^{n+d}K[[1/z]]$, where $d=\ord(f)\in \Z$. The sets $K[[\Delta_\alpha]]$ and $K[[d/dz]]$ are therefore subsets of the set of $\End_{K\text{-lin}}\big(K((1/z))\big)$. The above argument shows that they are also subsets of $\End_{K\text{-lin}}\big((1/z)\cdot K[[1/z]]\big)$ and $\End_{K\text{-lin}}(K[[1/z]])$. The following result will be useful.

\begin{lemma}
    \label{lemma: relation operator Delta and d/dz}
    For any $\alpha\in K$, we have
    \begin{align}
        \label{eq: action exp deritive on formal series}
        \exp\left(\alpha \frac{d}{dz}\right) = \Shift_\alpha \AND \exp\left(\frac{d}{dz}\right)-1 = \Delta_1.
    \end{align}
    So $K[[d/dz]] = K[[\Delta_1]]$,
    \begin{align}
        \label{eq: action Delta on formal series}
        \log(1+\Delta_1) = \frac{d}{dz} \AND (1+\Delta_1)^\alpha = \Shift_\alpha.
    \end{align}
\end{lemma}

\begin{proof}
    Since $\Delta_1 = \Shift_1-1$, we only have to prove the first equality in~\eqref{eq: action exp deritive on formal series}. Write
    \begin{align*}
        D := \exp\left(\alpha \frac{d}{dz}\right) = \sum_{k\geq 0} \frac{\alpha^k}{k!}\frac{d^k}{dz^k}.
    \end{align*}
    Let $n\geq 0$ be an integer. A quick computation yields
    \begin{align*}
         D\left(\frac{1}{z^n} \right) = \sum_{k\geq 0} \frac{(-n)(-n-1)\cdots (-n-k+1)}{k!}\frac{\alpha^k}{z^{n+k}}
         = \frac{1}{z^n}\frac{1}{(1+\alpha/z)^n} = \frac{1}{(z+\alpha)^n} = \Shift_\alpha \left(\frac{1}{z^n} \right),
    \end{align*}
    hence \eqref{eq: action exp deritive on formal series}.
\end{proof}

\noindent\textbf{Formal Laplace transform and Stirling numbers.} Following \cite[p.~154]{Bez-R}, we define the (formal) modified Laplace transform
\begin{align*}
    \fL_K :K[[z]]\longrightarrow K[[z]]; \ \ \  \sum_{k=0}^{\infty} a_k\frac{z^k}{k!}\mapsto \sum_{k=0}^{\infty} a_k z^k.
\end{align*}
Then $\fL_K$ is a homeomorphism of $K[[z]]$ (with respect to the $(z)$-adic topology). \\
Given a pair of non-negative integers $(k,n)$, we define $\sS(n,k)$ as the number of ways of partitioning a set of $n$ elements into $k$ non-empty sets. The numbers $\sS(n,k)$ (sometimes denoted by $\Big\{\!\!\begin{array}{c} n \\ k \end{array}\!\!\Big\}$, see \cite{Knu}) are called \textsl{Stirling number of the second kind} \cite{Stirling}. They satisfy the recurrence relation
\begin{align}
    \label{eq: rec relation stirling}
    \sS(n,k) = \sS(n-1,k-1) +  k\sS(n-1,k) \qquad (k,n\geq 1)
\end{align}
with initial conditions $\sS(0,0)=1$ and $\sS(n,0) = \sS(0,n) = 0$ for each $n\geq 1$. They also have the following generating functions
\begin{align}
    \label{eq: generating functions Stirling}
    \dfrac{(e^z-1)^n}{n!}=\sum_{k=n}^{\infty}\sS(k,n)\dfrac{z^k}{k!} \AND \dfrac{z^n}{(1-z)\cdots(1-nz)}=\sum_{k=n}^{\infty}\sS(k,n)z^k
\end{align}
for each integer $n\geq 0$ (see for example \cite[Chapter V, \S26--27]{Nielsen}). The above identities can be obtained by using~\eqref{eq: rec relation stirling}. We deduce from~\eqref{eq: generating functions Stirling} the following result.

\begin{lemma}\label{L seisitu}
    Let $n$ be a non-negative integer. Then,
    \begin{align*}
        \fL_K\left(\frac{(e^z-1)^n}{n!}\right)=\dfrac{z^n}{(1-z)\cdots(1-nz)}.
    \end{align*}
\end{lemma}

\medskip

\noindent\textbf{Formal modified Mellin transform.} The definition of our modified Mellin transforms are inspired by~\cite[D\'{e}finition~1]{B-D} and~\cite[Section $7$]{An2}, see Remark~\ref{remark inverse Mellin transform Andre} below. Recall that by Lemma~\ref{lemma: relation operator Delta and d/dz}, we have $K[[d/dz]] = K[[\Delta_1]]$.

\begin{definition}
    The correspondence $z \mapsto \Delta_1 = \exp(d/dz)-1$ defines an isomorphism of $K$--algebra
    \begin{align*}
        \hMellinInv_K: K[[z]] \longrightarrow K[[d/dz]].
    \end{align*}
    When there is no ambiguity, we simply write $\hMellinInv = \hMellinInv_K$.
\end{definition}

\begin{example}
    \label{example: mellin transform operators}
    We deduce from Lemma~\ref{lemma: relation operator Delta and d/dz} that for any $\alpha\in K$, we have
    \begin{align*}
        \hMellinInv\big((1+z)^\alpha\big) = \Shift_\alpha \AND \hMellinInv\big(\log(1+z)\big) = \dfrac{d}{dz}.
    \end{align*}
\end{example}

\begin{definition} \label{Mellin}
    We call \textsl{modified inverse Mellin transform} of power series the map
    \begin{align*}
        \MellinInv_K:K[[z]]\longrightarrow (1/z) \cdot K\left[\left[1/z\right]\right]; \ \ \ g(z)\mapsto \sum_{k=0}^{\infty}b_k\left(-\dfrac{1}{z}\right)^{k+1},
    \end{align*}
    where $(b_k)_{k\ge0}$ is defined by $g(e^z-1)=\sum_{k=0}^{\infty}b_kz^k/k!$. When there is no ambiguity, we simply write $\MellinInv = \MellinInv_K$.
\end{definition}

The transform $\MellinInv$ satisfies the following property.

\begin{proposition} \label{tenkai}
    Let $g(z)=\sum_{k=0}^{\infty}a_kz^k\in K[[z]]$. Then,
    \begin{align*}
        \MellinInv(g)(z)=\sum_{k=0}^{\infty}\dfrac{(-1)^{k+1} a_kk!}{z(z+1)\cdots(z+k)}.
    \end{align*}
\end{proposition}

\begin{proof}
    Define the morphism
    \begin{align*}
        \fT_K:K[[z]]\longrightarrow (1/z)\cdot K\left[\left[1/z\right]\right]; \ \ f(z)\longmapsto -\frac{1}{z}f\Big( -\frac{1}{z}\Big),
    \end{align*}
    so that $\MellinInv(g)(z) = \fT_K\circ \fL_K \big(g(e^z-1)\big)$. By Lemma~\ref{L seisitu}, we find
    \begin{align*}
        \fL_K \big(g(e^z-1)\big) = \sum_{k=0}^{\infty}a_k\fL_K\big(g(e^z-1)^k\big) = \sum_{k=0}^{\infty}\dfrac{a_k k! z^k}{(1-z)\cdots(1-kz)}.
    \end{align*}
    Then conclusion follows easily.
\end{proof}

\begin{remark}
    \label{remark inverse Mellin transform Andre}
    The ring of inverse factorial series (with complex coefficients) is
    \begin{align*}
        \C[!z!]:=\bigg\{\sum_{n\ge 0} \dfrac{a_n}{z(z+1)\cdots(z+n)} \,;\, a_n\in \C \bigg\},
    \end{align*}
    (it corresponds to the set $\C[!z!]^{(0)}$ with the notation of \cite[Section $7.2$]{An2}). Note that $\C[!z!] = (1/z)\cdot \C[[1/z]]$. In \cite[Section $7.2$]{An2} and \cite[Section~1.2]{B-D}, the authors consider the formal Mellin transform
    \begin{align*}
        \mathcal{M}:\C[!z!]\longrightarrow \C[[1-z]] ; \ \ \sum_{n\ge0} \dfrac{a_n}{z(z+1)\cdots(z+n)}\longmapsto \sum_{n\ge0} a_n\dfrac{(1-z)^n}{n!}.
    \end{align*}
    Proposition~\ref{tenkai} implies that $\MellinInv_{\C} = -\mathcal{M}^{-1} \circ \fT_{-1}$, where $\fT_{-1}: \C[[z]]\rightarrow \C[[1-z]]$ is the isomorphism given by $z\mapsto z-1$.
\end{remark}

We end this section with an analog of \cite[Proposition~3]{B-D}, which establishes a relation between the formal transforms $\hMellinInv$ and $\MellinInv$.

\begin{proposition} \label{key prop 0}
    For any $f(z), g(z) \in K[[z]]$, we have
    \begin{align}
        \label{eq: propo: diagram Mellin transform}
        \MellinInv\big(f(z)g(z)\big) = \hMellinInv(f(z))\big(\MellinInv(g(z)) \big).
    \end{align}
    This yields the following commutative diagram:
    \begin{align*}
        \begin{CD}
        K[[z]]\times K[[z]] @>{\hMellinInv\times \MellinInv} >> K\left[\left[{d}/{dz}\right]\right]\times (1/z)\cdot K \left[\left[1/z\right]\right]\\
        @VV{ }V @VV{ }V\\
        K[[z]] @>\MellinInv >> (1/z)\cdot K\left[\left[1/z\right]\right].
        \end{CD}
    \end{align*}
\end{proposition}

\begin{proof}
    Let $g(z)=\sum_{k=0}^{\infty}a_kz^k\in K[[z]]$. Since $\Delta_1(1/(z)_{k+1}) = -(k+1)/(z)_{k+2}$ for each integer $k\geq 0$, we have
    \begin{align*}
        \hMellinInv(z)(\MellinInv(g)(z)) =\Delta_1\left(\sum_{k=0}^{\infty}\dfrac{(-1)^{k+1} a_{k}k!}{z(z+1)\cdots(z+k)}\right)
                                         & =\sum_{k=0}^{\infty}\dfrac{(-1)^{k+2} a_{k}(k+1)!}{z(z+1)\cdots(z+k+1)} \\
                                         & = \MellinInv(zg)(z),
    \end{align*}
    the last equality coming from Proposition~\ref{tenkai}. By induction, we obtain $\hMellinInv(z^n)(\MellinInv(g)(z)) = \MellinInv(z^n g)(z)$ for each integer $n\geq 1$. Hence \eqref{eq: propo: diagram Mellin transform} (by linearity).
\end{proof}


\section{Formal $f$--integration transform}
\label{section: Pade approximation}

We keep the notation of Sections~\ref{Section: modified formal transforms}. In this section, we introduce and study the properties of the transform $\phi_f$, which will play a crucial role in the construction of our Pad\'{e} approximants (see \cite[Section~2]{Kaw}, \cite[Section~3]{KP} for other applications related to those maps).

\subsection{Notation and definitions}
\label{subsection: Pade approximation definition}

\begin{definition}
    We associate to any Laurent series $f(z) = {\sum_{k=0}^{\infty}}f_k/z^{k+1}\in (1/z)\cdot K[[1/z]]$, a $K$--linear morphism
    \begin{align} \label{phi f}
        \varphi_f:K[t]\longrightarrow K; \ \ \ t^k\mapsto f_k \ \ \ (k\ge0) .
    \end{align}
    We call the map $\phi_f$ the \textsl{formal $f$-integration transform}.
\end{definition}

\begin{remark}
    In the case $f(z) = -\log(1-1/z)$, the map $\phi_f$ is simply the operator $P(t)\mapsto \int_0^1 P(u)du$, which is the reason why we call it $f$-integration transform.
\end{remark}
Note that the $K$--linear map
\begin{align*}
    \Phi: (1/z)\cdot K[[1/z]]\longrightarrow {\rm{Hom}}_K(K[t],K)
\end{align*}
defined by $f\mapsto \varphi_f$ is an isomorphism. Given $f\in (1/z)\cdot K[[1/z]]$, the map $\varphi_f$ extends naturally in a $K[z]$-linear map $\varphi_f: K[z,t]\rightarrow K[z]$, and then to a $K((1/z))$-linear map $\varphi_f: L\rightarrow z^nK[[1/z]]$, where $L$ denotes ring which consists in all the elements of the form
\begin{align*}
     F(z,t) = \sum_{k = -n}^{\infty} \dfrac{P_k(t)}{z^{k}},
\end{align*}
with $n\in\Z$ and $P_k(t)\in K[t]$. Explicitly, for any element $F(z,t)$ as above, we have
\begin{align*}
    \phi_f(F(z,t)) = \sum_{k\geq -n} \frac{\phi_f(P_k(t))}{z^k} \in z^n K[[1/z]].
\end{align*}
With this notation, and seeing $1/(z-t) = \sum_{k\geq 1}t^k/z^{k+1}$ as an element of $L$, the formal Laurent series $f(z)$ satisfies the following crucial identities ({\it confer} \cite[(6.2) p.60 and (5.7) p.52]{N-S}):
\begin{align*}
    &f(z)=\varphi_f \left(\dfrac{1}{z-t}\right), \ \ \ P(z)f(z)-\varphi_f\left(\dfrac{P(z)-P(t)}{z-t}\right)\in (1/z)\cdot K[[1/z]],
\end{align*}
for any $P(z)\in K[z]$. Let us recall a condition, based on the morphism $\varphi_f$, for given polynomials to be Pad\'{e} approximants.

\begin{lemma}[{\it confer} {\rm{\cite[Lemma~2.3]{Kaw}}}] \label{equivalence Pade}
    Let $m, M$ be positive integers, $f_1(z),\ldots,f_m(z)\in (1/z)\cdot K[[1/z]]$ and $\boldsymbol{n}=(n_1,\ldots,n_m)\in \N^m$ with $\sum_{j=1}^m n_j \le M$. Let $P(z)\in K[z]$ be a non-zero polynomial with  $\deg P\le M$, and put $Q_j(z)=\varphi_{f_j}\left((P(z)-P(t))/(z-t)\right)\in K[z]$ for $1\le j \le m$. The following assertions are equivalent.
    \begin{enumerate}
      \item The vector of polynomials $(P,Q_1,\ldots,Q_m)$ is a weight $\boldsymbol{n}$ Pad\'{e}-type approximants of $(f_1,\ldots,f_m)$.
      \item  We have $t^kP(t)\in {\rm{ker}}\,\varphi_{f_j}$ for any pair of integers $(j,k)$ with $1\le j \le m$ and $0\le k \le n_j-1$.
    \end{enumerate}
\end{lemma}

\noindent\textbf{Notation.} In the next section, we will use the following operators. We use bold symbols to indicate when a map is defined on the ring $K[t]$ (in order to distinguish them from their analogues defined on $K[z]$ in Section~\ref{Section: modified formal transforms}). Fix an element $\alpha\in K$.

\begin{itemize}[label=$\bullet$]
    \item $[P]\in {\rm{End}}_{K\text{-lin}}(K[t])$ denotes the multiplication by the polynomial $P \in K[t]$. If there is no ambiguity, we will sometimes omit the brackets.
    \item We denote by $\Eval_{t=\alpha}\in {\rm{Hom}}_K(K[t],K)$ the $\alpha$--evaluation linear form, and by $\Mshift_{\alpha}\in {\rm{End}}_{K\text{-alg}}(K[t])$ (resp. $\bDelta_{\alpha} = \Mshift_\alpha- [1] \in {\rm{End}}_{K\text{-lin}}(K[t])$) the $\alpha$--shift (resp. $\alpha$--difference) operator. They are given by
        \begin{align*}
            \Eval_{t=\alpha}(P(t)) = P(\alpha), \quad \Mshift_{\alpha}(P(t)) = P(t+\alpha) \AND \bDelta_\alpha(P(t)) = P(t+\alpha)-P(t),
        \end{align*}
        for each $P(t)\in K[t]$.
    \item For any $f(z)\in K((1/z))$, we denote by $\pi(f(z))$ the unique element of $(1/z)\cdot K[[1/z]]$ such that $f(z)-\pi(f(z)) \in K[z]$. This defines a (surjective) projection
        \begin{align}
            \label{eq: def projection pi}
            \pi:K((1/z))\longrightarrow (1/z)\cdot K[[1/z]].
        \end{align}
    \item The map $\iota: K[z,\Shift_{\alpha}]\longrightarrow K[t,\Mshift_{-\alpha}]$ is defined for each $D = \sum_{i}[a_i(z)]\circ \Shift_{\alpha}\in K[z,\Shift_{\alpha}]$ by
        \begin{align*}
            D^* := \iota\Big(\sum_{i}[a_i(z)]\circ \Shift_{\alpha}\Big) =  \sum_{i}\Mshift_{-\alpha}\circ [a_i(t)],
        \end{align*}
        (this is a similar operator as the one in \cite[Exercise III(3)]{A}).
\end{itemize}
For any $D\in K[z,\Shift_{\alpha}]$, we view $D$ and $D^*$ as elements of $\End_{K\text{-lin}}(K[t][[1/z]])$ by setting for each $f(z,t) = \sum_{n\geq 0} P_k(t)/z^n \in K[t][[1/z]]$
\[
    D(f(z,t)) =  \sum_{n\geq 0} P_k(t)D\left(\frac{1}{z^n}\right) \AND D^*(f(z,t)) =  \sum_{n\geq 0} \frac{D^*(P_k(t))}{z^n}.
\]

\subsection{Properties of the operators $D$ and $D^*$}
\label{subsection: Pade preliminaries}

In the following lemma, we consider $1/(z-t) = \sum_{k\geq 0}t^k/z^{k+1}$ as an element of $K[t][[1/z]]$.

\begin{lemma} \label{lemma 1}
    Let $\alpha\in K$ and $D\in  K[z,\Shift_{\alpha}]$. Then there exists a polynomial $P(t,z)\in K[t,z]$ such that
    \begin{align}
        \label{eq: lemma relation D and D*}
        D\left(\dfrac{1}{z-t}\right)=P(t,z)+D^{*}\left(\dfrac{1}{z-t}\right).
    \end{align}
\end{lemma}

\begin{proof}
    By linearity, it suffices to prove the statement of $D=[z^m]\circ \Shift^{n}_{\alpha}$, where $m,n$ are non-negative integers.    We have
    \begin{align*}
        \Shift^{n}_{\alpha}\left(\dfrac{1}{z-t}\right) & =\Shift^{n}_{\alpha}\left(\sum_{k=0}^{\infty}\dfrac{t^k}{z^{k+1}}\right) = \sum_{k=0}^{\infty}\dfrac{t^k}{(z+n\alpha)^{k+1}}\\
        &=\sum_{k=0}^{\infty}\dfrac{t^k}{z^{k+1}}\left(\sum_{\ell=0}^{\infty}(-1)^{\ell}\binom{k+{\ell}}{k}\left(\dfrac{\alpha n}{z}\right)^{\ell}\right)=\sum_{k=0}^{\infty}\dfrac{(t-n\alpha)^k}{z^{k+1}},
    \end{align*}
    hence
    \begin{align}
        \label{eq proof: eq for relation D and D*}
        D\left(\dfrac{1}{z-t}\right)=P(t,z)+\sum_{k=0}^\infty\dfrac{(t-n\alpha)^{k+m}}{z^{k+1}}, \quad \textrm{where } P(t,z)=\sum_{k=0}^{m-1}(t-n\alpha)^kz^{m-k-1}.
    \end{align}
    (with the convention that $P(t,z) = 0$ if $m=0$). On the other hand
    \begin{align*}
        D^*\left(\dfrac{1}{z-t}\right) = \Mshift^{n}_{-\alpha}\circ[t^m]\left(\dfrac{1}{z-t}\right) = \sum_{k=0}^{\infty}\dfrac{(t-n\alpha)^{k+m}}{z^{k+1}}.
    \end{align*}
    Combining the above with \eqref{eq proof: eq for relation D and D*}, we obtain \eqref{eq: lemma relation D and D*}. This completes the proof of the lemma.
\end{proof}

\subsection{Properties of the transform $\phi_f$}

Lemma $\ref{lemma 1}$ allows us to show the following key proposition.

\begin{proposition} \label{key prop}
    Let $f(z)\in (1/z)\cdot K[[1/z]]$ and $D\in K[z,\Shift_{\alpha}]$. Then, viewing $\varphi_{\pi(D(f))}$ and $\varphi_f\circ D^{*}$ as elements of ${\rm{Hom}}_{K}(K[t],K)$, we have
    \begin{align*}
        \varphi_{\pi(D(f))}=\varphi_f\circ D^{*}.
    \end{align*}
\end{proposition}

\begin{proof}
    Let $P(t,z)\in K[t,z]$ be such that \eqref{eq: lemma relation D and D*} of Lemma~\ref{lemma 1} holds. Then, writing $P(z) = \phi_f(P(t,z))$ and since $\varphi_f$ acts only on the parameter $t$, we have
    \begin{align*}
        D(f(z))=D\circ \varphi_f\left(\dfrac{1}{z-t}\right)=\varphi_f \left(D\bigg(\dfrac{1}{z-t} \bigg)\right) & = P(z) + \varphi_f \left(D^*\bigg(\dfrac{1}{z-t} \bigg)\right) \\
        & = P(z)+\sum_{k=0}^{\infty}\dfrac{\varphi_f(D^{*}(t^k))}{z^{k+1}}.
    \end{align*}
    This shows that
    \begin{align*}
        \pi(D(f))=\sum_{k=0}^{\infty}\varphi_f(D^{*}(t^k))/{z^{k+1}} \ \ \text{and} \ \  \varphi_{\pi(D(f))}(t^k)=\varphi_f\circ D^{*}(t^k) \ \ \ \text{for all} \ k\ge 0.
    \end{align*}
    This completes the proof of Proposition \ref{key prop}.
\end{proof}

We deduce several interesting consequences from Proposition~\ref{key prop}.

\begin{corollary} \label{inc}
    Let $f(z)\in (1/z)\cdot K[[1/z]]$ and $D\in K[z,\Shift_{\alpha}]\setminus\{0\}$. The following assertions are equivalent.
    \begin{enumerate}
      \item \label{item: cor D, pi, phi_f: item 1} $D(f(z))\in K[z]$.
      \item \label{item: cor D, pi, phi_f: item 2} $D^{*}(K[t])\subseteq \ker \varphi_f$.
    \end{enumerate}
\end{corollary}

\begin{proof}
    Conditions \ref{item: cor D, pi, phi_f: item 1} and \ref{item: cor D, pi, phi_f: item 2} are equivalent to $\pi(D(f))=0$ and $\varphi_f \circ D^{*}=0$, respectively. Those last two assertions are clearly equivalent by Proposition \ref{key prop}.
\end{proof}

\begin{corollary}
    \label{cor: useful for phi_alpha,s}
    Let $f(z)\in (1/z)\cdot K[[1/z]]$ and $\alpha\in K$, and set $g(z) = \Shift_\alpha(f(z))\in (1/z)\cdot K[[1/z]]$. Then
    \begin{align*}
        \phi_{g} = \phi_f\circ \Mshift_{-\alpha}.
    \end{align*}
\end{corollary}
\begin{proof}
    As discussed at the beginning of Section~\ref{Section: modified formal transforms}, we have $g(z)\in (1/z)\cdot K[[1/z]]$, so that $\pi(g(z)) = g(z)$. Set $D=\Shift_\alpha$. By definition $D^* = \Mshift_{-\alpha}$, and we conclude by Proposition~\ref{key prop}.
\end{proof}

We also deduce from Proposition~\ref{key prop} the following results.

\begin{lemma}
    \label{lem: values of phi_f}
    Let $f\in (1/z)\cdot K[[1/z]]$ and write
    \begin{align*}
        f(z) = \sum_{k\geq 0} \frac{a_k k!}{z(z+1)\cdots (z+k)}.
    \end{align*}
    Then, for each $j\geq 0$, we have
    \begin{align*}
        \varphi_f\left(\frac{(t)_j}{j!}\right) = a_j.
    \end{align*}
\end{lemma}

\begin{proof}
    Recall that $\pi:K[z][[1/z]]\to (1/z)\cdot K[[1/z]]$ is the morphism defined in \eqref{eq: def projection pi}. Fix an integer $j\geq 0$ and set $D=(z)_j/j!$. Then $D^* = (t)_j/j!$, and Proposition~\ref{key prop} gives
    \begin{align}
        \label{eq proof: lem f g: eq2}
        \varphi_f\left(\frac{(t)_j}{j!}\right) = \phi_f(D^*(1)) = \varphi_{\pi(D(f))}(1).
    \end{align}
    On the other hand, since $(z)_j/(z)_{k+1}\in K[z]$ for $k=0,\dots,j-1$, we have 
    \begin{align*}
        \pi(D(f)) = \dfrac{1}{j!}\sum_{k=j}^{\infty}\dfrac{a_{k}k!}{(z+j)(z+j+1)\cdots(z+k)} \in \dfrac{a_j}{z}+\dfrac{1}{z^2}K[[1/z]],
    \end{align*}
    hence $\varphi_{\pi(D(f))}(1)=a_j$. Combined with \eqref{eq proof: lem f g: eq2}, this proves the lemma.
\end{proof}

Combined with Proposition~\ref{tenkai}, the above lemma has the following consequence.
\begin{corollary}
    \label{cor: values of phi_f}
    Let $g(z)=\sum_{k=0}^{\infty}a_kz^k\in K[[z]]$, and set $f(z) = \MellinInv(g)(z) \in (1/z)\cdot K[[1/z]]$.
    For each $j\geq 0$, we have
    \begin{align*}
        \varphi_f\left(\frac{(t)_j}{j!}\right) = (-1)^{j+1}a_j.
    \end{align*}
\end{corollary}

\section{Mellin transform and $p$--adic polygamma functions}
\label{section: Mellin transform and polygamma}

We keep the notation of Section~\ref{Section: modified formal transforms}. Denote by $\MellinInv = \MellinInv_K$ the inverse modified Mellin transform (see Definition~\ref{Mellin}). In this section, we introduce and study a family of formal series $R_{\alpha,s}(z)\in (1/z)\cdot K[[1/z]]$ connected to the $p$--adic polygamma functions. We also prove that the series $R_{\alpha,s}$ satisfies a rather simple difference equation.

\begin{definition}
    \label{def: R_alpha,s}
    Given an integer $s\geq 2$ and $\alpha\in K$, set
    \begin{align*}
         R_{\alpha,s}(z) & = \MellinInv\left(\dfrac{(1+z)^\alpha\log^{s-1}(1+z)}{z}\right), \\
         R_{\alpha,1}(z) & = \MellinInv\left(\dfrac{(1+z)^{\alpha}-1}{z}\right).
    \end{align*}
    For $\alpha=0$, we simply write
    \begin{align*}
        R_s(z) = R_{0,s}(z) = \MellinInv\left(\dfrac{\log^{s-1}(1+z)}{z}\right).
    \end{align*}
\end{definition}

In the following lemma we show that $R_s(z)\in (1/z)\cdot K[[1/z]]$ is the formal series corresponding to the series~\eqref{eq: series expansion of R_s} of the $p$--adic polygamma functions, which is why we are using the same notation.

\begin{lemma}
    \label{lem: Mellin R_alpha,s}
    With the notation of Definition~\ref{def: R_alpha,s}, we have
    \begin{align*}
         R_s(z) & = \sum_{k=s-2}^{\infty}(k-s+3)_{s-2}B_{k-s+2}\cdot\frac{(-1)^{k+1}}{z^{k+1}}, \\
         R_{\alpha,s}(z) & = \Shift_\alpha(R_s(z)) = R_s(z+\alpha), \\
         R_{\alpha,1}(z) & = \sum_{k=0}^{\infty}\frac{B_{k+1}(\alpha)-B_{k+1}}{k+1}\cdot\frac{(-1)^{k+1}}{z^{k+1}}.
    \end{align*}
    Furthermore
    \begin{enumerate}
      \item The series $R_{\alpha,s}(x)$ converges $p$--adically for each $x,\alpha\in\C_p$ with $|x+\alpha|_p > 1$.
      \item The series $R_{\alpha,1}(x)$ converges $p$--adically for each $x,\alpha\in\C_p$ with $|x|_p > |\alpha|_p$.
    \end{enumerate}
\end{lemma}

\begin{proof}
    Write $g(z) = \log^{s-1}(1+z)/z$ so that $R_s(z)=\MellinInv(g)(z)$. Using the generating function of Bernoulli numbers \eqref{eq: generating function B_k(x)}, we find
    \begin{align*}
        g(e^z-1) =\frac{z^{s-1}}{e^z-1} = \sum_{k=s-2}^{\infty}(k-s+3)_{s-2}B_{k-s+2}\cdot \dfrac{z^{k}}{k!}.
    \end{align*}
    The expansion of $R_s(z)$ follows from the definition of $\MellinInv$. The second identity is then a direct consequence of Lemma~\ref{func eq g(1+z)a} below. For the third identity, write $h(z)=\big((1+z)^\alpha-1\big)/z$. Again, using \eqref{eq: generating function B_k(x)}, we find
    \begin{align*}
        h(e^z-1) =\frac{e^{\alpha z}-1}{e^z-1} = \frac{1}{z}\sum_{k=0}^{\infty}\big(B_{k}(\alpha)-B_k\big)\cdot \dfrac{z^{k}}{k!} = \sum_{k=0}^{\infty}\big(B_{k+1}(\alpha)-B_{k+1}\big)\cdot \dfrac{z^{k}}{(k+1)!}.
    \end{align*}
    The expected formula follows from the definition of $\MellinInv$, since $R_{\alpha,1}(z) = \MellinInv(h)(z)$. The {convergence} in $\C_p$ of the series $R_{\alpha,s}(x)$ and $R_{\alpha,1}(x)$ {is} a consequence of \eqref{eq: valuation p-adic Bernoulli}.
\end{proof}

We deduce from Proposition $\ref{key prop 0}$ the following crucial identities.

\begin{lemma}
    \label{func eq g(1+z)a}
    Let $g(z) = \sum_{n\geq 0}a_nz^n \in K[[z]]$ and $\alpha\in K$. We have
    \begin{align*}
        \MellinInv\big(\log(1+z) g\big)(z) &= \dfrac{d}{dz}\big(\MellinInv(g)(z)\big),\\
        \MellinInv\big((1+z)^{\alpha} g \big)(z) &=\Shift_{\alpha}(\MellinInv(g)(z)), \\
        \MellinInv\big(((1+z)^{\alpha}-1) g \big)(z) &=\Delta_{\alpha}(\MellinInv(g)(z)).
    \end{align*}
\end{lemma}

\begin{proof}
    This follows from Proposition~\ref{key prop 0} combined respectively with
    \[
        \hMellinInv(\log(1+z))=d/dz \AND \hMellinInv((1+z)^{\alpha})=\Shift_{\alpha}
    \]
    coming from Lemma~\ref{lemma: relation operator Delta and d/dz}.
\end{proof}

\begin{proposition}[Difference equation of the Laurent series $R_{\alpha,s}$]
    \label{prop: difference equation R_alpha,s}
    Let $s\geq 2$ be an integer and $\alpha\in K$. We have
    \begin{align} \label{diff eq}
        \Delta_{1} (R_{\alpha,1}(z))=\dfrac{\alpha}{z(z+\alpha)} \AND \Delta_{1} (R_{\alpha,s}(z))=\dfrac{(-1)^s(s-1)!}{(z+\alpha)^{s}}.
    \end{align}
\end{proposition}

\begin{proof}
    According to Proposition~\ref{key prop 0}, and since $\Delta_1 = \hMellinInv(z)$, we have
    \begin{align*}
        \Delta_1(R_{\alpha,1}(z))= \hMellinInv(z)\Big(\MellinInv\left(\frac{(1+z)^{\alpha}-1}{z}\right) \Big) = \MellinInv\left((1+z)^{\alpha}-1\right) = \Delta_\alpha(\MellinInv(1)),
    \end{align*}
    the last inequality coming from Lemma~\ref{func eq g(1+z)a}. Combined with $\MellinInv(1) = -1/z$, this yields the first equality of~\eqref{diff eq}. We proceed in a similar way for the second equality. Note that since $R_{\alpha,s} = \Shift_{\alpha}(R_s(z))$ and $\Delta_1\circ \Shift_{\alpha}=\Shift_{\alpha} \circ \Delta_1$, we only have to prove the case $\alpha=0$. Using Proposition \ref{key prop 0}, we obtain
    \begin{align*} 
        \Delta_1(R_s(z))=\hMellinInv(z)\MellinInv\left(\dfrac{\log^{s-1}(1+z)}{z}\right)=
        \MellinInv\left(\log^{s-1}(1+z)\right)= \dfrac{d^{s-1}}{dz^{s-1}}\MellinInv(1),
    \end{align*}
    the last equality coming from the first identity of Lemma~\ref{func eq g(1+z)a}. We conclude by using $\MellinInv(1) = -1/z$.
\end{proof}

\section{Properties of the difference operator}
\label{section: prop of diff operator}

We keep the notation of Section~\ref{section: Pade approximation}. Recall that for any $P(t)\in K[t]$, we denote by $[P]$ the operator ``multiplication by $P(t)$''. The difference operator $\bDelta_{-1} = \Mshift_\alpha- [1]$ will be involved in the construction of the Pad\'{e} approximants in Section~\ref{subsection: Pade approximants for polygamma}. This motivates us to study its properties.

\begin{lemma}
    \label{lem: difference operator of product}
    For any positive integer $n$ and for any polynomial $P(t)\in K[t]$, we have
    \begin{align*}
        \bDelta^{n}_{-1}\circ [P(t)] = \sum_{k=0}^{n}\binom{n}{k} \Big[\Mshift^{k}_{-1} \circ \bDelta^{n-k}_{-1}(P(t))\Big]\circ  \bDelta^{k}_{-1}.
    \end{align*}
\end{lemma}

\begin{proof}
    Let $P(t), Q(t)\in K[t]$. We proceed by induction on $n$. For $n=1$, we easily check that
    \begin{align*}
        \bDelta_{-1}\big(P(t)Q(t)\big) = \bDelta_{-1}(P(t))\cdot Q(t) + \Mshift_{-1}(P(t))\cdot \bDelta_{-1}(Q(t)).
    \end{align*}
    Suppose now that the lemma is true for a positive integer $n$. Then
    \begin{align*}
        \bDelta_{-1}^{n+1}\big(P(t)Q(t)\big) & = \bDelta_{-1}^{n}\Big(\bDelta_{-1}\big(P(t)Q(t)\big) \Big) \\
        & = \bDelta_{-1}^{n}\Big( \bDelta_{-1}(P(t))\cdot Q(t) + \Mshift_{-1}(P(t))\cdot \bDelta_{-1}(Q(t)) \Big).
    \end{align*}
    We then get the result by applying the induction hypothesis with the pairs of polynomials $\big( \bDelta_{-1}(P(t)), Q(t)\big)$ and $\big(\Mshift_{-1}(P(t)), \bDelta_{-1}(Q(t)) \big)$ (and by using the commutativity of $\bDelta_{-1}$ and $\Mshift_{-1}$).
\end{proof}

\begin{lemma} \label{key 2 bis}
    Let $n$ be a positive integer. For any polynomial $P(t)\in K[t]$, we have
    \begin{align*}
        [P(t)]\circ \bDelta^{n}_{-1}=\sum_{j=0}^n\binom{n}{j} \bDelta^{n-j}_{-1} \circ \left[{\Mshift}^{n-j}_1\circ \bDelta^{j}_1(P(t))\right].
    \end{align*}
\end{lemma}

\begin{proof}
    We proceed by induction on $n$. For $n=1$, a direct computation ensures that for any polynomial $P(t)\in K[t]$, we have
    \begin{align}
    \label{eq proof: lem key 2: eq 1}
        [P(t)]\circ \bDelta_{-1} = \bDelta_{-1}\circ[\Mshift_1(P(t))] + [\bDelta_1(P(t))].
    \end{align}
    Let $n\geq 1$ be such that the assertion is true, and $P(t)\in K[t]$. By our induction hypothesis,
    \begin{align}
        [P(t)]\circ \bDelta^{n+1}_{-1} & =\sum_{j=0}^n\binom{n}{j} \bDelta^{n-j}_{-1}\circ \left[{\Mshift}^{n-j}_1\circ  \bDelta^{j}_1(P(t))\right]\circ \bDelta_{-1} \nonumber\\
        &=\sum_{j=0}^n\binom{n}{j} \bDelta^{n-j}_{-1}\circ \left(\bDelta_{-1} \circ \left[{\Mshift}^{n+1-j}_1\circ \bDelta^{j}_1(P(t))\right]+ \left[{\Mshift}^{n-j}_1\circ  \bDelta^{j+1}_1(P(t))\right] \right) \label{second eq}\\
        &=\sum_{j=0}^{n+1}\binom{n+1}{j} \bDelta^{n+1-j}_{-1}\circ \left[{\Mshift}^{n+1-j}_1\circ  \bDelta^{j}_1(P(t))\right], \nonumber
    \end{align}
    where we obtain \eqref{second eq} by applying \eqref{eq proof: lem key 2: eq 1} to the polynomial ${\Mshift}^{n-j}_1\circ  \bDelta^{j}_1(P(t))$, and by using the commutativity of $\bDelta_1$ and $\Mshift_1$. This concludes our induction step.
\end{proof}

\begin{lemma} \label{key 2}
    Let $n,d,m_1,\dots,m_d$ be non-negative integers with $n,d\geq 1$, and $\alpha_1,\dots,\alpha_d \in K$. Set
    \begin{align*}
        A(t) = \prod_{i=1}^{d}(t+\alpha_i)^{m_i+1} \AND A_n(t) = \prod_{i=1}^{d}(t+\alpha_i)_n^{m_i+1}.
    \end{align*}
    Suppose that $ P(t)\in A_n(t) K[t]$. Then, we have the following properties.
    \begin{enumerate}
        \item \label{item: lem key 2: item 1 bis} For each $k=0,\dots,n-1$, we have
          \begin{align*}
                \bDelta^{k}_{-1}(P(t))\in A(t) K[t].
          \end{align*}
        \item \label{item: lem key 2: item 2 bis} For any polynomial $Q(t)\in K[t]$, we have
          \begin{align*}
                Q(t)\bDelta^{n}_{-1}(P(t))\in \bDelta_{-1}\big(A(t)K[t]\big) + P(t)\bDelta^{n}(Q(t)).
          \end{align*}
        \item \label{item: lem key 2: item 3 bis} For any polynomial $Q(t)\in K[t]$ with $\deg Q(t) < n$, we have
          \begin{align*}
                Q(t)\bDelta^{n}_{-1}(P(t))\in \bDelta_{-1} \big(A(t)K[t] \big).
          \end{align*}
    \end{enumerate}
\end{lemma}

\begin{proof}
    \ref{item: lem key 2: item 1 bis}. Recall the identity $\bDelta_{-1}^{k}=\sum_{i=0}^k\binom{k}{i}(-1)^{k-i}\Mshift^{i}_{-1}$. Write $P(t) = A_n(t)R(t)$, with $R(t)\in K[t]$, and fix an integer $k$ with $0\leq k < n$. Then, we obtain
    \begin{align}
        \label{eq proof: lem key 2: eq 2}
        \bDelta_{-1}^{k}(P(t))=\sum_{i=0}^k\binom{k}{i}(-1)^{k-i}Q(t-i)\prod_{j=1}^{d}(t+\alpha_j-i)^{m_j}_n.
    \end{align}
    We conclude by noticing that since $k<n$, for each $0\le i \le k$, the polynomial $\prod_{j=1}^{d}(t+\alpha_j-i)^{m_j}_n$ is divisible by $\prod_{j=1}^{d}(t+\alpha_j)^{m_j} = A(t)$.

    \bigskip

    \ref{item: lem key 2: item 2 bis}. Fix a polynomial $Q(t)\in K[t]$, and define
    \[
        R_j(t) = P(t)\cdot{\Mshift}^{n-j}_1\circ \bDelta^{j}_1\big(Q(t)\big) \in A_n(t) K[t]
    \]
    for $j=0,\dots,n-1$. By Lemma~\ref{key 2 bis}, we have
    \begin{align*}
        Q(t) \bDelta^{n}_{-1}(P(t)) 
        &= \bDelta_{-1}\bigg(\sum_{j=0}^{n-1}\binom{n}{j} \bDelta^{n-1-j}_{-1} \big(R_j(t)\big)\bigg) + P(t)\bDelta_1^{n}(Q(t)).
    \end{align*}
    Applying \ref{item: lem key 2: item 1 bis} to the polynomial $R_j(t)$, we find $\bDelta^{n-1-j}_{-1}\big(R_j(t)\big)  \in A(t)K[t]$ for $j=0,\dots,n-1$, hence \ref{item: lem key 2: item 2 bis}.

    \medskip

    Finally, \ref{item: lem key 2: item 3 bis} is a direct consequence of \ref{item: lem key 2: item 2 bis}, since $\bDelta_1^{n}(Q(t))=0$ as soon as $\deg Q(t) < n$.
\end{proof}

\section{Construction of Pad\'{e} approximants}
\label{subsection: Pade approximants for polygamma}

We keep the notation of Sections~\ref{Section: modified formal transforms} and~\ref{section: Pade approximation}. This section is devoted to the explicit construction of Pad\'{e}-type approximants for the Laurent series $R_{\alpha,s}(z)$ introduced in Section~\ref{section: Mellin transform and polygamma}. Let $d,m_1,\ldots,m_d$ be positive integers and $\balpha=(\alpha_1,\cdots,\alpha_d)\in K^d$ with $\alpha_1=0$. 
Denote by $\Indset$ the set of indices
\begin{align*}
    \Indset=\{(i,s) \, ; \, 1\le i \le d \AND 1\le s \le m_i+1\} \setminus\{(1,1)\},
\end{align*}
and set
\begin{align*}
    M= \#\Indset = d-1+\sum_{i=1}^dm_i.
\end{align*}
For simplicity, for any $\alpha\in\Q$ and any positive integer $s$, we write
\begin{align}
    \label{eq: def phi_i,s}
    \phi_{\alpha,s} = \phi_{R_{\alpha,s}},
\end{align}
where $R_{\alpha,s}(z)$ as in Definition~\ref{def: R_alpha,s}. We also introduce the following formal series of $K[[z]]$.

\begin{definition}
    \label{def: functions g_i,s}
    For any $(i,s)\in\Indset$, set
    \begin{align*}
        g_{i,s}(z) := \sum_{k=0}^{\infty} a_{i,s,k}z^k =
        \begin{cases}
            \dfrac{(1+z)^{\alpha_i}\log^{s-1}(1+z)}{z} & \ \text{if}  \ s \geq 2\\
            \dfrac{(1+z)^{\alpha_i}-1}{z} & \ \text{if}  \ i \geq 2 \ \text{and} \ s=1.
        \end{cases}
    \end{align*}
\end{definition}

If follows easily from that definition that $a_{1,s,0} = \cdots = a_{1,s,s-3} = 0$, and
\begin{align}
    \label{eq: formula coeff a_i,s,k}
    a_{i,s,k} = \left\{
    \begin{array}{ll}
      \displaystyle\sum_{\substack{\ell_i\ge 0 \\ \ell_1+\cdots+\ell_{s-1}=k-s+2}}\dfrac{(-1)^{k-s+2}}{(\ell_1+1)\cdots (\ell_{s-1}+1)} & \textrm{if $i=1$ and $s\geq 2$ and $k\geq s-2$}, \\
      \displaystyle \sum_{j=0}^k\binom{\alpha_i}{j}a_{1,s,k-j} & \textrm{if $i\geq 2$ and $s\geq 2$}, \\
      \displaystyle \binom{\alpha_i}{k+1} & \textrm{if $i\geq 2$ and $s = 1$}.
    \end{array}\right.
\end{align}

We will bound from above the absolute value of the coefficients $a_{i,s,k}$ at the end of the present section. Their $p$--adic absolute values are estimated in the proof of Lemma~\ref{estimate phi_is}, while their denominators are studied in Lemma~\ref{den ask 0}. We have the following key-properties.

\begin{lemma}
    \label{lem: values phi_i,s with g_i,s}
    For each $(i,s)\in \Indset$ and each integer $k\geq 0$, we have
    \begin{align*}
        R_{\alpha_i,s}(z)=\MellinInv(g_{i,s}) \AND \phi_{\alpha_i,s}\bigg(\frac{(t)_k}{k!} \bigg) = (-1)^{k+1}a_{i,s,k}.
    \end{align*}
\end{lemma}

\begin{proof}
    The first part is simply equivalent to the definition of $R_{\alpha_i,s}$. Then, we deduce the last part as a consequence of Corollary~\ref{cor: values of phi_f}.
\end{proof}

\begin{theorem} \label{Pade R}
    Let $\ell,n$ be non-negative integers with $0\leq \ell \leq M$. For any $(i,s)\in\Indset$, define the polynomials
    \begin{align}
         A_{n,\ell}(z) &= A_{\ell}(z)=(-1)^{\ell}\dfrac{(z)_{\ell}}{\ell !} \prod_{i=1}^d\left((-1)^n\dfrac{(z+\alpha_i)_n}{n!}\right)^{m_i+1}, \notag \\
        P_{n,\ell}(z) &=P_{\ell}(z)=\Delta^{n}_{-1}\left(A_{\ell}(z)\right), \label{P}\\
        Q_{n,i,s,\ell}(z) &=Q_{i,s,\ell}(z)=\phi_{\alpha_i,s}\left(\dfrac{P_{\ell}(z)-P_{\ell}(t)}{z-t}\right).  \label{Q}
    \end{align}
    \begin{enumerate}
      \item \label{item: thm Pade R: item 1} The vector of polynomials $\big(P_{n,\ell}(z), Q_{n,i,s,\ell}(z)\big)_{(i,s)\in\Indset}$ forms a weight $(n,\ldots,n)\in \N^{M}$ Pad\'{e}-type approximant of $\big(R_{\alpha_i,s}(z)\big)_{(i,s)\in\Indset}$.
      \item \label{item: thm Pade R: item 2} We have the explicit formulas
             \begin{align*}
                P_{n,\ell}(z) &= \sum_{k=0}^n\binom{n}{k}(-1)^{n-k+\ell}\dfrac{(z-k)_\ell}{\ell!}A_0(z-k) = \sum_{j=0}^{Mn+\ell} p_{n,j,\ell} \dfrac{(z)_{j}}{j!}, \\
                Q_{n,i,s,\ell}(z) &=\sum_{j=1}^{Mn+{\ell}} p_{n,j,\ell}\left(\sum_{k=0}^{j-1}\dfrac{(-1)^{k+1}a_{i,s,k}\cdot k!}{(z)_{k+1}} \right)\dfrac{(z)_j}{j!},
            \end{align*}
            where for each integer $j$ with $0 \leq j\leq Mn+\ell$, the coefficient $p_{n,j,\ell}$ is given by
            \begin{align*}
                p_{n,j,\ell} = p_{j,\ell} = \sum_{k=n}^{j+n}\binom{j+n}{k}(-1)^{n-k}\binom{k}{\ell} \prod_{r=1}^d{\binom{k-\alpha_r}{n}}^{m_r+1}.
            \end{align*}
      \item \label{item: thm Pade R: item 3} For each $(i,s)\in \Indset$, denote by
            \begin{align*}
                \fR_{n,i,s,\ell}(z) = \fR_{i,s,\ell}(z) &= P_{n,\ell}(z)R_{\alpha_i,s}(z)-Q_{n,i,s,\ell}(z)
            \end{align*}
            the Pad\'{e} approximation of $R_{\alpha_i,s}$. Then
            \begin{align*}
                \fR_{n,i,s,\ell}(z)= \sum_{k=n}^{\infty}\dfrac{k!}{(k-n)!}\dfrac{\phi_{\alpha_i,s}\left((t+n)_{k-n} A_{n,\ell}(t)\right)}{z(z+1)\cdots(z+k)}.
            \end{align*}
    \end{enumerate}
\end{theorem}

\begin{remark}
    In the case of $d=1$ and $\ell=0$, explicit Pad\'{e} approximants of $R_2(z)$ have been studied by T.~J.~Stieltjes \cite{Sti}, J.~Touchard \cite{T} and L.~Carlitz \cite{Car} ({\it{confer}} \cite[Section $8$]{B}).
\end{remark}

It is worth mentioning that our construction relies on the difference equation \eqref{diff eq} satisfied by $R_{\alpha_i,s}(z)$. We will establish it through a difference analogue of the classical Rodrigues formula for orthogonal polynomials. In order to use Lemma~\ref{equivalence Pade}, we will study the kernel of $\phi_{\alpha_i,s}$. The methodology we present below investigates the Rodrigues formula for orthogonal polynomial systems, as discussed in \cite{Kaw}, in the context of a difference equation.

\medskip

Fix $(i,s)\in\Indset$. According to \eqref{diff eq} the Laurent series $R_{\alpha_i,s}(z)$ satisfies the following difference equation:
\begin{align*}
    (z+\alpha_i)^{s}\Delta_{1} (R_{\alpha_i,s}(z)) &=(-1)^{s}(s-1)!  & \textrm{if }  s\ge 2, \\
    z(z+\alpha_i) \Delta_{1} (R_{i,1}(z)) &=-\alpha_i & \textrm{if }  i\ge 2.
\end{align*}
Applying Corollary~\ref{inc} to $D =  [(z+\alpha_i)^{s}]\circ \Delta_{1}$ and $D =  [z(z+\alpha_i)]\circ \Delta_{1}$, we deduce from the above relation that
\begin{align}
    \bDelta_{-1}\big((t+\alpha_i)^{s}K[t] \big)) & \subseteq \ker \,\phi_{\alpha_i,s}  & \text{if }  s\ge 2, \label{subseteq 1}\\
    \bDelta_{-1}\big(t(t+\alpha_i)K[t] \big) & \subseteq \ker\,\phi_{\alpha_i,1}  & \text{if }  i\ge 2 \label{subseteq 2}.
\end{align}

The exact kernel of $\phi_{\alpha_i,s}$ will be determined in the next section, see Lemmas~\ref{lem: intersection kernels} and~\ref{lem: kernel for s=1}.

\begin{lemma}
    \label{lem: t^kP_n in ker phi_i,s}
    Let $n,\ell$ be integers with $1\leq n$ and $0\leq \ell \leq M$. Then, for any $(i,s)\in \Indset$, we have
    \begin{align} \label{contain ker}
        t^kP_{n,\ell}(t)\in {\rm{ker}}\,\phi_{\alpha_i,s} \qquad (0\le k \le n-1).
    \end{align}
\end{lemma}

\begin{proof}
    Lemma \ref{key 2}~\ref{item: lem key 2: item 3 bis} implies that for $k=0,\dots,n-1$, we have
    \begin{align*}
        t^kP_{n,\ell}(t)=[t^k]\circ \bDelta^{n}_{-1}(A_{n,\ell}(t))\in \bDelta_{-1}\circ[A(t)](K[t]),
    \end{align*}
    where $A(t) = \prod_{r=1}^{d}(t+\alpha_r)^{m_r+1}$. Combining the above with \eqref{subseteq 1} and \eqref{subseteq 2} (and since $\alpha_1 = 0$) we deduce \eqref{contain ker}.
\end{proof}

\begin{proof} [\textbf{Proof of Theorem~\ref{Pade R}~\ref{item: thm Pade R: item 1}}]
    The polynomial $A_{n,\ell}(z)$ has degree $n(M+1)+\ell$. The equality $\deg \Delta_{-1}(P)=\deg P-1$ valid for any $P\in K[z]$, implies that $\deg P_{n,\ell}(z)=Mn+\ell$. We conclude by combining Lemma~\ref{lem: t^kP_n in ker phi_i,s} and~\ref{equivalence Pade}.
\end{proof}

\begin{lemma} \label{lem: decompo pol on (z)_j}
    For any polynomial $A(z)\in K[t]$ and any integer $n\geq 0$, we have
    \begin{align}
        A(z) &=\sum_{j \geq 0} (-1)^j p_j \frac{(z)_j}{j!}, \label{delta P eq 1} \\
        \Delta^{n}_{-1}(A(z)) &= \sum_{i=0}^n\binom{n}{i}(-1)^{n-i} A(z-k) =\sum_{j \geq 0} (-1)^{j} p_{j+n} \dfrac{(z)_j}{j!},\label{delta P eq 2}
    \end{align}
    where
    \begin{align} \label{delta P bis}
        p_j = \Eval_{z=0}\circ \Delta^{j}_{-1}(A(z)) =\sum_{i=0}^j\binom{j}{i}(-1)^{j-i}A(-i).
    \end{align}
\end{lemma}

\begin{proof}
    The first equality of \eqref{delta P eq 2} and the second equality of \eqref{delta P bis} come from the identity  $\Delta^{j}_{-1}=\sum_{i=0}^j\binom{j}{i}(-1)^{j-i}\Shift^{i}_{-1}$. We deduce \eqref{delta P eq 1} with $p_j = \Eval_{z=0}\circ \Delta^{j}_{-1}(A(z))$ by using
    \begin{align*}
        \Delta_{-1}\left((-1)^j\dfrac{(z)_j}{j!}\right)=(-1)^{j-1}\dfrac{(z)_{j-1}}{(j-1)!} \quad (j\geq 1).
    \end{align*}
    The last equality of \eqref{delta P eq 2} is obtained in a similar way.
\end{proof}

\begin{lemma}
    \label{lem: eq (z)_j-(t)_j}
    For each integer $j\geq 0$, we have
    \begin{align*}
        \frac{(z)_j-(t)_j}{z-t} = (z)_j \sum_{k=0}^{j-1} \frac{(t)_k}{(z)_{k+1}}.
    \end{align*}
\end{lemma}

\begin{proof}
    We proceed by induction on $j$. For $j=0$, both sides are equal to $0$. For the induction step, it suffices to use the identity
    \begin{align*}
        \frac{(z)_{j+1}-(t)_{j+1}}{z-t} = (z+j)\bigg(\frac{(z)_{j}-(t)_{j}}{z-t}\bigg) + (t)_j.
    \end{align*}
\end{proof}

\begin{proof} [\textbf{Proof of Theorem~\ref{Pade R}~\ref{item: thm Pade R: item 2}}] 
    By Lemma~\ref{lem: decompo pol on (z)_j}, we have
    \begin{align*}
        A_{\ell}(z) = \sum_{j=0}^{n(M+1)+\ell} a_{\ell,j}(-1)^j \frac{(z)_j}{j!}
    \end{align*}
    where
    \begin{align} \label{amnlk}
        a_{\ell,j}=\sum_{i=0}^j\binom{j}{i}(-1)^{j-i}A_{\ell}(-i) = \sum_{i=n}^j\binom{j}{i}(-1)^{j-i}\binom{i}{\ell}\prod_{k=1}^d\binom{i-\alpha_k}{n}^{m_k+1},
    \end{align}
    where $A_\ell(-i)=0$ for each $i < n$ (since $\alpha_1=0$). The expected formula for $A_{\ell}(z)$ follows. Lemma~\ref{lem: decompo pol on (z)_j} together with \eqref{amnlk} gives the explicit formula for $P_{\ell}(z) = \Delta_{-1}^n(A_\ell(z))$. Finally, using Lemma~\ref{lem: eq (z)_j-(t)_j}, we find
    \begin{align*}
        Q_{i,s,\ell} = \phi_{\alpha_i,s}\left(\frac{P(z)-P(t)}{z-t} \right) = \sum_{j=0}^{Mn+\ell} \frac{p_{j,\ell}}{j!} \phi_{\alpha_i,s}\left(\frac{(z)_j-(t)_j}{z-t} \right) = \sum_{j=1}^{Mn+\ell} \frac{p_{j,\ell}}{j!} (z)_j \sum_{k=0}^{j-1} \frac{k!}{(z)_{k+1}} \phi_{\alpha_i,s}\left(\frac{(t)_k}{k!} \right),
    \end{align*}
    and we conclude by Lemma~\ref{lem: values phi_i,s with g_i,s}. This completes the proof of Theorem~\ref{Pade R}~\ref{item: thm Pade R: item 2}.
\end{proof}

\begin{proof} [\textbf{Proof of Theorem~\ref{Pade R}~\ref{item: thm Pade R: item 3}}] 
    Fix $(i,s)\in \Indset$ and for simplicity, set $f(z) = \fR_{i,s,\ell}(z)$. Write
    \begin{align*}
        f(z) = \sum_{k=n}^{\infty}\dfrac{r_{i,s,\ell,k} k!}{z(z+1)\cdots(z+k)}
    \end{align*}
    (which is possible since according to Theorem~\ref{Pade R}~\ref{item: thm Pade R: item 1}, the pair  $\big(P_{\ell}(z), Q_{i,s,\ell}(z)\big)$ is a weight $n$ Pad\'{e}-type approximant of $\big(R_{\alpha_i,s}(z)$).
    According to Lemma~\ref{lem: values of phi_f}, we have $r_{i,s,\ell,k} = \varphi_{f}((t)_k/k!)$, so that it only remains to prove that
    \begin{align} \label{cal}
        \varphi_{f}\left(\dfrac{(t)_k}{k!}\right)= \phi_{\alpha_i,s} \left(\dfrac{(t+n)_{k-n} A_{\ell}(t)}{(k-n)!}\right)
    \end{align}
    for each $k\geq n$. Set $D=P_\ell(z)$. Then $D^* = P_\ell(t)$, and since $\fR_{i,s,\ell}=\pi(P_{\ell}R_{\alpha_i,s}) = \pi(D(R_{\alpha_i,s}))$, Proposition~\ref{key prop} yields
    \begin{align*}
        \varphi_{f}\left(\dfrac{(t)_k}{k!}\right) &=\varphi_{\pi(D(R_{\alpha_i,s}))}\left(\dfrac{(t)_k}{k!}\right) =\phi_{\alpha_i,s}\left(\dfrac{(t)_k}{k!} P_{\ell}(t)\right)
        =\phi_{\alpha_i,s}\left(\left[\dfrac{(t)_k}{k!}\right]\circ \bDelta^{n}_{-1} (A_{\ell}(t))\right).
    \end{align*}
    Using Lemma~\ref{key 2}~\ref{item: lem key 2: item 2 bis},  
    \begin{align*}
        \left[\dfrac{(t)_k}{k!}\right]\circ \bDelta^{n}_{-1}(A_{\ell}(t))\in \bDelta_{-1}\circ[A(t)](K[t])+\left[\bDelta^{n}_{1}\left(\dfrac{(t)_k}{k!}\right)\right] (A_{\ell}(t)),
    \end{align*}
    where $A(t) = \prod_{j=1}^{d}(t+\alpha_j)^{m_j+1}$. By Eqs.~\eqref{subseteq 1} and \eqref{subseteq 2}, we have $\bDelta_{-1}\circ[A(t)](K[t])\subset \ker \phi_{\alpha_i,s}$. Finally, we deduce that
    \begin{align*}
        \varphi_{f}\left(\dfrac{(t)_k}{k!}\right) = \phi_{\alpha_i,s}\left(\left[\bDelta^{n}_{1}\left(\dfrac{(t)_k}{k!}\right)\right] (A_{\ell}(t))\right),
    \end{align*}
    and we conclude by using  the identity $\bDelta^{n}_{1}\left((t)_k/k!\right)=(t+n)_{k-n}/(k-n)!$.
\end{proof}

\noindent\textbf{Absolute value of the coefficients $a_{i,s,k}$.} We end this section by estimating roughly the absolute value of the coefficients $a_{i,s,k}$ appearing in Definition~\ref{def: functions g_i,s}. In the proof of Lemma~\ref{estimate phi_is} we will estimate their $p$--adic absolute value. We start by estimating the binomial coefficients.

\begin{lemma}\label{estimate binomial}
    Let $k\geq 0$ be an integer and $\alpha\in\C$. We have
    \begin{align*}
        \frac{|(\alpha)_k|}{k!}  \leq \left\{ \begin{array}{ll}
            e^{|\alpha|-1} k^{|\alpha|-1} & \textrm{if $|\alpha| > 1$ and $k>0$}, \\
            (k+1)^{-(1-|\alpha|)} & \textrm{if $|\alpha|\leq 1$ or $k=0$}.
        \end{array}\right.
    \end{align*}
    In particular
    \begin{align}
        \label{eq: lem: estimate binomial}
        \left|\binom{\alpha}{k}\right| \leq \frac{(|\alpha|)_k}{k!} \leq  e^{|\alpha|}k^{|\alpha|}.
    \end{align}
\end{lemma}

\begin{proof}
    We may assume that $k>0$ since $\binom{\alpha}{0} = 1$. Then, using the inequality $(1+x)\leq e^x$ valid for each $x\in[-1,\infty)$, we get
    \begin{align*}
        \left|\binom{\alpha}{k}\right| \leq \frac{(|\alpha|)_k}{k!} = \prod_{j=1}^{k} \Big(1+\frac{|\alpha|-1}{j} \Big) \leq \exp\Big( \sum_{j=1}^{k} \frac{|\alpha|-1}{j} \Big).
    \end{align*}
    We obtain the expected upper bounds by combining the above with the estimates
    \begin{align*}
         \log(k+1) \leq \sum_{j=1}^{k} \frac{1}{j} \leq 1+\log k,
    \end{align*}
    and by distinguishing between the case $|\alpha|-1 \leq 0$ and the case $|\alpha|-1 > 0$.
\end{proof}

\begin{lemma}
    \label{lem: size coeff a_i,s,k}
    Let $(i,s)\in \Indset$ and $k\geq 0$ be an integer. Then
    \begin{align*}
        |a_{i,s,k}| \leq e^{|\alpha_i|}(k+1)^{s+|\alpha_i|}.
    \end{align*}
\end{lemma}

\begin{proof}
    First, suppose $i=1$. Then $s\geq 2$ and $\alpha_1=0$, and using \eqref{eq: formula coeff a_i,s,k} we obtain the crude estimate
    \begin{align}
        \label{eq proof: abs value for coeff g_i,s: eq 1}
        |a_{1,s,k}| \leq \#\Big\{(\ell_1,\dots,\ell_s)\in \Z_{\geq 0}^{s-1} \,;\, \ell_1+\cdots+\ell_{s-1} = k-s+2  \Big\} \leq (k+1)^{s-1}.
    \end{align}
    We now assume that $i\geq 2$. If $s=1$, then using \eqref{eq: lem: estimate binomial} of Lemma~\ref{estimate binomial} together with \eqref{eq: formula coeff a_i,s,k}, we obtain
    \begin{align}
        \label{eq proof: | |_p for coeff g_i,s: eq 2}
        |a_{i,1,k}| = \Big|\binom{\alpha_i}{k+1}\Big| \leq  e^{|\alpha_i|}(k+1)^{|\alpha_i|}.
    \end{align}
    If $s\geq 2$, we combine again \eqref{eq: formula coeff a_i,s,k} with \eqref{eq proof: abs value for coeff g_i,s: eq 1} and \eqref{eq: lem: estimate binomial} of Lemma~\ref{estimate binomial} to get
    \begin{align}
        \label{eq proof: | |_p for coeff g_i,s: eq 3}
        |a_{i,s,k}| = \Big|\sum_{j=0}^k\binom{\alpha_i}{j}a_{1,s,k-j}\Big| \leq e^{|\alpha_i|}(k+1)^{s+|\alpha_i|}.
    \end{align}
\end{proof}

\section{Kernel of the formal integration maps}
\label{section: non-singular matrices}

One of the crucial steps in proving our main Theorem~\ref{main 1} is to show that the Pad\'{e} approximants constructed in Section~\ref{subsection: Pade approximants for polygamma} are linearly independent. In other words, we need the matrix, whose entries are formed by the Pad\'{e} approximants, to be non-singular. This will be a consequence of the theorem below, which is the main result of this section.

\medskip

We keep the notation of Section~\ref{section: Pade approximation}. Recall that the functions $R_{\alpha, s}$ (which are related to the polygamma functions) are introduced in Definition~\ref{def: R_alpha,s}, and that the function $\Phi: f\mapsto \phi_f$ is defined in Section~\ref{subsection: Pade approximation definition}.

\begin{definition}
    \label{def: phi_alpha,s}
    Given $\alpha\in K$ and an integer $s\geq 1$, we denote by $\phi_{\alpha,s}$ the morphism $\Phi(R_{\alpha,s})$. In the case $\alpha=0$ and $s\geq 2$, we simply write $\phi_s = \Phi(R_s)$.
\end{definition}

The following property will be useful when dealing with the functions $\phi_{\alpha,s}$.

\begin{lemma}
    \label{lem: link between phi_alpha,s and phi_s}
    Let $\alpha\in K$ and $s\geq 2$ an integer. Then
    \begin{align*}
        \phi_{\alpha,s} = \phi_s\circ \Mshift_{-\alpha}.
    \end{align*}
\end{lemma}

\begin{proof}
    This is a direct consequence of Corollary~\ref{cor: useful for phi_alpha,s} (since by definition $R_{\alpha,s} = \Shift_\alpha(R_s)$).
\end{proof}

\begin{theorem}
    \label{thm: non-singular matrix for non-vanishing of the det}
    Let $d\geq 1$ be an integer and $m_1,\dots,m_d \geq 0$ be integers, with $m_1\geq 1$. Set $M = m_1+\cdots +m_d+d-1$ and
    \[
        \Indset=\{(i,s) \, ; \, 1\le i \le d \AND 1\le s \le m_i+1\} \setminus\{(1,1)\}.
    \]
    Fix non-negative integers $N_0,\dots,N_d$. For $j=1,\dots,d$, let
    \[
        \ur(j) = \big(r_0^{(j)},\dots,r_{N_j}^{(j)}\big)
    \]
    be a $(N_j+1)$--tuple of integers with
    \begin{align}
        \label{hyp exponents r_i}
        m_j+1 \geq r_0^{(j)} \geq \cdots \geq r_{N_j}^{(j)} \geq 0,
    \end{align}
    and define
    \[
        A(t) = \prod_{j=1}^{d} B_j(t+\alpha_j), \quad\textrm{where } B_j(t) = \prod_{i=0}^{N_j} (t+i)^{r_i^{(j)}}.
    \]
    Let $\alpha_1,\dots,\alpha_d \in K$ satisfying the condition $\alpha_1 = 0$ and
    \begin{align}
        \label{hyp alpha_i - alpha_j}
         \alpha_i - \alpha_j \notin \Z \textrm{ for any distinct indices $i,j\in \{1,\dots,d\}$}.
    \end{align}
    Then, the following $M\times M$ matrix is non-singular
    \begin{align*}
        \cM :=  \Big( \phi_{\alpha_i,s}(t^\ell A(t)\Big)_{\substack{(i,s)\in \Indset \\ 0\leq \ell < M }} = \left(\begin{array}{cccc}
            \phi_{\alpha_1,2}\big( A(t) \big) & \phi_{\alpha_1,2}\big( tA(t) \big) & \cdots & \phi_{\alpha_1,2}\big( t^{M-1}A(t) \big) \\
            \vdots & \vdots & \cdots & \vdots \\
            \phi_{\alpha_1,m_1+1}\big( A(t) \big) & \phi_{\alpha_1,m_1+1}\big( tA(t) \big) & \cdots & \phi_{\alpha_1,m_1+1}\big( t^{M-1}A(t) \big) \\
            \phi_{\alpha_2,1}\big( A(t) \big) & \phi_{\alpha_2,1}\big( tA(t) \big) & \cdots & \phi_{\alpha_2,1}\big( t^{M-1}A(t) \big) \\
            \vdots & \vdots & \cdots & \vdots \\
            \phi_{\alpha_2,m_2+1}\big( A(t) \big) & \phi_{\alpha_2,m_2+1}\big( tA(t) \big) & \cdots & \phi_{\alpha_2,m_2+1}\big( t^{M-1}A(t) \big) \\
            \vdots & \vdots & \cdots & \vdots \\
            \phi_{\alpha_d,1}\big( A(t) \big) & \phi_{\alpha_d,1}\big( tA(t) \big) & \cdots & \phi_{\alpha_d,1}\big( t^{M-1}A(t) \big) \\
            \vdots & \vdots & \cdots & \vdots \\
            \phi_{\alpha_d,m_d+1}\big( A(t) \big) & \phi_{\alpha_d,m_d+1}\big( tA(t) \big) & \cdots & \phi_{\alpha_d,m_d+1}\big( t^{M-1}A(t) \big) \\
        \end{array} \right).
    \end{align*}
\end{theorem}

The strategy of our proof is as follows. We easily show that any point in the kernel of $\cM$ gives rise to a polynomial $Q(t)\in K[t]$ of degree $\leq M-1$ satisfying
\[
    Q(t)A(t) \in \bigcap_{(i,s)\in \Indset} \ker \phi_{\alpha_i,s}.
\]
The core of the demonstration of the theorem consists in expressing the above subspace as the image by the operator $\bDelta_{-1}$ of a rather simple ideal of $K[t]$ (see Section~\ref{subsection: study of the kernel}). This will allow us in Section~\ref{subsection: proof of theorem det M != 0} to prove that a non-zero polynomial $Q(t)$ as above has degree at least $M$, hence a contradiction if $Q(t)\neq 0$.


\subsection{Study of the kernels}
\label{subsection: study of the kernel}

We keep the notation of Section~\ref{subsection: Pade approximation definition} for the operators $\bDelta_\alpha$, $\Shift_\alpha$ and the linear maps $\phi_{\alpha,s} = \Phi(R_{\alpha,s})$. We start by expressing the kernel of the linear maps $\phi_{\alpha, s}$ in a simple way.

\begin{lemma}
    \label{lem: intersection kernels}
    Let $s\geq 2$ be an integer and fix a shift $\alpha\in K$.
    \begin{enumerate}
      \item \label{lem: intersection kernels: item 1} We have
        \begin{align*}
            \phi_{\alpha,s}\circ \bDelta_{-1}=(-1)^s\Eval_{t=-\alpha}\circ \left(\dfrac{d}{dt}\right)^{s-1}.
        \end{align*}
      \item \label{lem: intersection kernels: item 2} The kernel of $\phi_{\alpha,s}$ is
        \begin{align}
            \label{eq lem ker phi_s: eq1}
            \ker \phi_{\alpha,s} & = \bDelta_{-1}\circ \Mshift_\alpha \left(\Span[K]{t,\dots, t^{s-2}, t^{s}, t^{s+1}, \dots} \right).
        \end{align}
      \item \label{lem: intersection kernels: item 3} For any non-negative integer $m$, we have
        \begin{align}
            \label{eq lem ker phi_s: eq2}
            \bigcap_{s=2}^{m+1} \ker \phi_{\alpha,s} = \bDelta_{-1}\left((t+\alpha)^{m+1} K[t] \right).
        \end{align}
        with the convention that the left-hand side is equal to $K[t]$ if $m=0$.
    \end{enumerate}
\end{lemma}

\begin{proof}
    Since by Lemma~\ref{lem: link between phi_alpha,s and phi_s} we have $\phi_{\alpha,s} = \phi_s\circ\Mshift_{-\alpha}$, and since $\bDelta_{-1}$ commutes with any shift operator, it suffices to prove the lemma when $\alpha = 0$. Proposition~$\ref{key prop}$ gives $\varphi_s\circ \bDelta_{-1}= \varphi_{\Delta_1(R_s)}$. Eq.~\eqref{diff eq} of Proposition~\ref{prop: difference equation R_alpha,s} and a direct computation yield
    \begin{align*}
        \varphi_{\Delta_1(R_s)} = (-1)^{s}(s-1)!\phi_{1/z^{s}}=(-1)^s\Eval_{t=0}\circ \left(\dfrac{d}{dt}\right)^{s-1}.
    \end{align*}
    Hence \ref{lem: intersection kernels: item 1}. We deduce that
    \begin{align*}
        H:=\bDelta_{-1} \left( \ker\,\Eval_{t=0}\circ \left(\dfrac{d}{dt}\right)^{s-1} \right) = \bDelta_{-1} \left(\Span[K]{t,\dots, t^{s-2}, t^{s}, t^{s+1}, \dots} \right) \subset \ker \phi_{0,s}.
    \end{align*}
    Since $H$ is an hyperplane of $K[t]$ and $\phi_s$ is a non-zero linear form, the above inclusion is an equality, and \eqref{eq lem ker phi_s: eq1} follows. Eq.~\eqref{eq lem ker phi_s: eq2} is a consequence of \eqref{eq lem ker phi_s: eq1}. If $m=0$, we simply have $\bDelta_{-1}\big( (t+\alpha)K[t]\big) = K[t]$.
\end{proof}

\begin{lemma}
    \label{lem: kernel for s=1}
    Fix a shift $\alpha\in K {\setminus \{0\}}$.
    \begin{enumerate}
        \item  \label{lem: kernel for s=1: item 1} We have
            \begin{align}
            \label{eq phi_1(Delta)}
                \phi_{\alpha,1}\circ \bDelta_{-1} = \Eval_{t=0}-\Eval_{t=-\alpha}.
            \end{align}
        \item  \label{lem: kernel for s=1: item 2} The kernel of $\phi_{\alpha,1}$ is
            \begin{align*}
                \ker \phi_{\alpha,1} = \bDelta_{-1}\Big(t(t+\alpha) K[t] \Big)
            \end{align*}
    \end{enumerate}
\end{lemma}

\begin{remark}
    Note that $\phi_{0,1} = 0$, so that $\ker \phi_{0,1} = K[t]$.
\end{remark}

\begin{proof}
    \ref{lem: kernel for s=1: item 1}. Proposition~$\ref{key prop}$ yields $\varphi_{\alpha,1}\circ \bDelta_{-1}= \varphi_{\Delta_1(R_{\alpha,1})}$. Combining this with Eq.~\eqref{diff eq} of Proposition~\ref{prop: difference equation R_alpha,s}, we deduce that $\varphi_{\Delta_1(R_{\alpha,1})} = \phi_{f}$, where
    \begin{align*}
        f(z) = \frac{\alpha}{z(z+\alpha)} = \frac{1}{z} - \frac{1}{z+\alpha}.
    \end{align*}
    We conclude by noting that $\phi_{1/z} = \Eval_{t=0}$, and $\phi_{1/(z+\alpha)} = \phi_{1/z}\circ \Mshift_{-\alpha}$ by Corollary~\ref{cor: useful for phi_alpha,s}.

    \medskip

    \ref{lem: kernel for s=1: item 2}. Eq.~\eqref{eq phi_1(Delta)} easily implies that
    \begin{align*}
        H:=\bDelta_{-1}\big(t(t+\alpha) K[t] \big) \subset \ker \phi_{\alpha,1}
    \end{align*}
    (note that this is also a consequence of \eqref{subseteq 2}). Since $H$ is a hyperplane of $K[t]$ and $\phi_{\alpha,1} \neq 0$, the above inclusion is an equality.
\end{proof}

We now establish a generalization of Eq.~\eqref{eq lem ker phi_s: eq2} of Lemma~\ref{lem: intersection kernels} by taking into account several shifts simultaneously.

\begin{lemma}
    \label{lem: description kernels}
    Let $d, m_1,\dots,m_d$ be non-negative integers with $d\geq 1$. Let $\alpha_1,\dots,\alpha_d$ be pairwise distinct elements of $K$, with $\alpha_1=0$. We have
    \begin{align*}
        \bigcap_{i=1}^d\left( \bigcap_{s=1}^{m_i+1} \ker \phi_{\alpha_i,s} \right) = \bDelta_{-1}\Big(t^{m_1+1}(t+\alpha_2)^{m_2+1} \cdots (t+\alpha_d)^{m_d+1}  K[t] \Big).
    \end{align*}
\end{lemma}

\begin{proof}
    By Lemma~\ref{lem: intersection kernels}~\ref{lem: intersection kernels: item 3} and Lemma~\ref{lem: kernel for s=1}~\ref{lem: kernel for s=1: item 2}, we have to prove that the two following subspaces
    \begin{align*}
        V &:=\bigcap_{i=1}^d\bDelta_{-1} \left((t+\alpha_i)^{m_i+1} K[t] \right) \bigcap_{i=2}^d \bDelta_{-1}\Big(t(t+\alpha_i) K[t] \Big), \\
        W &:=\bDelta_{-1}\Big(\prod_{i=1}^{d}(t+\alpha_i)^{m_i+1}  K[t] \Big),
    \end{align*}
    are equal. The inclusion $W\subset V$ is trivial. Now, fix $P(t)\in V$, and let us prove that $P(t)\in W$. Since $P(t)\in \bDelta_{-1} \left(t^{m_1+1} K[t] \right)$, there exists $Q(t)\in K[t]$ such that
    \[
        P(t) = \bDelta_{-1}\Big( t^{m_1+1}Q(t) \Big).
    \]
    To conclude, it suffices to prove that $\prod_{i=2}^{d}(t+\alpha_i)^{m_i+1}$ divides $Q(t)$. Given $i\in\{2,\dots,d\}$, there exist $R_i(t), S_i(t)\in K[t]$ such that
    \begin{align*}
        P(t) = \bDelta_{-1}\Big( (t+\alpha_i)^{m_i+1}R_i(t) \Big) = \bDelta_{-1}\Big( t(t+\alpha_i)S_i(t) \Big).
    \end{align*}
    Since $\ker \bDelta_{-1} = K \subset K[t]$, we deduce the existence of $a_i, b_i \in K$ satisfying
    \begin{align*}
        t(t+\alpha_i)S_i(t) = (t+\alpha_i)^{m_i+1}R_i(t) + a_i = t^{m_1+1}Q(t) + b_i.
    \end{align*}
    Evaluating at $t=-\alpha_i$ and $t=0$, we find $a_i = b_i = 0$, so that
    \[
        t^{m_1+1}Q(t) = (t+\alpha_i)^{m_i+1}R_i(t).
    \]
    As $\alpha_i\neq 0$, it follows that $(t+\alpha_i)^{m_i+1}$ divides $Q(t)$. Since the $\alpha_i$'s are all distinct, we deduce that the polynomial $\prod_{i=2}^{d}(t+\alpha_i)^{m_i+1}$ divides $Q(t)$. Hence $P(t)\in W$.
\end{proof}


\subsection{Proof of Theorem~\ref{thm: non-singular matrix for non-vanishing of the det}}
\label{subsection: proof of theorem det M != 0}

\begin{lemma}
    \label{lem: condition Q(t)P(t) in the kernel intersection}
    Fix two non-negative integers $N,m$. Let $\ur=(r_0,\dots,r_N)$ be a $(N+1)$--tuple of integers satisfying $m+1\geq r_0\geq \cdots \geq r_N \geq 0$, and define
    \begin{align*}
        B_{\ur}(t) = \prod_{i=0}^{N} (t+i)^{r_i}.
    \end{align*}
    \begin{enumerate}
        \item \label{lem: condition Q(t)P(t) in the kernel intersection: item 1} Let $Q(t), R(t) \in K[t]$ be two polynomials satisfying $Q(t)B_{\ur}(t) = \bDelta_{-1}\big(t^{m+1}R(t)\big)$. Then the polynomial $B_{\ur}(t+1)$ divides $R(t)$.
        \item \label{lem: condition Q(t)P(t) in the kernel intersection: item 2} We have
            \begin{align*}
                B_{\ur}(t)K[t] \cap \bDelta_{-1}\Big(t^{m+1}K[t]\Big) = \bDelta_{-1}\Big(t^{m+1}B_{\ur}(t+1)K[t]\Big).
            \end{align*}
    \end{enumerate}
\end{lemma}

\begin{proof}
    For simplicity, write $A(t) = B_{\ur}(t)$. We may assume that $Q(t)$ and $R(t)$ are non-zero, otherwise \ref{lem: condition Q(t)P(t) in the kernel intersection: item 1} is automatic. By induction on $k$, let us prove that $(t+k+1)^{r_k}$ is a factor of $R(t)$ for $i=0,\dots,N$. By hypothesis, we have
    \begin{align}
        \label{eq inter 0}
         Q(t)A(t) = \bDelta_{-1}(t^{m+1}R(t)) = (t-1)^{m+1}R(t-1) -  t^{m+1}R(t).
    \end{align}
    Since $t^{r_0}$ divides $A(t)$ and $t^{m+1}R(t)$ (since $r_0 \leq m+1$), necessarily $t^{r_0}$ also divides $(t-1)^{m+1}R(t-1)$. It follows that $t^{r_0}$ divides $R(t-1)$, or equivalently $(t+1)^{r_0}$ divides $R(t)$. Suppose now that $(t+k+1)^{r_k}$ divides $R(t)$ for some integer $k$ with $0 \leq k< N$. Then, $(t+k+1)^{r_{k+1}}$ divides $A(t)$ as well as $R(t)$, since $r_{k+1}\leq r_k$. Eq.~\eqref{eq inter 0} ensures that $(t+k+1)^{r_{k+1}}$ divides $(t-1)^{m+1}R(t-1)$. We deduce that $(t+k+2)^{r_{k+1}}$ divides $R(t)$, which concludes our induction step. Therefore, the polynomial
    \begin{align*}
        \prod_{k=0}^{N} (t+k+1)^{r_k} = A(t+1)
    \end{align*}
    divides $R(t)$, hence \ref{lem: condition Q(t)P(t) in the kernel intersection: item 1}. It follows that
    \begin{equation}
        \label{eq proof inter 1}
        B_{\ur}(t)K[t] \cap \bDelta_{-1}\Big(t^{m+1}K[t]\Big) \subset \bDelta_{-1}\Big(t^{m+1}B_{\ur}(t+1)K[t]\Big).
    \end{equation}
    Conversely, the hypothesis $m+1\geq r_0\geq \cdots \geq r_N \geq 0$ implies that $B_{\ur}(t)$ divides $t^{m+1}B_{\ur}(t+1)$. Thus $B_{\ur}(t)$ also divides $\bDelta_{-1}\big(t^{m+1}B_{\ur}(t+1)\big)$. We easily deduce that \eqref{eq proof inter 1} is an equality, hence \ref{lem: condition Q(t)P(t) in the kernel intersection: item 2}.
\end{proof}

We now establish a generalization of Lemma~\ref{lem: condition Q(t)P(t) in the kernel intersection} which will be needed in order to prove Theorem~\ref{thm: non-singular matrix for non-vanishing of the det}.

\begin{lemma}
    \label{lem: condition Q(t)P(t) in the kernel intersection general}
    We keep the notation of Theorem~\ref{thm: non-singular matrix for non-vanishing of the det} and put
    \[
        B(t) = \prod_{i=0}^{d}(t+\alpha_i)^{m_i+1}.
    \]
    \begin{enumerate}
        \item \label{lem: condition Q(t)P(t) in the kernel intersection general: item 1} Let $Q(t), R(t) \in K[t]$ be two polynomials satisfying $Q(t)A(t) = \bDelta_{-1}\big(B(t)R(t)\big)$. Then the polynomial $A(t+1)$ divides $R(t)$.
        \item \label{lem: condition Q(t)P(t) in the kernel intersection general: item 2} We have
            \begin{align*}
                A(t)K[t] \cap \bDelta_{-1}\Big(B(t)K[t]\Big) = \bDelta_{-1}\Big(B(t)A(t+1)K[t]\Big).
            \end{align*}
    \end{enumerate}
\end{lemma}

\begin{proof}
    For $j=1,\dots,d$, write
    \begin{align*}
        Q(t)A(t) = Q_j(t+\alpha_j) B_{j}(t+\alpha_j) \AND B(t)R(t) = (t+\alpha_j)^{m_j+1} R_j(t+\alpha_j),
    \end{align*}
    with $Q_j(t), R_j(t)\in K[t]$. By hypothesis, we have $Q_j(t) B_{j}(t) = \bDelta_{-1}\big( t^{m_j+1} R_j(t)\big)$. Lemma~\ref{lem: condition Q(t)P(t) in the kernel intersection}~\ref{lem: condition Q(t)P(t) in the kernel intersection: item 1} implies that $B_{j}(t+1)$ divides $R_j(t)$. Equivalently, $B_{j}(t+\alpha_j+1)$ divides
    \[
        R_j(t+\alpha_j) = R(t)\prod_{\substack{i=1 \\ i\neq j}}^{d}(t+\alpha_i)^{m_i+1}.
    \]
    Our hypothesis \eqref{hyp alpha_i - alpha_j} on the $\alpha_i$ ensures that $B_{j}(t+\alpha_j+1)$ and $\displaystyle\prod_{i=1,i\neq j}^{d}(t+\alpha_i)^{m_i+1}$ are coprime polynomials. So $B_{j}(t+\alpha_j + 1)$ divides $R(t)$. Furthermore, \eqref{hyp alpha_i - alpha_j} also implies that $B_{1}(t+\alpha_1 + 1),\dots, B_{d}(t+\alpha_d + 1)$ are coprime. We conclude that the product $B_{1}(t+\alpha_1 + 1)\cdots B_d(t+\alpha_d+1) = A(t+1)$ divides $R(t)$. Hence \ref{lem: condition Q(t)P(t) in the kernel intersection general: item 1}. We also deduce that
    \begin{align*}
        A(t)K[t] \cap \bDelta_{-1}\Big(B(t)K[t]\Big) \subset \bDelta_{-1}\Big(B(t)A(t+1)K[t]\Big).
    \end{align*}
    Conversely, the hypothesis \eqref{hyp exponents r_i} implies that $A(t)$ divides $B(t)A(t+1)$. Thus $A(t)$ also divides $\bDelta_{-1}\big(B(t)(t+1)\big)$ and the above inclusion is an equality.
\end{proof}

\begin{proof}[\textbf{Proof of Theorem~\ref{thm: non-singular matrix for non-vanishing of the det}}]
    Given $a_0,\dots,a_{M-1} \in K$, we have
    \begin{align*}
        \cM C =
        \left(\begin{array}{c}
            \phi_{1,2}\Big(Q(t)A(t)\Big) \\
            \vdots \\
            \phi_{d,m_d+1}\Big(Q(t)A(t)\Big)
        \end{array}\right), \qquad \textrm{where } C = \left(\begin{array}{c}
            a_0 \\
            \vdots \\
            a_{M-1}
        \end{array}\right) \AND Q(t) = \sum_{k=0}^{M-1} a_kt^k.
    \end{align*}
    Suppose that $C\in\ker \cM$. Then, writing $B(t) = \prod_{i=1}^d (t+\alpha_i)^{m_i+1}$, we have
    \[
        Q(t) A(t)\in  \bigcap_{(i,s)\in \Indset} \ker \phi_{\alpha_i,s} = \bDelta_{-1}\Big(B(t)K[t]\Big),
    \]
    the last equality coming from Lemma~\ref{lem: description kernels}. Using Lemma~\ref{lem: condition Q(t)P(t) in the kernel intersection general}~\ref{lem: condition Q(t)P(t) in the kernel intersection general: item 2}, we deduce that there exists $R(t)\in K[t]$ such that
    \[
         Q(t) A(t) = \bDelta_{-1}\big(B(t)A(t+1) R(t)\big).
    \]
    We find
    \begin{align*}
        M-1 + \deg A(t) \geq \deg Q(t) + \deg A(t) & = \deg \bDelta_{-1}\Big(B(t)A(t+1) R(t)\Big) \\
        & = \deg B(t) + \deg A(t) + \deg R(t) -1.
    \end{align*}
    Since $\deg B(t) = M+1$, it follows that $\deg R(t) \leq -1$, hence $R(t)=0$. As a consequence $Q(t) =0$, or equivalently, $C=0$. Thus $\ker \cM = \{0\}$.
\end{proof}

\subsection{Linear independence of the Pad\'{e} approximants}

We keep the notation of Section~\ref{subsection: Pade approximants for polygamma}, with $K=\Q$. Let $\ell,n$ be non-negative integers with $0\le \ell\le M$. For each $(i,s)\in \Indset$, the polynomials $P_{n,\ell}(z)$, $Q_{n,i,s,\ell}(z)$, and the Pad\'{e} approximation
\begin{align*}
    \fR_{n,i,s,\ell}(z) = P_{n,\ell}(z)R_{\alpha_i,s}(z)-Q_{n,i,s,\ell}(z)
\end{align*}
of $R_{\alpha_i,s}(z)$ are defined in Theorem~\ref{Pade R}. The main result of this subsection in Theorem~\ref{thm: non-zero det Pade approx} below, which ensures the crucial non-vanishing property of certain determinants associated with the above Pad\'{e} approximants. It uses the following notation. For $\ell=0,\dots,M$, define the column vectors
\begin{align*}
    \bp_{n,\ell}(z) &={}^t\!\Big(P_{n,\ell}(z),Q_{n,1,2,\ell}(z),\ldots,Q_{n,1,m_1+1,\ell}(z),\ldots, Q_{n,d,1,\ell}(z),\cdots, Q_{n,d,m_d+1,\ell}(z)\Big) \\
    & = {}^t\!\Big(P_{n,\ell}(z), Q_{n,i,s,\ell}(z)\Big)_{(i,s)\in\Indset},
\end{align*}
and form the $(M+1)\times (M+1)$ matrix
\begin{align*}
    \Matp(z) = \big(\bp_{n,0}(z) ,\ldots , \bp_{n,M}(z) \big).
\end{align*}

\begin{theorem}
    \label{thm: non-zero det Pade approx}
    There exists $a\in\Q^\times$ such that $\det \Matp(z) = a$. In particular, for any $x\in\Q$, the Pad\'{e} approximants $\bp_{n,0}(x),\dots,\bp_{n,M}(x)$ are linearly independent over $K$.
\end{theorem}

Theorem~\ref{thm: non-zero det Pade approx} is a direct consequence of Lemma~\ref{sufficient condition} and Proposition~\ref{prop: non-vanishing determinant} below. Our strategy is the following. By definition $\Matp(z)$ is a polynomial. We show in Lemma~\ref{sufficient condition}, which is essentially an application of \cite[Lemma 4.2~(ii)]{DHK3}, that this polynomial is a constant, and we reduce the problem to showing that another determinant $\Theta_n$ is non-zero. This last property, established in Proposition~\ref{prop: non-vanishing determinant}, will be a consequence of Theorem~\ref{thm: non-singular matrix for non-vanishing of the det}. In order to prove the above results, let us introduce more notation. Define
\begin{align*}
    \Det_n(z)=\det \Matp(z) \AND
    \Theta_{n} = \det (\phi_{\alpha_i,s}(A_{n,\ell}(t)))_{\substack{0\le \ell \le M-1 \\ (i, s)\in \Indset}}\in \Q,
\end{align*}
where
\begin{align*}
    A_{n,\ell}(t)=(-1)^{\ell}\dfrac{(t)_{\ell}}{\ell !} \prod_{i=1}^d\left((-1)^n\dfrac{(t+\alpha_i)_n}{n!}\right)^{m_i+1}
\end{align*}
is defined in Theorem~\ref{Pade R}. For $\ell=0,\dots,M$, denote by $\br_{n,\ell}(z)$ the column vector
\begin{align*}
    \br_{n,\ell}(z) &={}^t\!\Big(P_{n,\ell}(z),\fR_{n,1,2,\ell}(z),\ldots,\fR_{n,1,m_1+1,\ell}(z),\ldots, \fR_{n,d,1,\ell}(z),\cdots, \fR_{n,d,m_d+1,\ell}(z)\Big) \\
    & = {}^t\!\Big(P_{n,\ell}(z), \fR_{n,i,s,\ell}(z)\Big)_{(i,s)\in\Indset},
\end{align*}
and form the $(M+1)\times (M+1)$ matrix
\begin{align*}
    \Matr(z) = \big(\br_{n,0}(z) ,\ldots , \br_{n,M}(z) \big).
\end{align*}

Then, by definition of the Pad\'{e} approximations $\fR_{i,s,\ell}(z)$, we have
\begin{align}
    \label{eq: matrix relation Pade approximations}
    U(z)\Matp(z) = \Matr(z), \quad \textrm{where }    U(z)= \begin{pmatrix}
        1 & 0 & \cdots& 0 & \cdots & \cdots & 0\\
        R_{n,1,2}(z)  & -1 & 0 & \vdots & \cdots & \cdots & 0\\
        \vdots & \cdots & \ddots & \vdots & \cdots & \cdots & \vdots\\
        R_{n,1,m_1+1}(z) & 0 & \cdots & -1& \cdots & \cdots & \vdots\\
        \vdots & \cdots & \vdots & \vdots & \cdots & \cdots & \vdots\\
        R_{n,d,1}(z)  & 0 & 0 & \vdots & -1 & \vdots & 0\\
        \vdots & \cdots & \vdots & \vdots & \cdots & \ddots & \vdots\\
        R_{n,d,m_d+1}(z) & 0 & \cdots & 0 & \ddots & \cdots & -1
    \end{pmatrix}.
\end{align}

\begin{lemma}  \label{sufficient condition}
    There exists  $c\in \Q^\times$ such that $\Det_n(z)=c\cdot \Theta_n$.
\end{lemma}

\begin{proof}
    The entries of the first  row of $\det \Matr(z)$ are the polynomials $P_{n,0}(z),P_{n,1}(z),\dots, P_{n,M}(z)$, which have degrees $nM,nM+1,\dots,nM+M$ respectively. Thus
    \begin{align}
        \label{eq proof: lemme det Matp and Theta: eq 1}
        P_{n,0}(z),\dots, P_{n,M-1}(z) \in \Q[z]_{\leq (n+1)M-1} \AND P_{n,M}(z) \in \tc z^{(n+1)M} + \Q[z]_{\leq (n+1)M-1},
    \end{align}
    where $\tc \in \Q^\times$ denotes the leading coefficient of the polynomial $P_{n,M}(z)$. On the other hand, for each $(i,s)\in \Indset$ and each $\ell$ with $0\leq \ell \leq M$, Theorem~\ref{Pade R}~\ref{item: thm Pade R: item 3} ensures that
    \begin{align}
        \label{eq proof: lemme det Matp and Theta: eq 2}
        \fR_{n,i,s,\ell} \in \frac{n!\phi_{\alpha_i,s}(A_{n,\ell}(t))}{z^{n+1}} + \frac{1}{z^{n+2}}\Q[[1/z]].
    \end{align}
    For the sake of completion, we now recall the main arguments of \cite[Lemma 4.2~(ii)]{DHK3}. Expanding $\det \Matr(z)$ along its first row and using \eqref{eq proof: lemme det Matp and Theta: eq 1} together with \eqref{eq proof: lemme det Matp and Theta: eq 2}, we find
    \begin{align*}
        \det \Matr(z) \in (-1)^M (n!)^M \tc \det\big(\phi_{\alpha_i,s}(A_{n,\ell}(t))\big)_{\substack{0\leq \ell\leq M-1\\ (i,s)\in \Indset}} + \frac{1}{z}\Q[[1/z]].
    \end{align*}
    Finally, according to \eqref{eq: matrix relation Pade approximations}, $\det \Matr(z) = (-1)^M\det \Matp(z) \in \Q[z]$. Combined with the above, we conclude that $\det \Matr(z) = (-1)^M (n!)^M \tc \, \Theta_n$.
\end{proof}

\begin{proposition}
    \label{prop: non-vanishing determinant}
    Let $\un = (n_{1,0},\dots, n_{1,m_1},\dots,n_{d,0},\dots, n_{d,m_d})$ be a $(M+1)$--tuple of non-negative integers. Then
    \begin{align*}
        \Theta_{\un} := \det\left(\phi_{\alpha_i,s}\left((-1)^{\ell}\dfrac{(t)_{\ell}}{\ell !} \prod_{k=1}^d\prod_{j=0}^{m_{k}}(-1)^{n_{k,j}}\dfrac{(t+\alpha_{k})_{n_{k,j}}}{n_{k,j}!}\right)\right)_{\substack{0\le \ell \le M-1 \\ (i,s)\in \Indset}} \neq 0.
    \end{align*}
    In particular $\Theta_n = \Theta_{(n,\ldots,n)} \neq 0$.
\end{proposition}


\begin{proof}
    Fix non-negative integers $m, n_0,\dots,n_m$, and set $N = \max\{n_0,\dots,n_m\}$. For $k=0,\dots,N$, define $r_k$ as the number of indices $i\in\{0,\dots,m\}$ such that $n_i > k$, and set $\ur = (r_0,\dots,r_N)$. Then $m+1\geq r_0 \geq \cdots \geq r_N = 0$ and
    \[
        \prod_{i=0}^{m} (t)_{n_i} = B_{\ur}(t) = \prod_{i=0}^{N} (t+i)^{r_i},
    \]
    where $B_{\ur}(t)$ is as in Lemma~\ref{lem: condition Q(t)P(t) in the kernel intersection}. We conclude that the polynomial
    \begin{align*}
        \prod_{k=1}^d\prod_{j=0}^{m_{k}}(t+\alpha_{k})_{n_{k,j}}
    \end{align*}
    has the same form as the polynomial $A(t)$ in the statement Theorem~\ref{thm: non-singular matrix for non-vanishing of the det}, hence
    \begin{align*}
        0 \neq {\rm{det}}\left(\phi_{\alpha_i,s}\Big(t^\ell A(t)\right)\Big)_{\substack{0\le \ell \le M-1 \\ (i,s)\in \Indset}} = \pm \Theta_{\un}\cdot\left(\prod_{\ell=0}^{M-1} \ell! \right)\cdot\left(\prod_{j=0}^{m_{k}}n_{k,j}!\right)^M.
    \end{align*}
    We conclude that $\Theta_\un\neq 0$.
\end{proof}

Although we will not need it in the following, it seems that in the simpler case $d=1$, we can express the determinant $\Theta_{\un}$ in a simple way.

\begin{conj} \label{main}
    Let $m,n_0,\dots,n_m$ non-negative integers with $m\geq 1$.
    The following identity holds
    \begin{align*}
        {\rm{det}}\left(\varphi_{s}\left(\dfrac{(t)_{\ell}}{\ell !} \prod_{j=0}^{m}\dfrac{(t)_{n_{j}}}{n_{j}!}\right)\right)_{\substack{0\le \ell \le m-1 \\ 2 \leq s \leq m+1 }}
        = \dfrac{(-1)^{m-1} m!  \prod_{i=0}^{m}n_i !}{(n_0+\cdots+n_{m}+m)!}.
    \end{align*}
\end{conj}

\section{Estimates}
\label{section: estimates}

We keep the notation of Section~\ref{subsection: Pade approximants for polygamma}, with $K=\Q$. So
\begin{align*}
    \Indset=\{(i,s) \, ; \, 1\le i \le d \AND 1 \le s \le m_i+1\} \setminus\{(1,1)\},
\end{align*}
and for each $\alpha\in\Q$ and any positive integer $s$, we have
\begin{align*}
    \phi_{\alpha,s} = \phi_{R_{\alpha,s}},
\end{align*}
(see Definitions~\ref{def: R_alpha,s} and~\ref{def: phi_alpha,s}). Let $\ell,n$ be non-negative integers with $0\le \ell\le M$. For each $(i,s)\in \Indset$, the polynomials $P_{n,\ell}(z)$, $Q_{n,i,s,\ell}(z)$, and the Pad\'{e} approximation
\begin{align*}
    \fR_{n,i,s,\ell}(z) = P_{n,\ell}(z)R_{\alpha_i,s}(z)-Q_{n,i,s,\ell}(z)
\end{align*}
of $R_{\alpha_i,s}(z)$ are defined in Theorem~\ref{Pade R}. Recall that the polynomial $P_{n,\ell}(z)$ has degree $nM+\ell$.

\subsection{Absolute value of the Pad\'{e} approximants}

We keep the notation introduced at the beginning of Section~\ref{section: estimates}. We describe the asymptotic behavior, as $n$ goes to infinity, of the polynomials $P_{n,\ell}(z)$ and $Q_{n,i,s,\ell}(z)$ evaluated at a fixed rational number $x$. In Section~\ref{section: poincare-perron rec} we will explain how we can improve  the rough estimate \eqref{eq prop: estimate Pade approximants for Q} below thanks to Perron's second theorem.

\begin{proposition}\label{est approximants}
    Let $x\in \Q$. Then, for any $(i,s)\in\Indset$ and any integer $\ell$ with $0\le \ell \le M$, we have
    \begin{align}
        \label{eq prop: estimate Pade approximants for P}
        \limsup_{n\to \infty} |P_{n,\ell}(x)|^{1/n} & \le 1, \\
        \label{eq prop: estimate Pade approximants for Q}
        \limsup_{n\to \infty} |Q_{n,i,s,\ell}(x)|^{1/n} & \le \rho(M):= \bigg(2\frac{(M+1)^{M+1}}{M^{M}}\bigg)^{M+1}. 
    \end{align}
\end{proposition}

\begin{proof}
    Fix $(i,s)\in\Indset$ and integer $\ell$ with $0\leq \ell \leq M$. Let $n$ be a positive integer. We first prove~\eqref{eq prop: estimate Pade approximants for P}. Theorem~\ref{Pade R}~\ref{item: thm Pade R: item 2} together with the identity $(y-k)_n = (y-k)_k (y)_{n-k}$ valid for each $k\leq n$ yields
    \begin{align*}
        P_{n,\ell}(z)  
         & =\pm\sum_{k=0}^n\binom{n}{k}(-1)^{n-k}\dfrac{(z-k)_\ell}{\ell!}\prod_{j=1}^d \left(\dfrac{(z+\alpha_j-k)_k(z+\alpha_j)_{n-k}}{n!}\right)^{m_j+1} \\
         & = \pm\sum_{k=0}^n\binom{n}{k}(-1)^{n-k} \dfrac{(z-k)_\ell}{\ell!} \prod_{j=1}^d \left(\dfrac{(z+\alpha_j-k)_k}{k!}\dfrac{(z+\alpha_j)_{n-k}}{(n-k)!}\binom{n}{k}^{-1}\right)^{m_j+1}.
    \end{align*}
    Noticing that $|(x+\alpha_j-k)_k| \leq (1+|x+\alpha_j|)_k$ and using Lemma~\ref{estimate binomial}, we find for $k=0,\dots,n$ the upper bounds
    \begin{align*}
        \max\left\{ \dfrac{|(x+\alpha_j-k)_k|}{k!}, \dfrac{|(x+\alpha_j)_{n-k}|}{(n-k)!}  \right\}  \leq  (en)^{|x+\alpha_j|}.
    \end{align*}
    We deduce the rough estimate
    \begin{align*}
        |P_{n,\ell}(x)| \leq \frac{(|x|+n)^\ell}{\ell!}(en)^\beta = e^{o(n)},\quad\textrm{where } \beta = 2\sum_{j=1}^{d} |x+\alpha_j|(m_j+1),
    \end{align*}
    hence \eqref{eq prop: estimate Pade approximants for P}. We now prove \eqref{eq prop: estimate Pade approximants for Q}. By Theorem~\ref{Pade R}~\ref{item: thm Pade R: item 2}, we have
    \begin{align*}
        Q_{i,s,\ell}(z) &=\sum_{j=1}^{Mn+{\ell}} p_{n,j,\ell}\cdot b_{j,i,s}(z),
    \end{align*}
    where
    \begin{align*}
        b_{j,i,s}(z):=\sum_{k=0}^{j-1}(-1)^{k+1}a_{i,s,k} \dfrac{(z+k+1)\cdots (z+j-1)}{(k+1)\cdots (j-1)j}
    \end{align*}
    (the coefficients $a_{i,s,k}$ are as in Definition~\ref{def: functions g_i,s}), and
    \begin{align}
        \label{eq proof: estimate Q: eq1}
          p_{n,j,\ell} = \sum_{k=n}^{j+n}\binom{j+n}{k}(-1)^{n-k}\binom{k}{\ell} \prod_{r=1}^d{\binom{k-\alpha_r}{n}}^{m_r+1}
    \end{align}
    as in Theorem~\ref{Pade R}~\ref{item: thm Pade R: item 2}.
    Evaluating at $z=x$ and using Lemma~\ref{estimate binomial}, we find for $j=1,\dots,Mn+\ell$
    \begin{align*}
        \dfrac{|(x+k+1)\cdots (x+j-1)|}{(k+1)\cdots (j-1)} \leq \dfrac{(|x|+k+1)\cdots (|x|+j-1)|}{(k+1)\cdots (j-1)} \leq \frac{(|x|)_j}{j!} \leq e^{|x|} j^{|x|}.
    \end{align*}
    Combining the above with Lemma~\ref{lem: size coeff a_i,s,k}, we get the estimate
    \begin{align}
        \label{eq proof: estimate Q: eq0}
        \max_{1\leq j \leq Mn+\ell}|b_{j,i,s}(x)| \leq e^{|x|+|\alpha_i|} (Mn+\ell)^{1+s+|x|+|\alpha_i|} = e^{o(n)}
    \end{align}
    as $n$ tends to infinity. We now estimate the coefficients $p_{n,j,\ell}$. For any integers $k,j$ with $1\leq j\leq Mn+\ell$ and $n\leq k \leq j+n$, we have
    \begin{align}
        \label{eq proof: estimate Q: eq2}
        \binom{j+n}{k}\binom{k}{\ell} \leq 2^{j+n}k^\ell \leq 2^{(M+1)n+\ell}((M+1)n+\ell)^\ell = e^{o(n)}2^{(M+1)n}
    \end{align}
    as $n$ tends to infinity. Put $\alpha=\lceil \max_{1\leq r \leq d} |\alpha_r| \rceil$. Then
    \begin{align}
        \label{eq proof: estimate Q: eq3}
        \left|\prod_{r=1}^d{\binom{k-\alpha_r}{n}}^{m_r+1} \right| \leq \prod_{r=1}^d{\binom{k+\alpha}{n}}^{m_r+1} = \binom{k+\alpha}{n}^{M+1} \leq \binom{(M+1)n+\ell+\alpha}{n}^{M+1}.
    \end{align}
    Finally, using Stirling's formula, we obtain, as $n$ tends to infinity,
    \begin{align*}
        \binom{(M+1)n+\ell+\alpha}{n} = \frac{((M+1)n+\ell+\alpha)!}{n!(Mn+\ell+\alpha)!} & = e^{o(n)} \frac{((M+1)n+\ell+\alpha)^{(M+1)n+\ell+\alpha}}{n^n(Mn+\ell+\alpha)^{Mn+\ell+\alpha}} \\
        & = e^{o(n)} \frac{((M+1)n)^{(M+1)n}}{n^n(Mn)^{Mn}} \\
        & = e^{o(n)} \left(\frac{(M+1)^{M+1}}{M^{M}}\right)^n.
    \end{align*}
    Combining the above with \eqref{eq proof: estimate Q: eq1}--\eqref{eq proof: estimate Q: eq3}, we deduce that
    \begin{align*}
        |p_{n,j,\ell}| \leq e^{o(n)} \left(2\frac{(M+1)^{M+1}}{M^{M}}\right)^{(M+1)n}
    \end{align*}
    as $n$ tends to infinity, uniformly on $j\leq Mn+\ell$. Together with \eqref{eq proof: estimate Q: eq0}, this yields $|Q_{n,i,s,\ell}(x)| \leq e^{o(n)}\rho(M)^n$.
\end{proof}

\subsection{$p$--adic absolute value of the Pad\'{e} approximations}

Let $n\geq 0$ be an integer. We keep the notation introduced at the beginning of Section~\ref{section: estimates}. Recall that the functions $\mu$ and $\den$ are defined in \eqref{eq def: den(alpha) and mu(alpha)}. This section is devoted to estimating the $p$--adic absolute values of the Pad\'{e} approximations $\fR_{n,i,s,\ell}(z) = P_{n,\ell}(z)R_{\alpha_i,s}(z)-Q_{n,i,s,\ell}(z)$ evaluated at some rational point $x$.

\begin{proposition} \label{remainder est}
    Let $p$ be a prime number, $(i,s)\in\Indset$ and $x\in \Q$. Assume that
    \begin{align*}
        |x|_p \cdot|\mu(\alpha_i)|_p  > 1.
    \end{align*}
    Then, for $\ell=0,\dots,M$, the series $\fR_{n,i,s,\ell}(x)$ converges to an element of $\Q_p$, and
    \begin{align*}
        \limsup_{n\to\infty}\big|\fR_{n,i,s,\ell}(x)\big|_p^{1/n} \le p^{-1/(p-1)}  \Big|x\mu(\alpha_i)^{M+1}\prod_{k=2}^d\mu(\alpha_{k})^{m_{k}+1}\Big|^{-1}_p.
    \end{align*}
\end{proposition}

\begin{remark}
    \label{remark: estimate remainder}
    Since $\mu(\alpha_i)$ divides $\mu(\balpha)$ for each $i$, we have $|\mu(\alpha_i)|_p \geq |\mu(\balpha)|_p$. So, the statement of Theorem~\ref{remainder est} still holds if we replace $\mu(\alpha_i)$ with $\mu(\balpha)$.
\end{remark}

The proof of Proposition~\ref{remainder est} uses the following notation. Given a vector $\balpha=(\alpha_1,\ldots,\alpha_m)\in \Q^m$ and $N\in \N$, we put
\begin{align}
    \label{eq: def mu_n}
    \mu_{N}(\balpha)=\den(\balpha)^N\prod_{\substack{q:\text{prime} \\ q|\den(\balpha)}}q^{\lfloor N/(q-1)\rfloor},
\end{align}
(recall that $\mu(\balpha)$ and $\den(\balpha)$ are defined in \eqref{eq def: den(alpha) and mu(alpha)}). It follows easily from the definition that
\begin{align}
    \label{eq: link mu_n and mu^n}
    |\mu_n(\balpha)|_p \geq |\mu(\balpha)|_p^n.
\end{align}
Note that for any integer $\ell$ and any rational number $\alpha$, we have $\den(\alpha+\ell) = \den(\alpha)$. Consequently, we also have $\mu_n(\alpha+\ell) = \mu_n(\alpha)$. We will frequently use the following classical lemma, which is a direct consequence of \cite[Lemma~2.2]{B}, to control the denominator of $(\alpha)_n/n!$.

\begin{lemma}\label{well-known}
    Let $n$ be a non-negative integer and $\alpha\in\Q$. Then, for $k=0,\dots,n$, we have
    \begin{align*}
        \mu_n(\alpha) \dfrac{(\alpha)_k}{k!}\in \Z \AND \mu_n(\alpha)\binom{\alpha}{k} \in \Z.
    \end{align*}
\end{lemma}

\begin{lemma} \label{estimate phi_is}
    Let $p$ be a prime number and $P(t) = \sum_{k=0}^np_k \cdot (t)_k/k!\in \Q[t]$ be a rational polynomial of degree $n \geq 0$. Suppose that $p_0,\dots,p_n\in \Z$. Then, for any $(i,s)\in \Indset$, we have
    \begin{align*}
        \left|\phi_{\alpha_i,s}(P(t)) \right|_p\le (n+1)^{s-1}\left| \mu_{n+1}(\alpha_i)\right|^{-1}_{p}.
    \end{align*}
\end{lemma}

\begin{proof}
    Fix $(i,s)\in\Indset$ and write $g_{i,s}(z) := \sum_{k=0}^{\infty} a_{i,s,k}z^k$ as in Definition~\ref{def: functions g_i,s}. According to Lemma~\ref{lem: values phi_i,s with g_i,s}, we have
    \begin{align}
        \label{eq proof: | |_p for coeff g_i,s: eq 00}
        |\phi_{\alpha_i,s}(P(t))|_p = \Big|\sum_{k=0}^n p_k (-1)^{k+1} a_{i,s,k}\Big|_p \leq \max_{0\leq k \leq n} |a_{i,s,k}|_p.
    \end{align}
    We now estimate the $p$--adic norm of the coefficients $a_{i,s,k}$ thanks to the explicit formulas \eqref{eq: formula coeff a_i,s,k}.

    \medskip

    \noindent\textbf{Case 1.} Suppose $i=1$. Then $s\geq 2$ and $\alpha_1=0$, and using \eqref{eq: formula coeff a_i,s,k} we obtain $a_{1,s,0} = \cdots = a_{1,s,s-2} = 0$ and the crude estimate
    \begin{align}
        \label{eq proof: | |_p for coeff g_i,s: eq 1}
        |a_{1,s,k}|_p \leq (k+1)^{s-1} \quad (k\geq s-1).
    \end{align}
    Note that we could easily get the better upper bound $((k+1)/(s-1))^{s-1}$ when $s\geq 2$ by using the inequality of arithmetic and geometric means, however it would not make a difference for our applications.

    \medskip

    \noindent\textbf{Case 2.}  Suppose $i\geq 2$. If $s=1$, then using Lemma~\ref{well-known} together with \eqref{eq: formula coeff a_i,s,k}, we obtain
    \begin{align}
        \label{eq proof: | |_p for coeff g_i,s: eq 2}
        |a_{i,1,k}|_p = \Big|\binom{\alpha_i}{k+1}\Big|_p \leq |\mu_{k+1}(\alpha_i)|^{-1}_p.
    \end{align}
    If $s\geq 2$, we combine again \eqref{eq: formula coeff a_i,s,k} with Lemma~\ref{well-known} and the estimates \eqref{eq proof: | |_p for coeff g_i,s: eq 1} to get
    \begin{align}
        \label{eq proof: | |_p for coeff g_i,s: eq 3}
        |a_{i,s,k}|_p = \Big|\sum_{j=0}^k\binom{\alpha_i}{j}a_{1,s,k-j}\Big|_p \leq |\mu_k(\alpha_i)|_p^{-1} (k+1)^{s-1}.
    \end{align}
    Finally, Eqs.~\eqref{eq proof: | |_p for coeff g_i,s: eq 1}--\eqref{eq proof: | |_p for coeff g_i,s: eq 3} together with \eqref{eq proof: | |_p for coeff g_i,s: eq 00} yields $|\phi_{\alpha_i,s}(P(t))|_p\le (n+1)^{s-1}|\mu_{n+1}(\alpha_i)|^{-1}_p$.
\end{proof}

\begin{proof}[\textbf{Proof of Proposition~\ref{remainder est}}]
    Fix an integer $\ell$ with $0\leq \ell \leq M$. Recall that by Theorem~\ref{Pade R}~\ref{item: thm Pade R: item 3}, we have
    \begin{align*}
        \fR_{n,i,s,\ell}(z)= \sum_{k=n}^{\infty}\dfrac{k!}{(k-n)!}\dfrac{\phi_{\alpha_i,s}\left((t+n)_{k-n} A_{n,\ell}(t)\right)}{z(z+1)\cdots(z+k)}.
    \end{align*}
    Let $k$ be an integer with $k\geq n$. Then
    \begin{align} \label{trivial factor}
        \left|\dfrac{k!}{(k-n)!} \dfrac{1}{x(x+1)\cdots(x+k)}\right|_p\le |n!|_p |x|^{-k-1}_p.
    \end{align}
    It remains to estimate the $p$--adic norm of the coefficient $\phi_{\alpha_i,s}\left((t+n)_{k-n} A_{n,\ell}(t)\right) \in \Q$. Note that the polynomial $P(t)=(t+n)_{k-n} A_{n,\ell}(t)$ as degree $nM+k+\ell$. Lemma~\ref{lem: decompo pol on (z)_j} implies that
    \begin{align*}
        P(t) = \sum_{j = 0}^{nM+k+\ell} (-1)^j p_j \sum_{h=0}^j\binom{j}{h}(-1)^{j-h}(n-h)_{k-n} A_{n,\ell}(-h) \frac{(z)_j}{j!}.
    \end{align*}
    On the other hand, according to Lemma~\ref{well-known}, we have for each integer $h\geq 0$
    \begin{align*}
        \Big(\prod_{r=2}^d\mu_n(\alpha_{r})^{m_r+1}\Big) A_{n,\ell}(-h) = \Big(\prod_{r=2}^d\mu_n(\alpha_{r})^{m_r+1}\Big)\binom{h}{\ell}\prod_{r=1}^d\binom{h-\alpha_r}{n}^{m_r+1} \in \Z,
    \end{align*}
    so that the polynomial $P(t)\prod_{j=2}^d\mu_n(\alpha_{j})^{m_i+1}$ satisfy the hypothesis of Lemma~\ref{estimate phi_is}. Using \eqref{eq: link mu_n and mu^n}, we conclude that
    \begin{align*}
        |\phi_{\alpha_i,s}(P(t))|_p  \leq (nM+k+\ell+1)^{s-1}|\mu(\alpha_i)|_p^{-({nM+k+\ell+1})}\Big|\prod_{j=2}^d\mu(\alpha_{j})^{m_i+1}\Big|_p^{-n}.
    \end{align*}
    Together with \eqref{trivial factor}, it follows that the series
    $\fR_{n,i,s,\ell}(x)$ converges in $\Q_p$ as soon as $|x\mu(\alpha_i)|_p \geq p$. In that case, writing $Q(k) = (k(M+1)+\ell+1)^{s-1}$, we find
    \begin{align*}
        \big|\fR_{n,i,s,\ell}(x)\big|_p & \leq |n!|_p \Big|\prod_{j=2}^d\mu(\alpha_{j})^{m_i+1}\Big|_p^{-n} \sum_{k\geq n}  |x|^{-k-1}_p Q(k)|\mu(\alpha_i)|_p^{-({nM+k+\ell+1})} \\
        &= o(e^n) |n!|_p \Big|\prod_{j=2}^d\mu(\alpha_{j})^{m_i+1}\Big|_p^{-n} |x|^{-n}_p |\mu(\alpha_i)|_p^{-n(M+1)}
    \end{align*}
    as $n$ tends to infinity, where the implicit constant does not depend on $n$. To conclude, we raise both side of the above inequality to the power $1/n$ and use the well-known estimate
    \begin{align*}
        \lim_{n\to\infty}|n!|_p^{1/n} \leq p^{-1/(p-1)}.
    \end{align*}
\end{proof}

\subsection{Denominators of the Pad\'{e} approximants}

Let $n\geq 0$ be an integer. We keep the notation introduced at the beginning of Section~\ref{section: estimates}. We now estimate the denominators of $P_{n,\ell}(x)$ and $Q_{n,i,s,\ell}(x)$ for $x\in \Q^\times$. Given a positive integer $N$, we denote by $d_N$ the least common multiple of $1,\ldots,N$. Recall that the function $\mu$ (resp. $\mu_N$) is defined in \eqref{eq def: den(alpha) and mu(alpha)} (resp. \eqref{eq: def mu_n}).  We will prove the following result.

\begin{proposition} \label{den ask}
    Let $x\in \Q\setminus \{0\}$. Put $m=\max_{1\leq i \leq d}\{m_i\}$ and define
    \begin{align*}
        D_n(\balpha,x)= \mu_{Mn+M}(x)\bigg(\prod_{i=2}^d\mu_n(\alpha_i)^{m_i+1}\bigg)\cdot \mu_{Mn+M}(\balpha) \cdot d_{Mn+M} \bigg(\prod_{k=1}^{m} d_{\lfloor (Mn+M)/k\rfloor}\bigg).
    \end{align*}
    Then, for any integer $\ell$ with $0\le \ell \le M$ and any $(i,s)\in \Indset$, we have
    \begin{align*}
        D_n(\balpha,x) P_{n,\ell}(x)\in \Z \AND D_n(\balpha,x)Q_{n,i,s,\ell}(x)\in \Z.
    \end{align*}
    Furthermore,
    \begin{align*}
        \limsup_{n\to \infty} |D_n(\balpha,x)|^{1/n}  = e^{\rho_\infty} \AND  \limsup_{n\to \infty} |D_n(\balpha,x)|_p^{1/n}  = e^{-\rho_p},
    \end{align*}
    where
    \begin{align*}
         \rho_\infty & = M\Big(1+\sum_{j=1}^m \dfrac{1}{j}\Big)+\log(\eta),\\
         \rho_p & = -\log|\eta|_p, \\
         \eta & = \mu(x)^M\mu(\balpha)^M\prod_{j=2}^d\mu(\alpha_j)^{m_j+1}.
    \end{align*}
\end{proposition}

The following lemma is a particular case of Shidlovsky's trick for estimating the denominators of coefficients of power of formal Laurent series ({\it confer} \cite[lemma $7$]{G} and \cite[p.17]{A}). Recall that $(a_{i,s,k})_{k\geq 0}$ are the coefficients of the series $g_{i,s}$ introduced in Definition~\ref{def: functions g_i,s}.

\begin{lemma} \label{den ask 0}
    Let $(i,s)\in\Indset$ and $N$ be an integer with $N\ge (s-1)n$. Then, for $k=0,\dots, N$, we have
    \begin{align*}
        a_{1,s,k} \prod_{j=1}^{s-1}d_{\lfloor N/j\rfloor} &  \in \Z, \qquad \textrm{if $i=1$ and $s\geq 2$}, \\
        a_{i,s,k} \cdot \mu_N(\alpha_i)\prod_{j=1}^{s-1}d_{\lfloor N/j\rfloor} &  \in \Z, \qquad \textrm{if $i\geq 2$ and $s\geq 2$}, \\
        a_{i,1,k} \cdot \mu_{N+1}(\alpha_i) &  \in \Z, \qquad \textrm{if $i\geq 2$ and $s = 1$}.
    \end{align*}
\end{lemma}

\begin{proof}
    Put $D_N=d_N d_{\lfloor N/2\rfloor} \cdots d_{\lfloor N/(s-1)\rfloor}$. Suppose that $i=1$ and $s\geq 2$. In view of \eqref{eq: formula coeff a_i,s,k}, it suffices to show that
    \begin{align} \label{enough to show}
    \dfrac{D_N}{(\ell_1+1)\cdots (\ell_{s-1}+1)}\in \Z
    \end{align}
    for any $\ell_1,\dots,\ell_{s-1}$ with $\ell_1+\cdots +{\ell}_{s-1}\le N-s+1$ and $\ell_1\ge \ell_2\ge \ldots \ge \ell_{s-1}\ge 0$.
    For such a choice of integers and for $j=1,\dots,s-1$, we have
    \begin{align*}
        j(\ell_j+1) \leq \ell_1+\dots + \ell_{s-1} + s-1 \leq  N,
    \end{align*}
    hence $\ell_j+1 \leq \lfloor N/j\rfloor$. Thus, $\ell_j+1$ is a factor of $d_{\lfloor N/j\rfloor}$. This implies \eqref{enough to show}.

    \medskip

    If $i=1$ and $s\geq 2$, then we get the result by using the case $i=1$ together with \eqref{eq: formula coeff a_i,s,k} and Lemma~\ref{well-known}. Similarly, we obtain the case $i\geq 2$ and $s=1$ by combining, again, \eqref{eq: formula coeff a_i,s,k} and Lemma~\ref{well-known}.
\end{proof}

\begin{proof}[\textbf{Proof of Proposition~\ref{den ask}}]
    Write $P_{n,\ell}(z)=\sum_{j=0}^{Mn+{\ell}}p_{j,\ell}\cdot(z)_j/j!$, where $p_{j,\ell}$ are defined in Theorem~\ref{Pade R}~\ref{item: thm Pade R: item 2}. Then, Lemma~\ref{well-known} yields
    \begin{align} \label{den pjl}
        \prod_{i=2}^d\mu_n(\alpha_i)^{m_i+1}p_{j,\ell}\in \Z \AND \mu_{Mn+{\ell}}(x)\prod_{i=2}^d\mu_n(\alpha_i)^{m_i+1} P_{n,\ell}(x)\in \Z,
    \end{align}
    thus $D_n(\balpha,x) P_{n,\ell}(x)\in \Z$. We now prove the second statement. Fix $(i,s)\in\Indset$. Again, using Theorem~\ref{Pade R}~\ref{item: thm Pade R: item 2}, we have
    \begin{align} \label{rep Q}
        Q_{n,i,s,\ell}(z)=\sum_{j=1}^{Mn+{\ell}}p_{j,\ell} \left(\sum_{k=0}^{j-1}\dfrac{(-1)^{k+1}a_{i,s,k} k!(z)_j}{j!(z)_{k+1}}\right)
    \end{align}
    where $a_{i,s,k}\in \Q$ are from Definition~\ref{def: functions g_i,s}. By Lemma~\ref{den ask 0}, for any integer $k$ with $0\le k < Mn+M$, we have
    \begin{align} \label{den aisk}
        \mu_{Mn+M}(\alpha_i)\prod_{j=1}^md_{\lfloor (Mn+M)/j\rfloor}a_{i,s,k}\in \Z.
    \end{align}
    Let $k,j$ be two integers with $0\leq k<j$. Notice
    \begin{align*}
        \dfrac{k!(z)_j}{j!(z)_{k+1}} =\dfrac{k!}{j(j-1)\cdots(j-k)}\dfrac{(z+k+1)_{j-k-1}}{(j-k-1)!}
        = \sum_{i=0}^k(-1)^{k-i}\binom{k}{i}\dfrac{1}{j-i}\dfrac{(z+k+1)_{j-k-1}}{(j-k-1)!},
    \end{align*}
    (the last inequality arises from the partial fraction decomposition of $1/(x(x-1)\cdots(x-k))$ evaluated at $x=j$). Lemma~\ref{well-known} implies that
    \begin{align}\label{den part with x}
        d_{Mn+M}\cdot \mu_{Mn+M}(x)\sum_{i=0}^k(-1)^{k-i}\binom{k}{i}\dfrac{1}{j-i}\dfrac{(x+k+1)_{j-k-1}}{(j-k-1)!} \in \Z.
    \end{align}
     We deduce from \eqref{den pjl}, \eqref{den aisk}, \eqref{den part with x} combined with \eqref{rep Q} that
    \begin{align*}
        \Big(\prod_{i=2}^d\mu_n(\alpha_i)^{m_i+1}\Big)\mu_{Mn+M}(\alpha_i)\Big(\prod_{j=1}^md_{\lfloor (Mn+M)/j\rfloor}\Big)d_{Mn+M}\cdot\mu_{Mn+M}(x)Q_{n,i,s,\ell}(x)\in\Z,
    \end{align*}
    hence $D_n(\balpha,x)Q_{n,i,s,\ell}(x)\in \Z$.

    \medskip

    \noindent\textbf{Asymptotic estimate of $D_n(\balpha,x)$.} First, note that
    \begin{align*}
        \lim_{N\to\infty} \mu_N(\balpha)^{1/N} = \mu(\balpha) \AND \lim_{N\to\infty} |\mu_N(\balpha)|_p^{1/N} = |\mu(\balpha)|_p.
    \end{align*}
    The same goes by replacing $\balpha$ with $x$. On the other hand, the prime number theorem ({\it confer} \cite{R-S1}) implies that $d_n=e^{n(1+o(1))}$, and $|d_n|_p\leq 1$ for each $n$ since $d_n$ is an integer. We deduce from the definition of $D_n$ and the above that
    \begin{align*}
        \lim_{n\to\infty}|D_n(\balpha,x)|^{1/n} &= \mu(x)^M\bigg(\prod_{i=2}^d\mu(\alpha_i)^{m_i+1}\bigg)\cdot \mu(\balpha)^M \cdot e^M \bigg(\prod_{k=1}^{m} e^{M/k}\bigg) = e^{\rho_\infty}, \\
        \lim_{n\to\infty}|D_n(\balpha,x)|_p^{1/n} & = |\mu(x)|_p^M\bigg(\prod_{i=2}^d|\mu(\alpha_i)|_p^{m_i+1}\bigg)\cdot |\mu(\balpha)|_p^M = e^{-\rho_p}.
    \end{align*}
\end{proof}

\section{Poincar\'{e}-Perron type recurrence}
\label{section: poincare-perron rec}

We keep the notation of Section~\ref{subsection: Pade approximants for polygamma}. Recall that $d,m_1,\dots,m_d$ are positive integers, and
\[
    M = d-1+ m_1+\dots+m_d.
\]
The goal of the section is to explain how we can improve the asymptotic estimates~\eqref{eq prop: estimate Pade approximants for Q} of Proposition~\ref{est approximants} for $\big(|Q_{n,i,s,\ell}(x)|^{1/n}\big)_{\geq 0}$. As an application, we obtain the following improvement for $M\leq 2$.

\begin{proposition}
    \label{prop: improve bounds via Poincare}
    Let $x\in \Q$. Suppose that $M\leq 2$. Then, for any $(i,s)\in\Indset$ and any integer $\ell$ with $0\le \ell \le M$, we have
    \begin{align*}
        \limsup_{n\to \infty} |P_{n,\ell}(x)|^{1/n} & \le 1, \\
         \limsup_{n\to \infty} |Q_{n,i,s,\ell}(x)|^{1/n} &\le 1.
    \end{align*}
\end{proposition}

Proposition~\ref{prop: improve bounds via Poincare} will be proven in Subsection~\ref{subsection: proof prop estimate poincare}. The idea behind the proof is to show that, for a fixed~$\ell$, the sequences $(P_{n,\ell}(x))_{n\geq 0}$ and $(Q_{n,i,s,\ell}(x))_{n\geq 0}$ satisfy a Poincar\'{e}-type recurrence of some order $J > 0$
\begin{align}
    \label{eq: poincare generic}
    a_J(n) u(n+j)+ a_{J-1}(n) u(n+j-1) + \cdots + a_0(n) u(n) = 0
\end{align}
for large enough $n$, where the coefficients $a_j(t)\in\Q[t]$ are polynomials and $a_J(t)\neq 0$. Then, we can apply Perron's Second Theorem below (see \cite{Perron} and \cite[Theorem~C]{Pituk}) to estimate precisely the growth of a solution of the above recurrence. Computations for small values of $M$ suggest that we can take $J=M+1$. Although the above approach works, each different integer $\ell$ with $0\leq \ell \leq M$ would lead to a different recurrence. In order to alleviate the computations, we will first reduce the problem to the study of some auxiliary sequences introduced in Subsection~\ref{subsection: auxiliary sequences}.

\begin{theorem}[Perron's Second Theorem]
    \label{thm: perron second}
    Let $J$ be a positive integer. Assume that for $j=0,\dots,J$ there exist a function $a_j:\Z_{\geq 0}\rightarrow \C$ and $b_j\in\C$ such that
    \begin{align*}
        \lim_{n\to\infty} a_j(n) = b_j \in\C,
    \end{align*}
    with $b_J\neq 0$. Denote by $\lambda_1,\dots,\lambda_J$ the (not necessarily distinct) roots of the characteristic polynomial
    \begin{align*}
        \chi(z) = b_Jz^J + b_{J-1}z^{J-1} + \cdots + b_0.
    \end{align*}
    Then, there exist $J$ linearly independent solutions $u_1,\dots,u_J$ of \eqref{eq: poincare generic}, such that, for each $j=1,\dots,J$,
    \begin{align*}
        \limsup_{n\to\infty} |u_j(n)|^{1/n} = |\lambda_j|.
    \end{align*}
    In particular, any solution $u$ of \eqref{eq: poincare generic} satisfies $\limsup_{n\to\infty} |u(n)|^{1/n} \leq \max_{1\leq j\leq J} |\lambda_j|$.
\end{theorem}

\begin{remark}
    In the above theorem, there are no restriction on the roots of $\chi(z)$, whereas in Poincar\'{e}'s Theorem and Perron's First Theorem, we ask that
    \begin{align}
        \label{eq: condition roots}
        |\lambda_i|\neq |\lambda_j| \textrm{ for } i\neq j,
    \end{align}
    see \cite[Theorem~A and~B]{Pituk}. This is important to note as it seems that the characteristic polynomials we are dealing with never satisfy condition~\eqref{eq: condition roots}, see Figure~\ref{table: char pol}.
\end{remark}

\subsection{Auxiliary sequences}
\label{subsection: auxiliary sequences}

Recalll that $\Delta_{-1}(P(z)) = P(z-1)-P(z)$ for each $P(z)\in\Q[z]$. Similarly, $\bDelta_{-1}(P(t)) = P(t-1)-P(t)$. The $K[z]$--morphisms $\phi_{\alpha_i,s}$ are defined in~\eqref{eq: def phi_i,s}. For each integer $n\geq M$ and each $(i,s)\in\Indset$, put
\begin{align*}
        \hA_{n}(z) & = (-1)^{n-M}\prod_{r=1}^d\left(\dfrac{(z+\alpha_r)_n}{n!}\right)^{m_r+1}, \notag \\
        \hP_{n}(z) & = \Delta^{n-M}_{-1}\left(A_n(z)\right), \\
        \hQ_{n,i,s}(z) &  =\phi_{\alpha_i,s}\left(\dfrac{\hP_{n}(z)-\hP_{n}(t)}{z-t}\right).
\end{align*}

The goal of this section is to prove the following result.

\begin{proposition}
    \label{prop: estimate Pade via auxiliary seq}
    Let $x\in\Q$. For each $(i,s)\in\Indset$ and each integer $\ell$ with $0\leq \ell \leq M$, we have
	\begin{align*}
		\limsup_{n\to\infty} |P_{n,\ell}(x)|^{1/n} & \leq \max_{0\leq j \leq M}\limsup_{n\to\infty} |\hP_{n}(x-j)|^{1/n},\\
		\limsup_{n\to\infty} |Q_{n,i,s,\ell}(x)|^{1/n} & \leq \max_{0\leq j \leq M}\limsup_{n\to\infty} |\hQ_{n,i,s}(x-j)|^{1/n}.
	\end{align*}
\end{proposition}

We first establish some useful intermediate lemmas. The following while is similar to Lemma~\ref{lem: t^kP_n in ker phi_i,s}.

\begin{lemma}
    \label{lem: kernel auxiliary seq}
	Let $j,n$ be non-negative integers with $j\leq M < n$. Then
	\begin{align}
        \label{eq: kernel auxiliary seq}
		t^k \hP_n(t-j) \in \bigcap_{(i,s)\in\Indset} \ker \phi_{\alpha_i,s} \qquad ( 0\leq k < n-M).
	\end{align}
\end{lemma}

\begin{proof}
    Write $\hA_n(t-j) = \hA_{n-M}(t)Q_j(t)$ with $Q_j(t)\in\Q[t]$, and fix an integer $k$ with $0\leq k < n-M$. By Lemma~\ref{key 2}~\ref{item: lem key 2: item 3 bis}, we have
    \begin{align*}
        t^k \hP_n(t-j) = t^k\bDelta^{n-M}_{-1}\big(\hA_{n-M}(t)Q_j(t)\big) & \in \bDelta_{-1}\bigg(\Q[t]\prod_{r=1}^d (t+\alpha_r)^{m_r+1}\bigg)  = \bigcap_{(i,s)\in\Indset} \ker \phi_{\alpha_i,s},
    \end{align*}
    the last equality coming from Lemma~\ref{lem: description kernels}. Hence~\eqref{eq: kernel auxiliary seq}.
\end{proof}

\begin{lemma}
    \label{lem : phi_i,s de a(z)P(z)}
    Let $P(z), a(z)\in \Q[z]$ with $\deg a(z) = d \geq 0$, and set $\tP(z) = a(z)P(z)$. Given $(i,s)\in\Indset$, we suppose that
    \begin{align*}
    	t^kP(t) \in \ker \phi_{\alpha_i,s} \quad (k=0,\dots,d-1).
    \end{align*}
    Put
    \begin{align*}
        Q(z) = \phi_{\alpha_i,s}\bigg(\frac{P(z)-P(t)}{z-t} \bigg) \AND \tQ(z) = \phi_{\alpha_i,s}\bigg(\frac{\tP(z)-\tP(t)}{z-t} \bigg).
    \end{align*}
    Then,  we have $\tQ(z) = a(z)Q(z)$.
\end{lemma}

\begin{proof}
     We follow the arguments in the proof of \cite[Lemma 3.8]{KP}. First, note that the polynomial
     \begin{align*}
        b(t) = \frac{a(z)-a(t)}{z-t} \in \Q[z,t]
     \end{align*}
     has degree at most $d-1$ in $t$. By hypothesis, the polynomial $b(t)P(t)$ belongs to the kernel of $\phi_{\alpha_i,s}$. To conclude, it suffices to write
     \begin{align*}
        \frac{\tP(z)-\tP(t)}{z-t} = a(z)\frac{P(z)-P(t)}{z-t} + b(t)P(t),
     \end{align*}
     and then to apply $\Q[z]$--linear morphism $\phi_{\alpha_i,s}$ to the above identity.
\end{proof}

\begin{proposition}
    \label{prop: expression Pade via auxiliary seq}
    Let $n,\ell$ be integers with $0\leq \ell \leq M$ and $2M\leq n$. For each $(i,s)\in\Indset$, we have
    \begin{align*}
        P_{n,\ell}(z) = \sum_{j=0}^{M} a_{j}(n,z) \hP_n(z-j) \AND  Q_{n,i,s,\ell}(z)  = \sum_{j=0}^{M} a_{j}(n,z) \hQ_{n,i,s}(z-j),
    \end{align*}
    where
    \begin{align*}
        a_{j}(n,z) = \sum_{k=0}^{M-j} \binom{n}{k}\binom{M-k}{j}\frac{(-1)^{j+k+\ell+Mn+M+n}}{\ell!} \Delta^{k}_{-1}\big((z-n+k)_\ell\big).
    \end{align*}
\end{proposition}

\begin{proof}
	According to Lemma~\ref{lem: difference operator of product}, and since $\Delta_{-1}^k((z)_\ell) = 0$ if $k>\ell$, we have
	\begin{align*}
		(-1)^{\ell + Mn+M+n}\ell! P_{n,\ell}(z) = \Delta_{-1}^n\big((z)_\ell \hA_n(z) \big)
		& = \sum_{k=0}^{n}\binom{n}{k} \Delta^{k}_{-1}\big((z-n+k)_\ell\big) \Delta^{n-k}_{-1}\big(\hA_n(z)\big) \\
		& = \sum_{k=0}^{M}\binom{n}{k} \Delta^{k}_{-1}\big((z-n+k)_\ell\big) \Delta^{M-k}_{-1}\big(\hP_n(z)\big).
	\end{align*}
    Using the identity $\Delta_{-1}^m = \sum_{j=0}^{m}\binom{m}{j}(-1)^{m-j}\tau_{-1}^j$, we find
    \begin{align*}
        \Delta^{M-k}_{-1}(\hP_n(z)) = \sum_{j=0}^{M-k} \binom{M-k}{j}(-1)^{M-k-j} \hP_n(z-j),
    \end{align*}
    and rearranging the terms in the previous expression, we obtain the expected formula for $P_{n,\ell}(z)$. It follows that
    \begin{align*}
        Q_{n,i,s,\ell}(z) = \phi_{\alpha_i,s}\bigg(\frac{P_{n,\ell}(z)-P_{n,\ell}(t)}{z-t} \bigg) & = \sum_{j=0}^{M}  \phi_{\alpha_i,s}\bigg(\frac{a_{j}(n,z)\hP_n(z-j)-a_{j}(n,t)\hP_n(t-j)}{z-t}\bigg).
    \end{align*}
    Note that for each $j$, the polynomial $a_j(n,z)\in\Q[z]$ has degree at most $M$ (in the variable $z$), and that according to Lemma~\ref{lem: kernel auxiliary seq}, we have $t^k \hP_n(t-j) \in \ker \phi_{\alpha_i,s}$ for $k=0,\dots,M-1$. By Lemma~\ref{lem : phi_i,s de a(z)P(z)}, we conclude that
    \begin{align*}
        \phi_{\alpha_i,s}\bigg(\frac{a_{j}(n,z)\hP_n(z-j)-a_{j}(n,t)\hP_n(t-j)}{z-t}\bigg) & =  a_j(n,z)\phi_{\alpha_i,s}\bigg(\frac{\hP_n(z-j)-\hP_n(t-j)}{z-t}\bigg) \\
        & = a_j(n,z) \hQ_{n,i,s}(z-j).
    \end{align*}
\end{proof}

\begin{proof}[\textbf{Proof of Proposition~\ref{prop: estimate Pade via auxiliary seq}}]
    This is a consequence of Proposition~\ref{prop: expression Pade via auxiliary seq}, noticing that for a fixed $x\in\Q$, the polynomials $a_j(n,x)$ (in the variable $n$) satisfy
    \begin{align*}
        \max_{0\leq j\leq M}\limsup_{n\to\infty} |a_j(n,x)|^{1/n} = 1.
    \end{align*}
\end{proof}

\subsection{Poincar\'{e}-Perron recurrence}

In view of Proposition~\ref{prop: estimate Pade via auxiliary seq}, it remains to estimate the auxiliary sequences introduced in the previous section. This will be done by using Perron's Second Theorem (see Theorem~\ref{thm: perron second}). First,  we reduce the problem by proving that it suffices to find a Poincar\'{e}-type recurrence for the sequence $\big(\hP_{n}(z)\big)_{n\geq M}$. The next result ensures that for each $(i,s)\in\Indset$, the sequence $\big(\hQ_{n,i,s}(z)\big)_{n\geq M}$ will also satisfy the same recurrence.

\begin{proposition}
    \label{lem: Poincare rec for the Q_n}
    Suppose that there exist integers $J,d_0,\dots,d_J \geq 0$ and sequences $\big(a_j(n,z)\big)_{\geq n}$ in $\Q[z]$ for $j=0,\dots,J$ with the following properties. For each integer $n\geq M$, we have $a_j(n,z)\in\Q[z]$ and
    \begin{align*}
        \deg a_j(n,z) \leq d_j \qquad (0\leq j\leq J).
    \end{align*}
    Assume that $\big(\hP_{n}(z)\big)_{n\geq M}$ satisfies the recurrence
    \begin{align}
        \label{eq: poincare general}
        \sum_{j=0}^J a_j(n,z) \hP_{n+j}(z) = 0
    \end{align}
    for each $n\geq M$. Then, for any $(i,s)\in\Indset$, the sequence $\big(\hQ_{n,i,s}(z)\big)_{n\geq M}$ also satisfies the recurrence \eqref{eq: poincare general} for each integer $n\geq M+\max\{0,N\}$, where $N = \max_{0\leq j\leq J} \{d_j-j\}$.
\end{proposition}

\begin{proof}
    By hypothesis, we have
    \begin{align*}
        0 &= \sum_{j=0}^J \frac{\tP_{j,n}(z) - \tP_{j,n}(t)}{z-t}, \qquad \textrm{where } \tP_{j,n}(z) = a_j(n,z) \hP_{n+j}(z).
    \end{align*}
    Fix $(i,s)\in\Indset$. Using Lemmas~\ref{lem: kernel auxiliary seq} and~\ref{lem : phi_i,s de a(z)P(z)} we find
    \begin{align*}
        0 = \phi_{\alpha_i,s}\bigg( \sum_{j=0}^J \frac{\tP_{j,n}(z) - \tP_{j,n}(t)}{z-t}\bigg) = \sum_{j=0}^J \phi_{\alpha_i,s}\bigg(\frac{\tP_{j,n}(z) - \tP_{j,n}(t)}{z-t} \bigg) = \sum_{j=0}^J a_j(n,z) \hQ_{n+j,i,s}(z),
    \end{align*}
    for each $n\geq M$ such that $n+j-M\geq d_j$ for $j=0,\dots,J$, \textit{i.e.} such that $n\geq M+N$.
\end{proof}

Recall that for any integer $n\geq M$, we have
\begin{align*}
    \hP_{n}(z) & = \Delta^{n-M}_{-1}\left(A_n(z)\right) = \sum_{k\in\Z} F(z,n,k),
\end{align*}
where $F(z,n,k)$ is the hypergeometric term
\begin{align*}
    F(z,n,k) 
    = \binom{n-M}{k}(-1)^{k} \prod_{r=1}^d\binom{z+\alpha_r-k-1}{n}^{m_r+1},
\end{align*}
with the convention that $\binom{n-M}{k} = 0$ if $k<0$ for $k> M-n$. We can show that there exist an explicit integer $J=J(M)$ and polynomials $a_j(y,z)\in \Q[y,z]$ for $j=0,\dots,J$, not all zero, such that \eqref{eq: poincare general} holds for each integer $n\geq M$. Given a specific $x\in\Q$, we can also ensure that at least one of the coefficient $a_j(n,x)$ is non-zero. The proof is based on Sister Celine's Method applied to the hypergeometric term $F(z,n,k)$ (see \cite[Theorem~4.4.1]{ZeilbergerAeqB}), together with the arguments used in the proof of \cite[Theorem 6.2.1]{ZeilbergerAeqB} (which provides an explicit bound for the order $J$ of the recurrence). Numerical computations for small values of $M$ suggest that we can always take
\begin{align*}
    J(M) = M+1.
\end{align*}
Then, we can compute $J$ and the coefficients $a_j(n,z)$ in \eqref{eq: poincare general} by using Zeilberger's algorithm \cite[Chapter~6]{ZeilbergerAeqB} and the software MAPLE for small values of $M$.

\subsection{Small values of $M$}
\label{subsection: proof prop estimate poincare}

Fix $x\in\Q$ and recall that $\alpha_1=0$. If $M\leq 2$, then $d=1$ and $m_1 = M$. In this section we give the explicit Poincar\'{e}-Perron recurrence satisfied by $(\hP_n(z))_{n\geq M}$ for the above values of $M$. This will allow us to prove Proposition~\ref{prop: improve bounds via Poincare}. The MAPLE's programs used for our computations are available at \verb"https://apoels-math-u.net/Maple/polygamma_pade.zip".

\medskip

\noindent\textbf{Case $d=1$ and $m_1=1$.} Then
\begin{align*}
    \hP_{n}(z) &= \sum_{k=0}^{n-1} F(z,n,k),\qquad \textrm{where } F(z,n,k) = (-1)^{k}\binom{n-1}{k} \binom{z+n-k-1}{n}^2.
\end{align*}
Zeilberger's algorithm ensures that $(\hP_{n,0}(z))_{n\geq 1}$ satisfies the recurrence
\begin{align}
    \label{eq: Poincare rec d=1 and m=1}
    (n^2+5n+6) u(n+2) -((8+4z)n+6z+12)u(n+1) -(n^2+n)u(n) = 0.
\end{align}
Its characteristic polynomial is
\[
    \chi(T) = T^2-1
\]
(independent of $z$), whose roots have modulus $1$. By Proposition~\ref{lem: Poincare rec for the Q_n}, the sequence $(\hQ_{n,1,2}(z))_{n\geq 1}$ also satisfies \eqref{eq: Poincare rec d=1 and m=1}. Together with Theorem~\ref{thm: perron second}, this yields, for each $(i,s)\in\Indset$
\begin{align}
    \label{eq: estimate auxilliary M=1}
    \max_{0\leq j \leq 1}\limsup_{n\to\infty} |\hP_{n}(x-j)|^{1/n} \leq 1 \AND \max_{0\leq j \leq 1}\limsup_{n\to\infty} |\hQ_{n,1,2}(x-j)|^{1/n} \leq 1.
\end{align}

\medskip

\noindent\textbf{Case $d=1$ and $m_1=2$.} We have
\begin{align*}
    \hP_{n}(z) &= \sum_{k=0}^{n-1} F(z,n,k),\qquad \textrm{where } F(z,n,k) = (-1)^{k}\binom{n-1}{k} \binom{z+n-k-1}{n}^3.
\end{align*}
Zeilberger's algorithm ensures that $(\hP_{n,0}(z))_{n\geq 1}$ satisfies the recurrence
\begin{align}
    \label{eq: Poincare rec d=1 and m=2}
    a_3(n,z) u(n+3) + a_2(n,z)u(n+2) + a_1(n,z)u(n+1)+ a_0(n,z)u(n) = 0,
\end{align}
where
\begin{align*}
    a_3(n,z) &= 2(3n+5)(n+4)(2n+7)(n+3)^2 \\
    a_2(n,z) &= -(3n+7)\big(9n^4+74n^3+3(9x^2+45x+127)n^2+(585x+117x^2+976)n+120(x^2+7+5x)\big), \\
    a_1(n,z) &= 2n(n+1)(9n^3+48n^2+80n+43) \\
    a_0(n,z) &= -n(3n+8)(n-1)(n+1)^2.
\end{align*}
Its characteristic polynomial is
\[
    \chi(T) = 4T^3-9T^2+6T-1 = (T-1)^2(4T-1),
\]
(independent of $z$), whose largest roots have modulus $1$. By Proposition~\ref{lem: Poincare rec for the Q_n}, the sequence $(\hQ_{n,1,s}(z))_{n\geq 1}$ also satisfies \eqref{eq: Poincare rec d=1 and m=2} for $s=2,3$. Together with Theorem~\ref{thm: perron second}, this yields, for each $(i,s)\in\Indset$
\begin{align}
    \label{eq: estimate auxilliary M=2}
    \max_{0\leq j \leq 2}\limsup_{n\to\infty} |\hP_{n}(x-j)|^{1/n} \leq 1 \AND \max_{0\leq j \leq 2}\limsup_{n\to\infty} |\hQ_{n,i,s}(x-j)|^{1/n} \leq 1.
\end{align}

\begin{proof}[\textbf{Proof of Proposition~\ref{prop: improve bounds via Poincare}}]
    Suppose that $M\leq 2$ and fix $x\in\Q$. Then \eqref{eq: estimate auxilliary M=1} and~\eqref{eq: estimate auxilliary M=2} combined with Proposition~\ref{prop: estimate Pade via auxiliary seq} yields the expected result.
\end{proof}

\noindent\textbf{Case $M\geq 3$.} It is possible to apply the above method when the parameters $M$ is larger than $2$. However, the computing time increases significantly at each step. Maple's computations suggest that the first characteristic polynomials are the following (we did the computation for $d=1$):

\begin{figure}[H]
  \centering
  \begin{align*}
    \begin{array}{|c||c|}
        \hline
        M & \textrm{characteristic polynomial } \chi(T)  \\
        \hline \hline
        1 & (T-1)(T+1) \\
        \hline
        2 & (T-1)^2(4T-1)\\
        \hline
        3 &  (T-1)(T+1)(27T^2+1) \\
        \hline
        4 & (T-1)^2(-1+16T)(4T+1)^2\\
        \hline
        5 & (T-1)(T+1)(3125T^4+625T^2+1) \\
        \hline
        6 & (T-1)^4(-1+64T)(27T+1)^2 \\
        \hline
        7 & (T-1)(T+1)(823543T^6+6000099T^4+12005T^2+1) \\
        \hline
    \end{array}
\end{align*}
  \caption{Expected characteristic polynomials for small values of $M$}\label{table: char pol}
\end{figure}

It is seems that for $M=7$, the characteristic polynomial has two root of modulus $2.698\cdots >1$.

\begin{remark}
    \label{remark function g(M) relaxed}
    It would be desirable to determine the characteristic polynomial $\chi_M(T)$ as well as the modulus $\rho_M$ of its largest roots for arbitrary $M$, since this would allow us to relax Condition~\eqref{ineq x} of our main Theorem~\ref{main 1} by replacing $g(M)$ defined in \eqref{eq: intro def g(M)} with $\log \rho_M$.
\end{remark}

\section{Proof of Theorem~\ref{main 1}}
\label{section: proof main thm}

Let $d,m_1,\ldots,m_d$ be positive integers and $\balpha=(\alpha_1,\cdots,\alpha_d)\in \Q^d$ with $\alpha_1=0$ and $\alpha_i-\alpha_j\notin \Z$ for any $i\neq j$. Put $m=\max_{1\leq i\leq d} m_i$, $M=\sum_{i=1}^dm_i+d-1$, $\bm = (m_1,\dots,m_d)$, and
\begin{align*}
    \Indset=\{(i,s) \, ; \, 1\le i \le d \AND 1\le s \le m_i+1\} \setminus\{(1,1)\}.
\end{align*}
Recall that $q_p = p$ if $p\geq 3$ and $q_p = 4$ if $p=2$. The function $\mu$ is defined in \eqref{eq def: den(alpha) and mu(alpha)}. Define
\begin{align*}
    f(\balpha,\bm) = g(M) + M\bigg(1+\sum_{j=1}^m \dfrac{1}{j}\bigg)+\log \Big(\mu(\balpha)^M\prod_{i=2}^d \mu(\alpha_i)^{m_i+1}\Big)-\log|\mu(\balpha)|_p,
\end{align*}
where the function $g$ is as in \eqref{eq: intro def g(M)}, namely $g(M)=0$ if $M\leq 2$ and
\begin{align*}
    g(M) = (M+1)\log \bigg(\frac{2(M+1)^{M+1}}{M^M}\bigg) \qquad \textrm{if $M>2$.}
\end{align*}
Note that $f(\balpha,\bm)$ is equal to $f(\balpha,m)$ defined in \eqref{eq def: F_p(alpha,m)} when $m_i = m$ for each $i$.
In this section, we prove the following stronger version of Theorem \ref{main 1} (which corresponds to the case $\bm=(m,\dots,m)$).

\begin{theorem} \label{main 1 strong}
    Let $p$ be a prime number and $x\in \Q$ satisfying
    \begin{align}
        \label{eq: thm: main 1 strong: condition x bis}
         |x|_p \geq q_p\max_{1\leq i\leq d}\{1, |\alpha_i|_p\}
    \end{align}
    and
    \begin{align} \label{ineq x main strong}
       \dfrac{\log p}{p-1}+\log |x|_p > M \log \big(\mu(x)|\mu(x)|_p\big) + f(\balpha,\bm).
    \end{align}
    Then the $m_1+\cdots+m_d+1$ elements of $\Q_p$
    \begin{align*}
        1, G_p^{(2)}(x+\alpha_1), \ldots, G_p^{(m_1+1)}(x+\alpha_1), \ldots, G_p^{(2)}(x+\alpha_d), \ldots, G_p^{(m_d+1)}(x+\alpha_d)
    \end{align*}
    are linearly independent of $\Q$.
\end{theorem}

We will deduce Theorem~\ref{main 1 strong} from the next theorem. For any non-negative integers $\ell,n$ with $0\le \ell\le M$ and each $(i,s)\in \Indset$, the polynomials $P_{n,\ell}(z)$, $Q_{n,i,s,\ell}(z)$, and the Pad\'{e} approximation $\fR_{n,i,s,\ell}(z) = P_{n,\ell}(z)R_{\alpha_i,s}(z)-Q_{n,i,s,\ell}(z)$ of $R_{\alpha_i,s}(z)$ are defined in Theorem~\ref{Pade R}. Recall that $R_{\alpha,s}$ is as in Definition~\ref{def: R_alpha,s}.

\begin{theorem} \label{main 1 strong - technical}
    Let $p$ be a prime number and $x\in \Q$ satisfying
    \begin{align}
        \label{eq: thm: main 1 strong: condition x}
         |x|_p \geq q_p\max_{1\leq i\leq d}\{1, |\alpha_i|_p\}.
    \end{align}
    Let $\beta,\rho_\infty,\rho_p,\delta$ be real numbers and $(D_n)_{n\geq 0}$ be a sequence of positive integers such that, for each $(i,s)\in\Indset$ and each $\ell=0,\dots,M$, the numbers $D_nP_{n,\ell}(x)$ and $D_nQ_{n,i,s,\ell}(x)$ are integers (for each integer $n\geq 0$) and
    \begin{align*}
        & \limsup_{n\to\infty} \max \big\{|P_{n,\ell}(x)|,|Q_{n,i,s,\ell}(x)|\big\} ^{1/n}  \leq e^\beta, \\
        & \limsup_{n\to\infty} D_n^{1/n}  \leq e^{\rho_\infty}, \\
        & \limsup_{n\to\infty} |D_n|_p^{1/n}  \leq e^{-\rho_p}, \\
        & \limsup_{n\to\infty} |\fR_{n,i,s,\ell}(x)|_p^{1/n}  \leq e^{-\delta}.
    \end{align*}
    Suppose that $\beta+\rho_\infty < \delta+\rho_p$. Then the $M+1$ elements $(1,R_{\alpha_i,s}(x))_{(i,s)\in\Indset}$ of $\Q_p$ are linearly independent over $\Q$.
\end{theorem}

\begin{proof}
    The proof is classical, see for example \cite[Proposition $5.7$]{DHK2} (in the case $K=\Q$, $v_0=p$), which also provides an effective irrationality measure result. For the sake of completion we recall the arguments. First, note that the condition~\eqref{eq: thm: main 1 strong: condition x} ensures that $|x+\alpha_i|_p = |x|_p \geq q_p > 1$ for $i=1,\dots,d$, so that $R_{\alpha_i,s}(x)$ converges $p$--adically for each $(i,s)\in\Indset$ by Lemma~\ref{lem: Mellin R_alpha,s}. If $p$ does not divide the denominator $\den(\alpha_i)$ of $\alpha_i$, we have $|x|_p|\cdot|\mu(\alpha_i)|_p = |x|_p > 1$. If $p$ divides $\den(\alpha_i)$, then
    \[
        |\mu(\alpha_i)|_p = |\den(\alpha_i)|_p p^{-1/(p-1)} = |\alpha_i|_p^{-1}p^{-1/(p-1)}.
    \]
    and \eqref{eq: thm: main 1 strong: condition x} implies that $|x|_p\cdot |\mu(\alpha_i)|_p \geq q_p p^{-1/(p-1)} > 1$. Consequently, the series $\fR_{n,i,s,\ell}(x)$ converges $p$--adically, see Proposition~\ref{remainder est}. By contradiction, suppose that $(1,R_{\alpha_i,s}(x))_{(i,s)\in\Indset}$ are linearly dependent over $\Q$. Then, there exists $(b,b_{i,s})_{(i,s)\in\Indset}\in\Z^{M+1}\setminus\{0\}$ such that
    \begin{align}
        \label{eq proof: main thm: eq1}
        b+\sum_{(i,s)\in\Indset} b_{i,s} R_{\alpha_i,s}(x) = 0.
    \end{align}
    Given a positive integer $n$, define
    \begin{align*}
        \hp_{n,\ell} := D_nP_{n,\ell}(x) \AND \hq_{n,i,s,\ell} := D_n Q_{n,i,s,\ell}(x),
    \end{align*}
    for each $(i,s)\in\Indset$ and each $\ell=0,\dots,M$. By hypothesis $\hp_{n,\ell}$ and $\hq_{n,i,s,\ell}$ are integers. Theorem~\ref{thm: non-zero det Pade approx} implies that the $(M+1)\times (M+1)$ matrix
    \begin{align*}
        \begin{pmatrix} \hp_{n,\ell} \\ \hq_{n,i,s,\ell}\end{pmatrix}_{\substack{ 0\le \ell \le M \\ (i,s) \in \Indset}}
    \end{align*}
    is non-singular. Consequently, there exists an integer $\ell$ with $0\leq \ell \leq M$ such that
    \begin{align*}
        K_n = K_n(\ell) := b\hp_{n,\ell} + \sum_{(i,s)\in\Indset} b_{i,s}\hq_{n,i,s,\ell}
    \end{align*}
    is a non-zero integer. Our hypothesis implies that $\limsup_{n\to\infty} |K_n|^{1/n} \leq e^{\beta+\rho_\infty}$. Since $|K_n||K_n|_p\geq 1$, it follows that
    \begin{align}
        \label{eq proof: main thm: eq2}
        \liminf_{n\to\infty} |K_n|_p^{1/n} \geq \liminf_{n\to\infty} |K_n|^{-1/n} \geq e^{-(\beta+\rho_\infty)}.
    \end{align}
    On the other hand, using \eqref{eq proof: main thm: eq1}, we find
    \begin{align*}
        K_n = K_n - \hp_{n,\ell}\Big(b+\sum_{(i,s)\in\Indset} b_{i,s} R_{\alpha_i,s}(x)\Big) & = \sum_{(i,s)\in\Indset} b_{i,s}(\hq_{n,i,s,\ell}-R_{\alpha_i,s}(x)\hp_{n,\ell}) \\
        & = \sum_{(i,s)\in\Indset} b_{i,s} D_n \fR_{n,i,s,\ell}(x),
    \end{align*}
    from which we deduce the upper bound
    \begin{align*}
        \limsup_{n\to\infty} |K_n|_p^{1/n} \leq \limsup_{n\to\infty} |D_n\fR_{n,i,s,\ell}(x)|_p^{1/n} \leq e^{-\delta-\rho_p}.
    \end{align*}
    Together with \eqref{eq proof: main thm: eq2}, we deduce that $e^{-(\beta+\rho_\infty)} \leq e^{-(\delta+\rho_p)}$, which contradicts our hypothesis $\beta+\rho_\infty < \delta+\rho_p$.
\end{proof}

\begin{proof}[\textbf{Proof of Theorem~\ref{main 1 strong}}]
    For each integer $n\geq 0$, set $D_n = D_n(\balpha,x)$, where $D_n(\balpha,x)$ is as in Proposition~\ref{den ask}. Define $\beta = g(M)$ and
    \begin{align*}
        & \rho_\infty = M\Big(1+\sum_{j=1}^m \dfrac{1}{j}\Big) +\log\Big(\mu(x)^M\mu(\balpha)^M\prod_{i=2}^d\mu(\alpha_i)^{m_i+1}\Big), \\
        & \rho_p = -\log\Big|\mu(x)^M\mu(\balpha)^M\prod_{i=2}^d\mu(\alpha_i)^{m_i+1}\Big|_p, \\
        & \delta =\log\left( p^{1/(p-1)} \cdot\Big|x\mu(\balpha)^{M+1}\prod_{i=2}^d\mu(\alpha_{i})^{m_{i}+1}\Big|_p \right).
    \end{align*}
    Condition \eqref{ineq x main strong} is equivalent to $\beta+\rho_\infty < \delta+\rho_p$. By Propositions~\ref{est approximants},~\ref{prop: improve bounds via Poincare},~\ref{den ask} and~\ref{remainder est} (also see Remark~\ref{remark: estimate remainder}) the hypotheses of Theorem~\ref{main 1 strong - technical} are satisfied with the above parameters. Therefore, the elements $(1,R_{\alpha_i,s}(x))_{(i,s)\in\Indset}$ of $\Q_p$ are linearly independent over $\Q$. To conclude, it suffices to notice that \eqref{def R_s via G_p} and \eqref{eq: series expansion of R_s} combined with Lemma~\ref{lem: Mellin R_alpha,s} yields, for each $(i,s)\in\Indset$ with $s\geq 2$,
    \begin{align*}
        R_{\alpha_i,s}(x) = R_s(x+\alpha_i) = -G_p^{(s)}(x+\alpha_i).
    \end{align*}
\end{proof}

\begin{proof}[\textbf{Proof of Theorem~\ref{main cor 2}}]
    With the notation of the corollary, we have
    \begin{align*}
         \mu(\balpha)=\mu(\alpha_2)= p^{b+1/(p-1)}.
    \end{align*}
    A short computation also yields  $M=2m+1$ (since $d=2$) and
    \begin{align*}
        f(\balpha,d,m) &= g(M) + M\bigg(1+\sum_{j=1}^m \dfrac{1}{j}\bigg)+\log \Big(\mu(\balpha)^M\prod_{i=2}^d \mu(\alpha_i)^{m+1}\Big)-\log|\mu(\balpha)|_p \\
        & =  g(M) +(2m+1)\bigg(1+\sum_{j=1}^m \dfrac{1}{j}\bigg)+\log \Big(\mu(\balpha)^{3m+3} \Big),
    \end{align*}
    and the condition~\eqref{ineq x} of Theorem~\ref{main 1} is equivalent to condition~\eqref{ineq p cor 2}, with $x=p^{-a}$ and $\alpha_2=p^{-b}$. Also note that the condition $\delta >0$ implies that $a>b$, so that $|x|_p > |\alpha_2|_p$.
\end{proof}

\noindent
{\bf Acknowledgements.}

The authors extend sincere gratitude to Professors Daniel Bertrand, Sinnou David, and Noriko Hirata-Kohno for their invaluable suggestions.
Additionally, this work is partially supported by the Research Institute for Mathematical Sciences, an international joint usage and research center located at Kyoto University.
The first author is supported by JSPS KAKENHI Grant Number JP24K16905.

\bibliography{}

\begin{thebibliography}{99}



\bibitem{A}
Y.~Andr\'{e},
{\it $G$-functions and geometry},
Aspects of Mathematics, E13. Friedr. Vieweg \& Sohn, Braunschweig, 1989.

\bibitem{An2}
Y.~Andr\'{e},
{\it S\'{e}ries Gevrey de type arithm\'{e}tique, I. Th\'{e}or\`{e}mes de puret\'{e} et de dualit\'{e}},
Ann. of Math. {\textbf{151}} (2000), 705--740.

\bibitem{B-R}
K.~Ball and T.~Rivoal,
{\it Irrationalit\'{e} d'une infinit\'{e} de valeurs de la fonction z\^{e}ta aux entiers impairs},
Invent.\ Math.\ \textbf{146}.~1 (2001), 193--207.

\bibitem{B-D}
M.~Barkatou and A.~Duval,
{\it Sur les s\'{e}ries formelles solutions d'\'{e}quations aux diff\'{e}rences polyn\^{o}miales},
Ann. Institut Fourier \textbf{44} (1994), 495--524.

\bibitem{B1}
P.~Bel,
{\it Fonctions $L$ $p$-adiques et irrationalit\'{e}},
Ann.\ Scuola Norm. Sup. Pisa Cl. Sci. (5) Vol. \textbf{IX} \ (2010), 189--227.

\bibitem{B2}
P.~Bel,
{\it Irrationalit\'{e} des valeurs de $\zeta_p(4,x)$},
Journal de Th\'{e}orie des Nombres de Bordeaux, \textbf{31}, no.\ 1 (2019), 81--99.

\bibitem{B}
F.~Beukers,
{\it Irrationality of some $p$-adic $L$-values},
Acta Math.\ Sin. \textbf{24},\ no.\ 4, (2008), 663--686.

\bibitem{Bez-R}
J.~P.~Bezivin and P.~Robba,
{\it A new $p$-adic method for proving irrationality and transcendence results},
Ann. of Math. (2) \textbf{129},\ (1989),\ no. 1, 151--160.

\bibitem{Cal}
F.~Calegari,
{\it Irrationality of Certain $p$-adic Periods for Small $p$},
Int. Math. Res. Not., \textbf{20}, (2005), 1235--1249.


\bibitem{Car}
L.~Carlitz,
{\it Some polynomials of Touchard connected with the Bernoulli numbers},
Canad. J. Math., \textbf{9}, (1957), 188--190.


\bibitem{Cla}
T.~Clausen,
{\it Lehrsatz aus einer Abhandlung $\ddot{\text{u}}$ber die Bernoullischen Zahlen},
Astron. Nachr. \textbf{17}, 1840, 351--352.



\bibitem{C}
H.~Cohen,
{\it Number theory. Vol. II. Analytic and modern tools},
Graduate Texts in  Mathematics, vol. \textbf{240}, Springer, New York, 2007.


\bibitem{DHK1}
S.~David, N.~Hirata-Kohno and M.~Kawashima,
\textit{ Can  polylogarithms at algebraic points be linearly independent $?$},
Moscow Journal in Combinatorics and Number Theory, \textbf{9} (2020), 389--406.

\bibitem{DHK2}
S.~David, N.~Hirata-Kohno and M.~Kawashima,
\textit{ Linear Forms in Polylogarithms},
Ann. Sc. Norm. Super. Pisa Cl. Sci. (5) Vol. XXIII (2022), 1447-1490

\bibitem{DHK3}
S.~David, N.~Hirata-Kohno and M.~Kawashima,
\textit{Linear independence criteria for generalized polylogarithms with distinct shifts},
Acta Arithmetica \textbf{206} (2022), 127-169.


\bibitem{Dia1}
J.~Diamond,
{\it The $p$-adic log gamma function and $p$-adic Euler constants},
Trans. Amer. Math. Soc. \textbf{233}, 1977, 321--337.

\bibitem{Dia2}
J.~Diamond,
{\it On the values of $p$-adic $L$-functions at positive integers},
Acta Arith., \textbf{35}, 1979, 223--237.


\bibitem{F}
S.~Fischler,
{\it Linear independence of odd zeta values using Siegel's lemma},
preprint, available at https://arxiv.org/pdf/2109.10136.pdf.


\bibitem{FSZ}
 S.~Fischler, J.~Sprang and W.~Zudilin,
{\it Many odd zeta values are irrational},
Compositio Mathematica \textbf{155} (2019), no. 5, 938--952.

\bibitem{G}
A.~I.~Galochkin,
{\it Estimate from below of polynomials in the values of analytic functions of a certain class},
Math. USSR Sbornik \textbf{24} (1974) 385--407. Original article in Mat. Sbornik \textbf{95} (137) (1974).

\bibitem{HKS}
M.~Hirose, M.~Kawashima and N.~Sato,
{\it A lower bound of the dimension of the vector space spanned by the special values of certain functions},
Tokyo J. Math. vol. \textbf{40}, 439--479.



\bibitem{Kaw}
M.~Kawashima,
{\it Rodrigues formula and linear independence for values of hypergeometric functions with parameters vary},
accepted for publication in J. of Australian Math. Soc., available at https://arxiv.org/pdf/2305.19616.pdf

\bibitem{Kratt}
C.~Krattenthaler,
{\it Advanced determinant calculus},
S\'{e}minaire Lotharingien Combin. \textbf{42} (1999) (``The
Andrews Festschrift''), Article B42q, 67 pp.

\bibitem{KP}
M.~Kawashima and A.~Po\"{e}ls,
\textit{ Pad\'{e} approximation for a class of hypergeometric functions and parametric geometry of numbers},
J. of Number Theory, \textbf{243}, (2023), 646--687.



\bibitem{Knu}
D.~Knuth,
{\it Two notes on notation},
Amer.~Math.~Monthly, \textbf{99} (1992), 403--422.

\bibitem{LaiYuZeta}
L.~Lai and P.~Yu,
{\it A note on the number of irrational odd zeta values},
Compositio Math., \textbf{156} (2020), no. 8, 1699--1717.

\bibitem{Li}
L.~Lai,
{\it On the irrationality of certain $2$-adic zeta values},
preprint, available at https://arxiv.org/pdf/2304.00816.pdf

\bibitem{Li}
L.~Lai,
{\it On the irrationality of certain $2$-adic zeta values},
arXiv: 2304.00816 [math.NT], 2023.

\bibitem{Li-Sprang}
L.~Lai and J.~Sprang,
{\it Many $p$-adic odd zeta values are irrational},
arXiv: 2306.10393 [math.NT], 2023.

\bibitem{Nielsen}
N.~Nielsen,
{\it {H}andbuch der {T}heorie der {G}ammafunktion},
(Leipzig: Teubner, 1906).


\bibitem{N-S}
E.~M.~Niki\v{s}in and V.~N.~Sorokin,
{\it Rational Approximations and Orthogonality},
American Math. Soc., Translations of Mathematical Monographs, \textbf{95} (1991).

\bibitem{P}
M.~Pr\'{e}vost,
{\it A new proof of the irrationality of $\zeta(2)$ and $\zeta(3)$ using Pad\'{e} approximants},
Journal of Computational and Applied Mathematics Volume \textbf{67}, Issue 2, 1996, 219--235.

\bibitem{Perron}
O.~Perron,
{\it $\ddot{U}$ber Summengleichungen and Poincar\'{e}sche Differenzengleichungen},
Math. Annalen., \textbf{84}, 1--15 (1921).

\bibitem{ZeilbergerAeqB}
M.~Petkovsek, W.~Herbert and D.~Zeilberger,
{\it $A=B$},
A. K. Peters (1996), with foreword by Donald E. Knuth, available at https://www2.math.upenn.edu/$\sim$wilf/AeqB.html.



\bibitem{Pituk}
M.~Pituk,
{\it More on Poincar\'{e}'s and Perron's theorems for difference equations*},
J. Differ. Equ. Appl., \textbf{8}(3), 201--216 (2002).


\bibitem{R1}
T.~Rivoal,
{\it La fonction z$\hat{e}$ta de Riemann prend une infinit\'{e} de valeurs irrationnelles aux entiers impairs},
C. R. Acad. Sci. Paris S\'{e}r. I Math. \textbf{331} (2000), no. 4, 267--270.


\bibitem{R2}
T.~Rivoal,
{\it Simultaneous polynomial approximations of the Lerch function},
Canad.\ J.\ Math.\ \textbf{61} (6), (2009), 1341--1356.

\bibitem{Rob}
A.~Robert,
{\it A Course in p-adic Analysis},
Graduate Texts in Math., \textbf{198}, Springer-Verlag (2000).

\bibitem{R-S1}
J.~B.~Rosser and L.~Schoenfeld,
{\it Approximate formulas for some functions of prime numbers},
Illinois J.\ Math.\ \textbf{6}, (1962), 64--94.

\bibitem{Shi}
W.~H.~Schikhof,
{\it Ultrametric calculus},
Cambridge Studies in Advanced Mathematics, vol. \textbf{4}, Cambridge University Press, Cambridge, 2006, An introduction to $p$-adic analysis, Reprint of the 1984 original [MR0791759].

\bibitem{S}
C.~Siegel,
{\it $\ddot{U}$ber einige Anwendungen diophantischer Approximationen},
Abhandlungen der Preu$\beta$ischen Akademie der Wissenschaften.
Physikalisch-mathematische Kalasse 1929, Nr.1.

\bibitem{SprangL}
J.~Sprang,
{\it A Linear independence result for $p$-adic $L$-values},
Duke Math. J. \textbf{169}, no. 18 (2020), 3439--3476.


\bibitem{Sti}
T.~J.~Stieltjes,
{\it Sur quelques integrales d\'{e}finies et leur d\'{e}veloppement en fractions continues},
Quarterly J. Math. London, \textbf{24}, (1890), 370--382.

\bibitem{Stirling}
J.~Stirling,
{\it Methodus Differentialis. London, 1930},
English translation, The Differential Method, 1749

\bibitem{T}
J.~Touchard,
{\it Nombres exponentiels et nombres de bernoulli},
Canad. J. Math., \textbf{8}, (1956), 305--320.




\bibitem{Vol}
A.~Volkenborn,
{\it Ein $p$-adisches Integral und seine Anwendungen. I},
Manuscripta Math. \textbf{7} (1972), 341--373.

\end{thebibliography}

\footnotesize {

}

\begin{scriptsize}
    \begin{minipage}[t]{0.38\textwidth}
        Makoto Kawashima
        \\kawasima@mi.meijigakuin.ac.jp
        \\Institute for Mathematical Informatics,
        \\Meiji Gakuin University,
        \\Totsuka, Yokohama, Kanagawa
        \\224-8539, Japan\\\\
    \end{minipage}
    \begin{minipage}[t]{0.32\textwidth}
        Anthony Po\"{e}ls
        \\ poels@math.univ-lyon1.fr
        \\Universite Claude Bernard Lyon 1,
        \\Institut Camille Jordan UMR 5208,
        \\69622 Villeurbanne, France\\\\
    \end{minipage}
\end{scriptsize}

\end{document}